\documentclass[11pt,reqno]{amsart}

\usepackage{amsmath,amssymb,mathrsfs}
\usepackage{graphicx,cite,cases}
\usepackage{multirow}
\newtheorem{theo}{Theorem}[section]

\newtheorem{lemm}[theo]{Lemma}

\newtheorem{rema}[theo]{Remark}

\numberwithin{equation}{section}

\begin{document}

\title [Electromagnetic field enhancement]{Electromagnetic field enhancement in
a subwavelength rectangular open cavity}

\author{Yixian Gao}
\address{School of Mathematics and Statistics, Center for Mathematics and
Interdisciplinary Sciences, Northeast Normal University, Changchun, Jilin
130024, P.R.China}
\email{gaoyx643@nenu.edu.cn}

\author{Peijun Li}
\address{Department of Mathematics, Purdue University, West Lafayette, IN 47907,
USA.}
\email{lipeijun@math.purdue.edu}

\author{Xiaokai Yuan}
\address{Department of Mathematics, Purdue University, West Lafayette, IN 47907,
USA.}
\email{yuan170@math.purdue.edu}

\thanks{The research of YG was supported in part by NSFC grant 11571065,
11671071, National Research program of China Grant 2013CB834100
and FRFCU2412017FZ005. The research of PL was supported in part by the NSF grant
DMS-1151308.}



\begin{abstract}

Consider the transverse magnetic polarization of the electromagnetic scattering
of a plane wave by a perfectly conducting plane surface, which contains a
two-dimensional subwavelength rectangular cavity. The enhancement is
investigated fully for the electric and magnetic fields arising in such an
interaction. The cavity wall is assumed to be a perfect electric conductor,
while the cavity bottom is allowed to be either a perfect electric conductor
or a perfect magnetic conductor. We show that the significant field
enhancement may be achieved in both nonresonant and resonant regimes. The proofs
are based on variational approaches, layer potential techniques, boundary
integral equations, and asymptotic analysis. Numerical experiments are
also presented to confirm the theoretical findings.

\end{abstract}

\maketitle

\section{Introduction}

The electromagnetic scattering by open cavities has significant applications in
many scientific areas. For instance, the radar cross section (RCS) is a
quantity which measures the detectability of a target by a radar system.
Clearly, it is of high importance in military and civil use for a deliberate
control in the form of enhancement or reduction of the RCS of a target. Since
the cavity RCS can dominate the total RCS, it is critical to have a thorough
understanding of the electromagnetic scattering by a target, particularly a
cavity, for successful implementation of any desired control of its RCS. The
underlying scattering problems have received much attention by many researchers
in both the engineering and the mathematical communities. There has been a
rapid development of the mathematical theories and computational methods in this
area \cite{Ammari2002, BaoYun2016, BaoYunZhou2012}. We refer
to \cite{li2016survey} for a survey of open cavity scattering problems.

In optics, rough optical device surfaces containing subwavelength cavities
have created a lot of interest in electromagnetic field enhancement and
extraordinary optical transmission effects in recent years, due to
their important potential applications, such as in biological and chemical
sensing, spectroscopy, Terahertz semiconductor devices
\cite{astilean2000light, barbara2003electromagnetic, chen2013atomic,
ebbesen1998extraordinary, garcia2010light, Kriegsmann2004, seo2009terahertz,
sturman2010transmission, takakura2001optical, yang2002resonant}. The
amplification of the confined fields in the cavity can be strikingly high. These
amazing features of light localization and enhancement can be quite useful in
imaging, microscopy, spectroscopy, and communication \cite{sarrazin2007bounded,
liedberg1983surface}. However, there are still controversies on the mechanisms
which contribute to anomalous field enhancement \cite{garcia2010light}. The
complication arises from the multiscale nature of the structures and various
enhancement behaviors that it induces. For instance, the enhancement can be
attributed to surface plasmonic resonance  \cite{ebbesen1998extraordinary,
garcia2010light}, non-plasmonic resonances \cite{takakura2001optical,
yang2002resonant}, or even without the resonant effect \cite{lin2015}.

It is very helpful to provide a rigorous mathematical analysis for these
phenomena, which turns out to be extremely challenging. The particular light
patterns observed in the far-field or in the near-field may be the results of
complex multi-interactions of surface waves, cavity resonances, resonant
tunneling of plasmon waves, skin depth effects, etc. It requires the resolution
of the full three-dimensional Maxwell equations in non-smooth geometries even
for a qualitative description of the diffractive properties of
electromagnetic waves in subwavelength structures. For most resonance phenomena,
the localization or enhancement of light is very sensitive to geometrical
parameters and the frequency. This motivates us to begin with simple geometries,
where the analytical approaches are made possible to obtain an accurate
description of the electromagnetic field. In \cite{Bonnetier2010,
Bonnetier20102}, the field enhancement was considered for a single rectangle
cavity and double rectangle cavities. Using asymptotic expansions of Green's
function, they showed that the limiting Green function is a perfectly conducting
plane with a dipole in place of the cavity when the width of the cavity
shrinks to zero. In \cite{Clausel2006, Patrick2006, Joly2006}, an asymptotic
expansion method was proposed to study the solution of the Helmholtz equation in
a domain including a single thin slot. A quantitative analysis was made in
\cite{lin2015, linzhang2017} for the field enhancement of the Helmholtz equation
for a single narrow open slit. We refer \cite{AmmariZhang2015,
ammari2016minnaert} for the study of closely related Helmholtz resonators and
resonances in bubbly media.

In this paper, we consider the electromagnetic scattering of a plane wave by a
perfectly electrically conducting plane surface, which contains a subwavelength
rectangular cavity. More specifically, the cavity is assumed to be invariant in
the $x_3$-axis and only the transverse magnetic polarization (TM) is studied.
Hence the three-dimensional Maxwell equations can be reduced into the
two-dimensional Helmholtz equation. The cavity wall is assumed to be a perfect
electric conductor (PEC), and the cavity bottom is assumed to be either a
perfect magnetic conductor (PMC) or a perfect electric conductor (PEC). A cavity
is called the PEC-PMC type when it has a PEC wall and a PMC bottom, while a
cavity is called the PEC-PEC type when it has a PEC wall and a PEC bottom. We
investigate the electromagnetic field enhancement for both types of cavities.
Using a combination of variational approaches, layer potential techniques,
boundary integral equations, and asymptotic analysis, we demonstrate the field
enhancement mechanisms for the underlying scattering problems different boundary
conditions. Denote by $\lambda, \epsilon, d$ the wavelength of the incident
field, the width of the cavity, and the depth of the cavity, respectively. For
the PEC-PMC cavity, we prove in the nonresonant regime that the electric field
enhancement has an order $O(\lambda/d)$ and the magnetic field has no
enhancement in the cavity provided that the length scales satisfy $\epsilon\ll
d\ll \lambda$. Moreover, if $\epsilon\ll\lambda$, we show that the
Fabry--Perot type resonance occurs. In the resonant regime, we derive an
asymptotic expansion for the resonant frequencies and deduce an order
$O(1/\epsilon)$ enhancement for both the electric and magnetic fields. For the
PEC-PEC cavity, we carry the same analysis in both the nonresonant and resonant
regimes. The electric field is shown to have a weak enhancement of order
$O(\lambda\sqrt{\epsilon/d})$ in the nonresonant regime with length
scales satisfying $\epsilon\ll d\ll\lambda$. In the resonant regime, a similar
quantity analysis is done for the resonant frequencies and the same order
$O(1/\epsilon)$ is achieved for the electric and magnetic field enhancement.
These field enhancement results are summarized in Table \ref{tab}. Numerical
experiments are also done and used to confirm our theoretical findings.

\begin{table}
\centering
\caption{Electromagnetic field enhancement. $\epsilon$: width of the
cavity, $d$: depth of the cavity, $\lambda$: wavelength of the
incident field.}
\begin{tabular}{ | c | c | c | c | c | }
\hline
\hline
\multirow{2}{*}{enhancement}
 & \multicolumn{2}{c|}{PEC-PMC cavity} & \multicolumn{2}{c|}{PEC-PEC cavity}  \\
\cline{2-5}
 & nonresonant & resonant & nonresonant  & resonant \\
\hline
electric field  & $O(\lambda/d)$ & $O(1/\epsilon)$ &
$O(\lambda\sqrt{\epsilon/d})$ & $O(1/\epsilon)$\\
\hline
magnetic field & no & $O(1/\epsilon)$ & no
& $O(1/\epsilon)$
\\
\hline
\hline
\end{tabular}
\label{tab}
\end{table}

The rest of the paper is organized as follows. The PEC-PMC cavity is examined in
sections \ref{sec2}--\ref{sec3} and the PEC-PEC cavity is studied in sections
\ref{sec4}--\ref{sec5}. Sections \ref{sec2} and \ref{sec4} focus on the field
enhancement for the non-resonant case. In each section, the model problems are
introduced, the approximate problems are proposed, the error estimates are
derived between the solutions of the exact and approximate problems, and the
electric and magnetic field enhancement are discussed. Sections \ref{sec3} and
\ref{sec5} address the field enhancement for the resonant case. In each
section, the boundary integral equations are given, the asymptotic analysis is
carried for the Fabry--Perot resonances, and the electric and magnetic field
enhancement are presented under the resonant frequencies. The paper is concluded
with some general remarks and directions for future work in section \ref{sec6}.

\section{PEC-PMC cavity without resonance}\label{sec2}

In this section, we discuss the electromagnetic field enhancement for the
PEC-PMC cavity problem when the wavenumber is sufficiently small, i.e., no
resonance occurs.

\subsection{Problem formulation }\label{PF}

\begin{figure}
\centering
\includegraphics[width=0.45\textwidth]{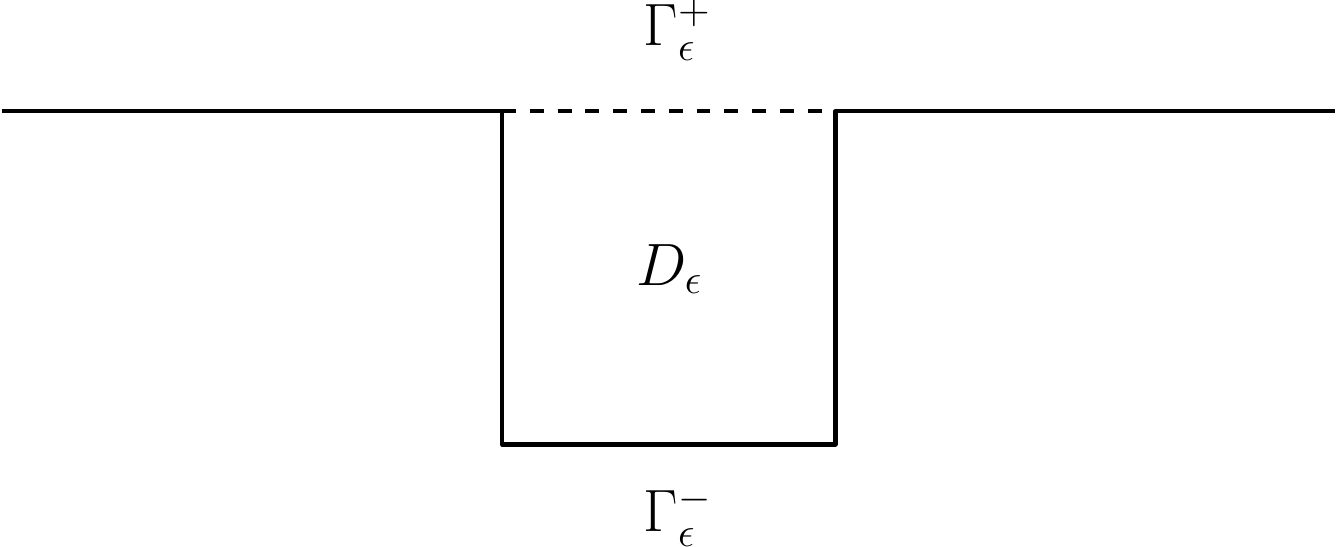}
\caption{The problem geometry of a rectangular open cavity.}
\label{fig_1}
\end{figure}

As is shown in Figure \ref{fig_1}, the problem geometry is assumed to
be invariant in the $x_3$-axis. Let $\boldsymbol x=(x_1, x_2)\in\mathbb R^2$.
Consider a two-dimensional rectangular open cavity $D_\epsilon=\{\boldsymbol
x\in\mathbb R^2: 0<x_1<\epsilon,\,-d<x_2<0\}$, where $\epsilon>0$ and $d>0$ are
constants. Denote by $\Gamma_\epsilon^+=\{\boldsymbol x\in\mathbb R^2:
0<x_1<\epsilon,\,x_2=0\}$ and $\Gamma_\epsilon^-=\{\boldsymbol x\in\mathbb R^2:
0<x_1<\epsilon,\,x_2=-d\}$ the open aperture and the bottom side of the cavity,
respectively. Let $\mathbb R^2_+=\{\boldsymbol x\in\mathbb R^2:
x_2>0\}$ and $\Gamma_0=\{\boldsymbol x\in\mathbb R^2: x_2=0\}$. Denote
$\Omega=\mathbb R^2_+\cup D_\epsilon$.

We assume that the cavity and the upper half-space are filled with the same
homogeneous medium. The electromagnetic wave propagation is governed by the
time-harmonic Maxwell equations (time dependence $e^{{- \rm i} \omega t}$):
\begin{equation}\label{me}
\nabla \times \boldsymbol E  - {\rm i } \omega \mu \boldsymbol H =0, \quad
\nabla \times \boldsymbol H + {\rm i} \omega \varepsilon \boldsymbol E =0
\quad \text {in} ~\Omega \times \mathbb R,
\end{equation}
where $\omega>0$ is the angular frequency, $\varepsilon>0$ is the electric
permittivity, $\mu>0$ is the magnetic permeability, $\boldsymbol E$ and
$\boldsymbol H$ are the electric field and the magnetic field, respectively.
Furthermore, we assume that the boundary $\partial \Omega \setminus
\Gamma_{\epsilon}^{-}$ is a perfect electrical conductor (PEC) while the
boundary $\Gamma_{\epsilon}^{-}$ is a perfect magnetic conductor (PMC). Hence
$(\boldsymbol E, \boldsymbol H )$ satisfy the following boundary conditions:
\begin{equation}\label{bc}
 \boldsymbol\nu \times  \boldsymbol E = 0 ~ \text{on}~ (\partial \Omega
\setminus \Gamma_\epsilon^-) \times \mathbb R, \quad \boldsymbol\nu \times
\boldsymbol H=0 ~ \text {on}~ \Gamma_\epsilon^- \times \mathbb R,
\end{equation}
where $\boldsymbol\nu = (\nu, 0)$ is the unit outward normal vector on
$\partial\Omega\times\mathbb R$, i.e., $\nu$ is the unit outward normal vector
on $\partial\Omega$.

We consider a polarized time-harmonic electromagnetic  wave $(\boldsymbol E^{\rm
inc}, \boldsymbol H^{\rm inc})$ that impinges on the cavity from above. In the
transverse magnetic polarization (TM), the magnetic field is parallel to the
$x_3$-axis, which implies that the incident magnetic wave $\boldsymbol H^{\rm
inc} = (0, 0, u^{\rm inc})$. Here $u^{\rm inc}(\boldsymbol x)= e^{{\rm i} \kappa
\boldsymbol x\cdot\boldsymbol d}$ is a plane wave propagating in the direction
$\boldsymbol d = (\sin\theta, -\cos \theta)$, where $ -\frac{\pi}{2} < \theta  <
\frac{\pi}{2}$ is the incident angle and $\kappa=\omega\sqrt{\varepsilon\mu}$
is the wavenumber. The corresponding incident electric field
$\boldsymbol E^{\rm inc}$ is determined by Ampere's law $\boldsymbol E^{\rm inc}
= {\rm i} (\omega \varepsilon )^{-1} \nabla \times \boldsymbol H^{\rm
inc}=\sqrt{\mu/\varepsilon}(\cos\theta, \sin\theta, 0)e^{{\rm
i}\kappa\boldsymbol x\cdot\boldsymbol d}.$
Denote by $\lambda=2\pi/\kappa$ the wavelength of the incident wave. In
this section, we assume that the length scale of the underlying geometry
satisfies $\epsilon  \ll d \ll \lambda.$

Let $\boldsymbol H = (0, 0, u_{\epsilon})$ be the total magnetic field. It can
be verified from \eqref{me}--\eqref{bc} that $u_\epsilon$ satisfies
\begin{align}\label{hlme}
\begin{cases}
\Delta u_{\epsilon }+\kappa^2 u_\epsilon = 0 \quad
&\text {in}~\Omega,\\
\partial_\nu u_{\epsilon} =0 \quad  &\text {on}~ \partial \Omega \setminus
\Gamma_\epsilon^{-} ,\\
u_{\epsilon}=0 \quad  & \text {on}~ \Gamma_\epsilon^{-}.
 \end{cases}
\end{align}
Denote by $u^{\rm r}$ the reflected field due to the interaction between the
incident field $u^{\rm inc}$ and the PEC plane $\Gamma_0$. It can be
shown that $u^{\rm ref}(\boldsymbol x) = e ^{{\rm i} \kappa \boldsymbol
x\cdot\boldsymbol d'}$, where $\boldsymbol d' =
(\sin\theta, \cos \theta).$ In $\mathbb R_2^{+}$, the total field $u_{\epsilon
}$ consists of the incident wave $u^{\rm inc},$ the reflected wave $u^{\rm
ref}$, and the scattered field $u_{\epsilon }^{\rm sc}$. The scattered fields
$u_{\epsilon}^{\rm sc} $ is required to satisfy the Sommerfeld radiation
condition:
\begin{align}\label{rdc}
\lim\limits_{r \rightarrow \infty} \sqrt r \left( \partial_r u_{\epsilon }^{\rm
sc} - {\rm i} \kappa u_{\epsilon }^{\rm sc} \right)=0, \quad r = |\boldsymbol
x|.
\end{align}

\subsection{An approximate model}

We introduce an approximate model for the problem \eqref{hlme}--\eqref{rdc}
in order to estimate  the field enhancement inside the cavity.

Let
\[
\phi_0 (x_1) = \frac{1}{ \sqrt {\epsilon}}, \quad \phi_n (x_1) = \sqrt
{\frac{2}{ \epsilon }} \cos \left( \frac{n \pi x_1}{\epsilon } \right),~ n \geq
1
\]
be an orthonormal basis on the interval $(0, {\epsilon }).$ It follows from
the Neumann boundary condition in \eqref{hlme} that $u_{\epsilon} $ can be
expanded as the series of waveguide modes:
\begin{align}\label{wg}
u_{\epsilon } (\boldsymbol x) = \sum\limits_{n =0}^{\infty} \left(  \alpha_n^{+}
e^{ -{\rm i} \beta_n x_2} + \alpha_n^{-} e^{{\rm i} \beta_n (x_2 + d)}\right)
\phi_n (x_1), \quad \boldsymbol x\in D_\epsilon,
\end{align}
where the coefficients $\beta_n$ are defined as
\begin{align}\label{Kpp}
 \beta_n =
 \begin{cases}
  \kappa, \quad & n=0,\\
  {\rm i} \sqrt{({n \pi}/{\epsilon})^2 -\kappa^2}, \quad & n \geq 1.
 \end{cases}
\end{align}
If ${\epsilon }$ is small enough (in fact  it only needs  $\epsilon  <
\lambda/2 $ here), it follows that $\beta_n$ is are pure
imaginary numbers for all $n \geq 1$. It is easy to note that for each $n$, if
$ n \pi /{\epsilon }  \leq  \kappa$,  the series consists of two propagating
wave modes traveling upward and downward respectively; if $ n \pi / \epsilon
> \kappa$ the series consists of two evanescent wave modes decaying
exponentially away from the bottom side and the open aperture of the cavity,
respectively.

Using the series \eqref{wg}, we may reformulate \eqref{hlme}--\eqref{rdc} into
the following coupled problem:
\begin{align}\label{eqe}
\begin{cases}
\Delta u_{\epsilon}+ \kappa^2 u_{\epsilon }=0 \quad  &\text
{in}~\mathbb R_2^+,\\
u_{\epsilon} = \sum\limits_{n =0}^{\infty} \left(  \alpha_n^{+}
e^{ -{\rm i} \beta_n x_2} + \alpha_n^{-} e^{{\rm i} \beta_n (x_2 + d)}\right)
\phi_n (x_1) \quad & \text {in}~D_{\epsilon},\\
\partial_\nu u_{\epsilon}=0 \quad & \text {on}~ \partial \Omega \setminus
\Gamma^-_{\epsilon} ,\\
u_{\epsilon } =0 \quad & \text {on}~\Gamma_{\epsilon}^{-},\\
u_{\epsilon } (x_1, 0+) =u_{\epsilon} (x_1, 0-),~\partial_{x_2}  u_{\epsilon}
(x_1, 0+)=\partial_{x_2}  u_{\epsilon} (x_1, 0-) \quad &\text{on}~
\Gamma_\epsilon^+,\\
\lim\limits_{r \rightarrow \infty} \sqrt r \left( \partial_r u_{\epsilon}^{\rm
s} - {\rm i} \kappa u_{\epsilon}^{\rm s} \right)=0 \quad  &\text {in}~\mathbb
R_2^+,
\end{cases}
\end{align}
where $u_\epsilon^{\rm sc} = u_{\epsilon} -(u^{\rm inc} +u^{\rm ref})$ in
$\mathbb R_2^+$. The continuity conditions are imposed along
the open aperture $\Gamma_{\epsilon}^+$, where $ 0+$ and $0-$ indicate the
limits taken from above and below $\Gamma_{\epsilon}^+$, respectively.

For simplicity of notation, we define $u_{\epsilon}^+ (x_1)= u_{\epsilon} (x_1,
0)$ and $ u_{\epsilon}^{-} (x_1)=u_{\epsilon} (x_1, -d)$. Let $u_{{\epsilon},
n}^{\pm}=\langle u_{\epsilon}^{\pm}, \phi_n \rangle_{\Gamma_{\epsilon}^{\pm}}$
be the Fourier coefficients for $u_{\epsilon}^+$ and $u_{\epsilon}^{-}$,
respectively, where the inner product $\langle \cdot, \cdot
\rangle_{\Gamma_{\epsilon}^{\pm}}$ is defined by
\begin{align*}
\langle  u_{\epsilon}^{\pm},
\phi_n\rangle_{\Gamma_{\epsilon}^{\pm}}:=\int_0^{\epsilon} u_{\epsilon}^{\pm}
(x_1) \bar \phi_n (x_1) {\rm d} x_1.
\end{align*}
For any $s \in \mathbb R$, denote by $H^s (\Gamma_{\epsilon}^+)$ the
trace functional space on $\Gamma_{\epsilon}^+$ with the norm given by
\begin{align}\label{snm}
\|u_{\epsilon}^+\|^2_{H^s (\Gamma_{\epsilon}^+)}=\sum\limits_{n=0}^{\infty}
\left(1+ \left( \frac{n \pi}{\epsilon} \right)^2 \right)^s |u_{{\epsilon},
n}^+|^2.
\end{align}

Matching the series \eqref{wg} on $\Gamma_{\epsilon}^+, \Gamma_{\epsilon}^-$
and using the continuity conditions in \eqref{eqe}, we get
\begin{align*}
\begin{cases}
 \alpha_n^+ +\alpha_n^{-} e^{{\rm i} \beta_n d}= u_{{\epsilon}, n}^+,\\
 \alpha_n^+ e^{{\rm i} \beta_n d} + \alpha_n^{-}= u_{ {\epsilon}, n}^{-}.
 \end{cases}
\end{align*}
Solving the above equations by Cramer's rule yields that
\begin{align}\label{ec}
\alpha_n^+= \frac{e^{{\rm i} \beta_n d} u_{{\epsilon}, n}^{-} -u_{{\epsilon},
n}^+}{ e^{{\rm i} 2 \beta_n d}-1},\quad
\alpha_n^-= \frac{e^{{\rm i} \beta_n d} u_{{\epsilon}, n}^{+} -u_{{\epsilon},
n}^-}{ e^{{\rm i} 2 \beta_n d}-1}.
\end{align}
Therefore, in the cavity $D_\epsilon$, the total field
\begin{align}\label{wf}
 u_{\epsilon} (\boldsymbol x) =\sum\limits_{n =0}^{\infty} \left(  \frac{e^{{\rm
i} \beta_n d} u_{{\epsilon}, n}^{-} -u_{{\epsilon}, n}^+}{ e^{{\rm i} 2 \beta_n
d}-1} e^{ -{\rm i} \beta_n x_2} + \frac{e^{{\rm i} \beta_n d} u_{{\epsilon},
n}^{+} -u_{{\epsilon}, n}^-}{ e^{{\rm i} 2 \beta_n d}-1} e^{{\rm i} \beta_n (x_2
+ d)}\right) \phi_n (x_1).
\end{align}
Since $u_\epsilon$ satisfies the homogeneous Dirichlet boundary condition on
$\Gamma_{\epsilon}^{-}$, we have
\begin{align}
\partial_{x_2} u_{\epsilon} (x_1, 0-)&
=\sum\limits_{n=0}^{\infty} {\rm i }
\beta_n \left( \frac{e^{{\rm i} \beta_n d} u_{{\epsilon}, n}^{+} -u_{{\epsilon},
n}^-}{ e^{{\rm i} 2 \beta_n d}-1} e^{{\rm i} \beta_n  d}-\frac{e^{{\rm i}
\beta_n d} u_{{\epsilon}, n}^{-} -u_{{\epsilon}, n}^+}{ e^{{\rm i} 2 \beta_n
d}-1}\right)\phi_n (x_1) \quad \text {on}~\Gamma_{\epsilon}^+, \label{td}\\
 u_{\epsilon } (x_1, -d)& =\sum\limits_{n=0}^{\infty}
 \left(\frac{e^{{\rm i} \beta_n d} u_{{\epsilon}, n}^{-} -u_{{\epsilon}, n}^+}{
e^{{\rm i} 2 \beta_n d}-1} e^{ {\rm i} \beta_n  d}+ \frac{e^{{\rm i} \beta_n d}
u_{{\epsilon}, n}^{+} -u_{{\epsilon}, n}^-}{ e^{{\rm i} 2 \beta_n d}-1}
 \right)\phi_n (x_1)=0 \quad \text {on}~\Gamma_{\epsilon}^{-}. \label{ttd}
\end{align}
It follows from \eqref{ttd} that
\begin{align}\label{uf}
 u_{{\epsilon}, n}^{-} = 0.
\end{align}
Substituting \eqref{uf} into \eqref{td} yields
\begin{align}\label{tnd}
 \partial_{x_2} u_{\epsilon} (x_1, 0-) =  \sum\limits_{n=0}^{\infty} {\rm i }
\beta_n  \frac{e^{{\rm i} 2 \beta_n d} +1}{e^{{\rm i} 2 \beta_n d} -1}
u_{{\epsilon}, n}^+ \phi_n (x_1)
 \quad \text {on}  ~\Gamma_{\epsilon}^+.
\end{align}

Given a function $u\in H^1 (\mathbb R_2^+)$ and let $u ^{+}= u  (x_1, 0)$,
we define a Dirichlet-to-Neumann (DtN) operator:
\begin{align}\label{dtn}
 \mathscr B^{\rm PMC} [u] = \sum\limits_{n=0}^{\infty} {\rm i } \beta_n
\frac{e^{{\rm i} 2 \beta_n d} +1}{e^{{\rm i} 2 \beta_n d} -1} u_{ n}^+ \phi_n
(x_1)\quad\text{on}~\Gamma_\epsilon^+,
\end{align}
where the Fourier coefficients $u_{n}^+ = \langle u^{+}, \phi_n
\rangle_{\Gamma_{\epsilon}^+}.$

\begin{lemm}\label{tth}
The DtN operator $\mathscr B^{\rm PMC}$ is bounded from $H^{1/2}
(\Gamma_{\epsilon}^+) \rightarrow  H^{-1/2} (\Gamma_{\epsilon}^+)$, i.e.,
\begin{align*}
\|\mathscr B^{\rm PMC} [u] \|_{H^{-1/2} (\Gamma_{\epsilon}^+)} \leq C (\lambda,
d) \|u\|_{H^{1/2} (\Gamma_{\epsilon}^+)}\quad \forall u \in
{H^{1/2} (\Gamma_{\epsilon}^+)},
\end{align*}
where the constant $C (\lambda, d) =\max \left\{4,  \frac{\kappa} {|\sin \kappa
d |} \right\}$.
\end{lemm}

\begin{proof}
 For any $u, w \in H^{1/2} (\Gamma_{\epsilon}^+)$, they have the
Fourier series expansions
\begin{align*}
u =\sum\limits_{n=0}^{\infty} u_{n}^+ \phi_{n} (x_1), \quad  w
=\sum\limits_{n=0}^{\infty} w_{n}^+ \phi_{n} (x_1),
\end{align*}
where $u_{n}^+ = \langle u, \phi_{n} \rangle_{\Gamma_{\epsilon}^+}$ and
$w_{n}^+ = \langle w, \phi_{n} \rangle_{\Gamma_{\epsilon}^+}$. It follows from
\eqref{dtn} that
\begin{align*}
\langle \mathscr B^{\rm PMC} [u], w \rangle_{\Gamma_{\epsilon}^+}
=\sum\limits_{n=0}^{\infty} {\rm i} \beta_n \frac{e^{{\rm i} 2 \beta_n d}
+1}{e^{{\rm i} 2 \beta_n d} -1} u_n^+ \bar {w}_n^+.
\end{align*}
Note that the coefficients $\beta_n$ are given by \eqref{Kpp}. For $n=0$, a
straightforward calculation yields that
\begin{align}\label{eq1}
\left|{\rm i} \kappa \frac{e^{{\rm i}2 \kappa  d} +1}{e^{{\rm i}2 \kappa d} -1}
u_0^+{ \bar w}_{0}^+ \right | = \frac{\kappa |\cos \kappa d| }{ |\sin  \kappa
d|} |u_0^+||{ \bar w}_{0}^+| \leq  \frac{ \kappa }  {  |\sin \kappa d|}
|u_0^+||{ \bar w}_{0}^+|,
\end{align}
where we use the fact that $d \ll \lambda$, which implies $\kappa d  =
\frac{2 \pi d }{\lambda} $ tends to zero and $ \cos( \kappa d )\geq
1/2$. For $n \geq 1$, using the fact that $\epsilon \ll \lambda$, we get
\begin{align}\label{eq2}
 \left|{\rm i} \beta_n
 \frac{e^{{\rm i} 2 \beta_n d} +1}{e^{{\rm i} 2 \beta_n d} -1} u_n^+ \bar {w}_n^+ \right|
 \leq  4 \frac{n \pi}{\epsilon} |u_n^+| | \bar {w}_n^+|
 \leq 4  \left( 1 + \left(\frac{n \pi}{\epsilon} \right)^2 \right)^{1/2} |u_n^+|
 \left( 1 + \left(\frac{n \pi}{\epsilon} \right)^2 \right)^{1/2} | \bar w_n^+|,
\end{align}
where we use the estimates
\begin{align*}
|{\rm i} \beta_n|= \left|- \sqrt{\left( \frac{ n\pi}{\epsilon} \right)^2
-\kappa^2} \right| \leq \frac{n \pi}{\epsilon}, \quad \left|\frac{e^{{\rm i} 2
\beta_n d} +1}{e^{{\rm i} 2 \beta_n d} -1}  \right| = \left| \frac{e^{- 2
\sqrt{(n \pi/{\epsilon})^2 -\kappa^2} d} +1}{e^{- 2 \sqrt{(n \pi/{\epsilon})^2
-\kappa^2} d} -1} \right| \leq 4.
\end{align*}
Combining \eqref{eq1}--\eqref{eq2} and using \eqref{snm}, we obtain from the
Cauchy--Schwarz inequality that
\begin{align*}
 \left|\langle \mathscr B^{\rm PMC} [u], w \rangle_{\Gamma_{\epsilon}^+} \right|
\leq C (\lambda, d)\|u\|_{H^{1/2} (\Gamma_{\epsilon}^+)}\|w\|_{H^{1/2}
(\Gamma_{\epsilon}^+)},
\end{align*}
which completes the proof.
\end{proof}

Using \eqref{tnd}--\eqref{dtn} and the continuity of $\partial_{x_2}
u_{\epsilon}$ on $\Gamma^+_\epsilon$, we obtain the transparent
boundary condition (TBC):
\begin{align*}
 \partial_{x_2} u_{\epsilon} =\mathscr B^{\rm PMC}  [u_{\epsilon}] \quad \text
{on}~\Gamma_{\epsilon}^{+}.
\end{align*}
Hence the problem \eqref{hlme}--\eqref{rdc} can be reduced to the following
boundary value problem:
\begin{align}\label{rdp}
\begin{cases}
\Delta u_{\epsilon} + \kappa^2 u_{\epsilon}=0 \quad  & \text {in}~\mathbb
R^2_+,\\
\partial_\nu u_{\epsilon} =0 \quad  &\text {on}~  \Gamma_0\setminus
\Gamma_{\epsilon}^+,\\
\partial_{x_2} u_{\epsilon} = \mathscr B^{\rm PMC} [u_{\epsilon}] \quad & \text
{on}~\Gamma_{\epsilon}^+,\\
\lim\limits_{r \rightarrow \infty} \sqrt r \left( \partial_r u_{\epsilon}^{\rm
sc} - {\rm i} \kappa u_{\epsilon}^{\rm sc} \right)=0 \quad  &\text {in}~\mathbb
R^2_+.
\end{cases}
\end{align}
By calculating the Fourier coefficients $u_{{\epsilon}, n}^+$, which
give the coefficients $u_{{\epsilon}, n}^-$ by \eqref{uf}, we may obtain the
solution in the cavity $D_{\epsilon}$ from the formula \eqref{wf}.

To find an approximate model to \eqref{rdp}, we examine the series \eqref{wg}
more closely. Note that if $\epsilon\ll \lambda$, the wave modes $e^{-{\rm i}
\beta_n x_2} \phi_n (x_1)$ and $e ^{{\rm i} \beta_n (x_2 + d)} \phi_n(x_1)$
decay exponentially in the cavity, with a decaying rate of $O
(e^{-n/{\epsilon}})$ for all $n \geq 1$. Only the leading wave modes $e^{-{\rm
i} \kappa x_2} \phi_0 (x_1)$ and $e^{{\rm i} \kappa (x_2 +d)} \phi_0 (x_1)$
propagate in the cavity. The observation motivates us to approximate the DtN map
\eqref{dtn} by dropping the high order modes and define an approximate DtN
map by
\begin{align}\label{adtn}
\mathscr B_0^{\rm PMC} [v]= {\rm i }  \kappa  \frac{e^{{\rm i} 2 \kappa d}
+1}{e^{{\rm i} 2  \kappa d}-1} v_0^+ \phi_0 (x_1).
\end{align}

Now we arrive at an approximate model problem:
\begin{align}\label{ardp}
\begin{cases}
\Delta v_{\epsilon} + \kappa^2 v_{\epsilon}=0 \quad & \text {in}~\mathbb
R^2_+,\\
\partial_\nu v_{\epsilon} =0  \quad  &\text {on}~   \Gamma_0\setminus
\Gamma_{\epsilon}^+,\\
\partial_{x_2} v_{\epsilon}= \mathscr B_0^{\rm PMC} [v_{\epsilon}]\quad & \text
{on}~\Gamma_{\epsilon}^+,\\
\lim\limits_{r \rightarrow \infty} \sqrt r \left( \partial_r v_{\epsilon}^{\rm
sc} - {\rm i} \kappa  v_{\epsilon}^{\rm sc} \right)=0 \quad  &\text {in}~\mathbb
R^2_+,
  \end{cases}
\end{align}
where $ v_{\epsilon}= 0 $ on $\Gamma_{\epsilon}^{-}$ and $v_{\epsilon}^{\rm sc}
= v_{\epsilon} -(u^{\rm inc} +u^{\rm ref})$ in $\mathbb R^2_+$. Accordingly, we
may approximate $u_{\epsilon}$ by one single mode:
\begin{align}\label{om}
 v_{\epsilon} (\boldsymbol x)= \left(  \tilde  \alpha_0^+ e^{-{\rm i} \kappa
x_2} + \tilde  \alpha_0^{-} e^{{\rm i} \kappa (x_2 + d)}\right)\phi_0
(x_1),\quad \boldsymbol x\in D_\epsilon,
\end{align}
where the coefficients $\tilde \alpha_0^+$ and  $\tilde \alpha_0^-$  are given by
\begin{align}\label{za}
\tilde \alpha_0^+= \frac{e^{{\rm i} \kappa d} v_{{\epsilon}, 0}^{-} -
v_{{\epsilon}, 0}^{+}}{ e^{{\rm i}  2 \kappa d} -1}, \quad \tilde \alpha_0^-=
\frac{e^{{\rm i} \kappa d} v_{{\epsilon}, 0}^{+} - v_{{\epsilon}, 0}^{-}}{
e^{{\rm i}  2 \kappa d} -1}.
\end{align}
It follows from  $ v_{\epsilon} =0 $ on $\Gamma_{\epsilon}^{-} $ that  $\tilde
{\alpha}_0^{-} = -\tilde {\alpha}_0^{+} e^{{\rm i} \kappa d}$, which yields
\begin{align}\label{om1}
 v_{{\epsilon}, 0}^{-} = 0.
\end{align}

\subsection{Enhancement of the approximated field}

This section introduces the estimates for the solution of the approximate model
problem.

\begin{theo}\label{get}
Let $v_{\epsilon}$ be the solution of the approximate model problem \eqref{ardp}
and be given by \eqref{om} in the cavity, then there exist positive constants
$C_1, C_2$ independent of $\lambda, \epsilon, d$ such that
\begin{align*}
C_1 \sqrt{\epsilon/d} \leq  \|\nabla v_{\epsilon}\|_{L^2
(D_{\epsilon})} \leq C_2 \sqrt{\epsilon/d}.
\end{align*}
\end{theo}

\begin{proof}
In $\mathbb R^2_+$, since the incident field $u^{\rm inc}$ and the reflected
field $u^{\rm ref}$ satisfy the Helmholtz equation in  \eqref{ardp}, the
scattered field $v_{\epsilon}^{\rm sc}$ also satisfies
\begin{align*}
\Delta v^{\rm sc}_{\epsilon} + \kappa^2 v_{\epsilon}^{\rm sc}=0 \quad  \text
{in}~\mathbb R^2_+.
\end{align*}
Noting that $\partial_{x_2} u^{\rm inc} + \partial_{x_2} u^{\rm ref} =0$ on
$\Gamma_0$, we have
\begin{align*}
 \partial_{x_2} v_{\epsilon}^{\rm s} =0 \quad  \text {on}~\Gamma_0\setminus
\Gamma_{\epsilon}^+.
\end{align*}

Let $G$ be the half-space Green function of the Helmholtz equation with Neumann
boundary condition, i.e., it
satisfies
\begin{align*}
\begin{cases}
\Delta G (\boldsymbol x, \boldsymbol y)+ \kappa^2 G (\boldsymbol x, \boldsymbol
y ) = \delta (\boldsymbol x, \boldsymbol y), \quad & \boldsymbol x, \boldsymbol
y \in \mathbb R_2^+,\\
\frac{\partial G (\boldsymbol x, \boldsymbol y)}{ \partial
\nu_{\boldsymbol y}} =0, \quad &\text {on}~\Gamma_0.
\end{cases}
\end{align*}
It is easy to note that
\begin{equation}\label{gfhs}
G (\boldsymbol x, \boldsymbol y)= - \frac{\rm i}{4} \left( H_0^{(1)} (\kappa
|\boldsymbol x -\boldsymbol y|) +H_0^{(1)} (\kappa |\boldsymbol x' -\boldsymbol
y|)\right),
\end{equation}
where $\boldsymbol y=(y_1, y_2)$, $H_{0}^{(1)}$ is the Hankel function of the
first kind with order $0$, and $ \boldsymbol x'$ is  the reflection of the point
$\boldsymbol x$ with respect to $\Gamma_0$, i.e., $\boldsymbol
x'=(x_1, -x_2)$.

It follows from Green's identity that
\begin{align*}
v_{\epsilon}^{\rm sc} (\boldsymbol x) =\int_{\Gamma_{\epsilon}^+} G
(\boldsymbol x, \boldsymbol y)\frac{\partial v_\epsilon^{\rm
sc}(\boldsymbol y)}{\partial\nu_{\boldsymbol y}}{\rm d} s_{\boldsymbol y},
\quad \boldsymbol x \in \mathbb R^2_+,
\end{align*}
which gives
\begin{align*}
 v_{\epsilon} (\boldsymbol x) = u^{\rm inc}(\boldsymbol x)+ u^{\rm
r}(\boldsymbol x) +\int_{\Gamma_{\epsilon}^+} G (\boldsymbol x, \boldsymbol y)
\frac{\partial v_\epsilon^{\rm sc}(\boldsymbol y)}{\partial\nu_{\boldsymbol y}}
{\rm d} s_{\boldsymbol y}, \quad \boldsymbol x \in \mathbb R^2_+.
\end{align*}
By the fact that $ \partial_{x_2} u^{\rm inc} + \partial_{x_2} u^{\rm ref} =0$
on $\Gamma_0$, especially on $ \Gamma_{\epsilon}^+$, we have
\begin{align*}
\frac{\partial v_{\epsilon}^{\rm sc}}{ \partial \nu_{\boldsymbol y}} =
\frac{\partial v_{\epsilon}}{ \partial \nu_{\boldsymbol y}}
- \frac{\partial (u^{\rm inc} + u^{\rm ref})}{ \partial \nu_{\boldsymbol y}}
= \frac{\partial v_{\epsilon}}{ \partial \nu_{\boldsymbol y}} \quad \text
{on}~\Gamma_{\epsilon}^+.
\end{align*}
It follows from the continuity of single layer potentials
\cite{Colton1983, Colton1998}  and the above equality that
\begin{align}\label{vs}
v_{\epsilon} (\boldsymbol x)= u^{\rm inc}(\boldsymbol x)+ u^{\rm
ref}(\boldsymbol x)- \frac{\rm i}{2} \int_{\Gamma_{\epsilon}^+} H_0^{(1)}
(\kappa |\boldsymbol x -\boldsymbol y|) \frac{\partial v_{\epsilon}(\boldsymbol
y)}{ \partial \nu_{\boldsymbol y}} {\rm d} s_{\boldsymbol y}, \quad \boldsymbol
x \in \Gamma_{\epsilon}^+.
 \end{align}
Using \eqref{om}--\eqref{om1} yields that
\begin{align*}
\partial_{x_2} v_{\epsilon}= {\rm i} \kappa \left(-\tilde {\alpha}_0^+ + \tilde
{\alpha}_0^{-} e^{{\rm i} \kappa d} \right) \phi_0 (x_1) \quad \text
{on}~\Gamma_{\epsilon}^+.
\end{align*}
Substituting the above equality into \eqref{vs} and using the fact that $\phi_0
(x_1) = \frac{1}{ \sqrt {\epsilon}}$, we get
\begin{align*}
v_{\epsilon} (x_1, 0)=u^{\rm inc} (x_1, 0) +u^{\rm ref} (x_1, 0)
+\frac{\kappa}{2} \left( -\tilde {\alpha}_0^+ + \tilde {\alpha}_0^{-} e^{{\rm i}
\kappa d }\right) \frac{1}{\sqrt {\epsilon}} h_1 (x_1),  \quad x_1 \in (0,
\epsilon),
\end{align*}
where \begin{align}\label{H1}
h_1 (x_1)=\int_0^{\epsilon} H_{0}^{(1)} (\kappa |x_1 -y_1|) {\rm d} y_1.
\end{align}

Therefore, the Fourier coefficients $v_{\epsilon, 0}^+$ may be expressed as
\begin{align*}
v_{\epsilon, 0}^+ = \langle u^{\rm inc}, \phi_0
\rangle_{\Gamma_{\epsilon}^+}+\langle u^{\rm ref}, \phi_0
\rangle_{\Gamma_{\epsilon}^+} +\frac{\kappa}{2} \left( -\tilde {\alpha}_0^+
+\tilde {\alpha}_0^{-} e^{{\rm i} \kappa d }\right) \frac{1}{\sqrt {\epsilon}}
\langle h_1, \phi_0 \rangle_{\Gamma_{\epsilon}^+}.
\end{align*}
It follows from the fact  $u^{\rm inc} (x_1, 0)= u^{\rm ref} (x_1, 0)$ and
\eqref{om} that
\begin{align*}
\tilde \alpha_{0}^+ +\tilde \alpha_{0}^{-} e^{{\rm i} \kappa d} =2 \langle
u^{\rm inc}, \phi_0  \rangle_{\Gamma_{\epsilon}^+}
+ \frac{\kappa}{2  \sqrt {\epsilon}}  \langle h_1, \phi_0
\rangle_{\Gamma_{\epsilon}^+}  \left( -\tilde {\alpha}_0^+
+ \tilde {\alpha}_0^{-} e^{{\rm i} \kappa d }\right).
\end{align*}
By $ \tilde {\alpha}_0^{-} =- \tilde {\alpha}_0^{+} e^{{\rm i} \kappa d}$ and
the above equation, we obtain
\begin{align}\label{alp}
 \tilde \alpha_0^+ = \frac{2  \langle u^{\rm inc}, \phi_0
\rangle_{\Gamma_{\epsilon}^+} }{  (1+ c_0) -  (1- c_0) e^{{\rm i} 2 \kappa
d}}, \quad \tilde \alpha_0^{-}=- \frac{2  e^{{\rm i } \kappa d} \langle u^{\rm
inc}, \phi_0  \rangle_{\Gamma_{\epsilon}^+} }{  (1+ c_0) -  (1- c_0)
e^{{\rm i} 2 \kappa d}},
\end{align}
where $c_0 =  \frac{\kappa}{2  \sqrt {\epsilon}}  \langle h_1, \phi_0
\rangle_{\Gamma_{\epsilon}^+}.$

Note that $x_1, y_1 \in \Gamma_{\epsilon}^+$ for small $|x_1 -y_1|$,
asymptotically, it holds that  (cf. \cite{Watson1995})
\begin{align*}
H_0^{(1)} (\kappa  |x_1 -y_1|) = \frac{2  {\rm i}}{\pi} \ln |x_1 -y_1| + \frac{2
{\rm i}}{\pi} \ln \frac{\kappa}{2} +\gamma_0+O (|x_1 -y_1|^2 \ln |x_1 -y_1|),
\end{align*}
where $\gamma_0$ is the Euler constant defined by
$\gamma_0=\lim\limits_{\tau\rightarrow \infty} \sum\limits_{m=1}^{\tau}
\frac{1}{m} -\ln \tau.$
A direct calculation yields
\begin{align}\label{C0}
c_0 =  \frac{\kappa}{2  \sqrt {\epsilon}}  \langle h_1, \phi_0
\rangle_{\Gamma_{\epsilon}^+}= \frac{\kappa}{ 2 {\epsilon} }
\int_0^{\epsilon}\int_0^ {\epsilon} H_0^{(1)} (\kappa |x_1 -y_1|) {\rm d} y_1
{\rm d} x_1 =\frac{{\rm i} \kappa}{\pi} {\epsilon} \ln {\epsilon} + O
({\epsilon}).
\end{align}
It follows from the assumption $\epsilon\ll d \ll \lambda$ that
\begin{align}
\label{inq1} &|(1+c_0) - (1-c_0) e^{{\rm i} 2 \kappa d}| \leq 2 |e^{{\rm i} 2
\kappa d}-1|=4 \kappa d,\\
\label{inq2} &|(1+c_0) - (1-c_0) e^{{\rm i} 2 \kappa d}|
\geq \frac{1}{2}|e^{{\rm i} 2 \kappa d} -1|\geq  \kappa d.
\end{align}
Using the definition of the incident wave gives
\begin{align}\label{inq3}
|\langle u^{\rm inc}, \phi_0
\rangle_{\Gamma_{\epsilon}^+}|=\left|\int_0^{\epsilon} e^{{\rm i} \kappa
\sin\theta x_1} \frac{1}{ \sqrt {\epsilon}} {\rm d} x_1 \right| =\sqrt
{\epsilon} + O({\epsilon}^{\frac{3}{2}}).
\end{align}
Combining \eqref{alp}--\eqref{inq3}, we obtain
\begin{align*}
|\tilde \alpha_0^{+}| \leq C  \frac{\sqrt {\epsilon}} {\kappa d}, \quad  |\tilde
\alpha_0^{-}| \leq C  \frac{\sqrt {\epsilon }} {\kappa d},
\end{align*}
where the  positive constant $C$ is independent of $\lambda, \epsilon, d$.

It follows from \eqref{om} that
\begin{align*}
\|\nabla v_{\epsilon}\|^2_{L^2  (D_{\epsilon})} \leq \int_0^
{\epsilon}\int_{-d}^{0} \left( \kappa  (|\tilde \alpha_0^+| +|\tilde
\alpha_0^{-}|) \phi_0 (x_1) \right)^2 {\rm d} x_2 {\rm d} x_1 \leq C_2^2
(\epsilon/d).
\end{align*}
Substituting \eqref{alp} into \eqref{om} yields
\begin{align}\label{nep}
 v_{\epsilon} (\boldsymbol x)= \left(
 \frac{2  \langle u^{\rm inc}, \phi_0  \rangle_{\Gamma_{\epsilon}^+} }{
 (1+ c_0) -  (1- c_0) e^{{\rm i} 2 \kappa d}} e^{-{\rm i} \kappa x_2} - \frac{2
e^{{\rm i } \kappa d} \langle u^{\rm inc}, \phi_0
\rangle_{\Gamma_{\epsilon}^+} }{  (1+ c_0) -  (1- c_0) e^{{\rm i} 2 \kappa d}}
e^{{\rm i} \kappa  (x_2 + d)}\right)\phi_0 (x_1).
\end{align}
A simple calculation yields
\begin{align}\label{lb}
 \frac{\partial v_{\epsilon}}{\partial {x_2}}
 &= -{\rm i} \kappa e^{{\rm i}  \kappa d } \frac{2  \langle u^{\rm inc},
\phi_0  \rangle_{\Gamma_{\epsilon}^+} }{  (1+ c_0) - (1- c_0) e^{{\rm i} 2
\kappa d}} \phi_0 (x_1) \left( e^{ -{\rm i} \kappa  (x_2 +d)} + e^{ {\rm i}
\kappa (x_2 +d)}\right) \nonumber\\
 &= - {\rm i}\kappa e^{{\rm i}  \kappa d } \frac{2  \langle u^{\rm inc},
\phi_0  \rangle_{\Gamma_{\epsilon}^+} }{  (1+ c_0) -  (1- c_0)e^{{\rm i} 2
\kappa  d}} 2 \cos \kappa (x_2 +d)\phi_0 (x_1).
\end{align}
Noting that $\epsilon$ is small enough and $\kappa (x_2 + d) \ll 1$, we have
from \eqref{inq1}--\eqref{inq3} and \eqref{lb} that there exists a positive
constant $C_1$ such that
\begin{align*}
 \left|\frac{\partial v_{\epsilon}}{\partial {x_2}} \right| \geq  C_1 \frac
{1}{d},
\end{align*}
which yields that
\begin{align*}
\|\nabla v_{\epsilon}\|^2_{L^2 (D_{\epsilon})} \geq  \int_{0}^{\epsilon}
\int_{-d}^{0} \left|\frac{\partial v_{\epsilon}}{\partial x_2} \right|^2 {\rm
d} x_{2} {\rm d} x_1 \geq C_{1}^2 (\epsilon/d),
\end{align*}
which completes the proof.
\end{proof}

\begin{theo}\label{ght}
Let $v_{\epsilon}$ be the solution of the approximate model problem \eqref{ardp}
and be given by \eqref{om}, then there exist two positive constant
$C_3$ and $C_4$ independent of ${\epsilon}, d, \lambda$ such that
\begin{align*}
C_3 \sqrt{{\epsilon} d} \leq  \|v_{\epsilon}\|_{L^2 (D_{\epsilon})} \leq C_4
\sqrt{{\epsilon} d}.
\end{align*}
\end{theo}

\begin{proof}
It follows from \eqref{nep} that
\begin{align*}
\int_{0}^{\epsilon}\int_{-d}^{0} &|v_{\epsilon} (x_1, x_2)|^2 {\rm d} x_1 {\rm
d} x_2= \frac {\left| 2 \langle u^{\rm inc},
\phi_{0}\rangle_{\Gamma_{\epsilon}^+}  e^{{\rm i} \kappa d}\right|^2} {\left|
(1+c_0)- (1 -c_0)^{{\rm i} 2 \kappa d}\right|^2} \int_{-d}^{0}| e^{-{\rm i}
\kappa (x_2 +d)}-e^{{\rm i} \kappa (x_2 +d)}|^2 {\rm d} x_2 \\
 & =\frac {\left| 2 \langle u^{\rm inc}, \phi_{0}\rangle_{\Gamma_{\epsilon}^+}
\right|^2} {\left| (1+c_0)- (1 -c_0)^{{\rm i} 2 \kappa d}\right|^2}
\int_{-d}^{0} 4 \sin^2 \kappa (x_2 +d)
{\rm d} x_2.
\end{align*}
Noting that $d \ll \lambda$, we have
\begin{align}\label{I1}
\int_{-d}^{0}  4 \sin^2  \kappa (x_2 +d) {\rm d} x_2 \leq  4 \int_{-d}^{0}
(\kappa (x_2 +d) )^2 {\rm d} x_{2} \leq \frac{4}{3}  \kappa^2 d^{3}
\end{align}
and
\begin{align}\label{I2}
\int_{-d}^{0}  4 \sin^2  \kappa (x_2 +d) {\rm d} x_2  = 2 (d - \frac{\sin 2
\kappa d}{2 \kappa}) = \frac{4}{3} \kappa^2 d^3 -O (\kappa^4 d^{5}) \geq
\frac{2}{3} \kappa^2 d^3.
\end{align}
It follows from \eqref{inq1}--\eqref{inq3} that
\begin{align*}
c_1 \frac{\epsilon}{ \kappa^2 d^2} \leq  \frac {\left| 2 \langle u^{\rm inc},
\phi_{0}\rangle_{\Gamma_{\epsilon}^+}  e^{{\rm i} \kappa d}\right|^2}
{\left| (1+c_0)- (1 -c_0)^{{\rm i} 2 \kappa d}\right|^2}  \leq c_2
\frac{\epsilon}{ \kappa^2 d^2}
\end{align*}
for some positive constants $c_1$ and  $c_2$.

Combining the above estimates, we obtain that there exist constants $C_3$
and $C_4$ independent  on $\lambda, \epsilon, d$ such that
\begin{align*}
 C_3 ^2  {\epsilon} d  \leq \|v_{\epsilon}\|^2_{L^2 (D_{\epsilon})} \leq C_4 ^2
{\epsilon} d,
\end{align*}
which completes the proof.
\end{proof}

\subsection{Accuracy of the approximate model}\label{sub2.4}

To show the accuracy of the approximate model, we introduce a TBC in
$\mathbb R^2_+$ and reformulate the problems \eqref{rdp} and \eqref{ardp} in a
bounded domain. Let $B^+_R=\{\boldsymbol x\in\mathbb R^2:
|\boldsymbol x-\boldsymbol x_c|<R, ~ x_2>0\}$ be the upper half disc with
radius $R$ centered at $\boldsymbol x_c=(\epsilon/2, 0)$, where $R$ is a
sufficiently large positive real number. Denote by $\partial
B^+_{R}=\{\boldsymbol x\in\mathbb R^2: |\boldsymbol x-\boldsymbol x_c|=R, ~
x_2>0\}$ be the upper half circle. Let $\Gamma^+_{R}=\left\{\boldsymbol
x\in\mathbb R^2: |x_1 - {\epsilon}/2|< R, x_2 =0 \right\}$ be the line segment
on $\Gamma_0$ with length $2 R$. The problem geometry is shown
in Figure \ref{pg_2}.

\begin{figure}
\centering
\includegraphics[width=0.4\textwidth]{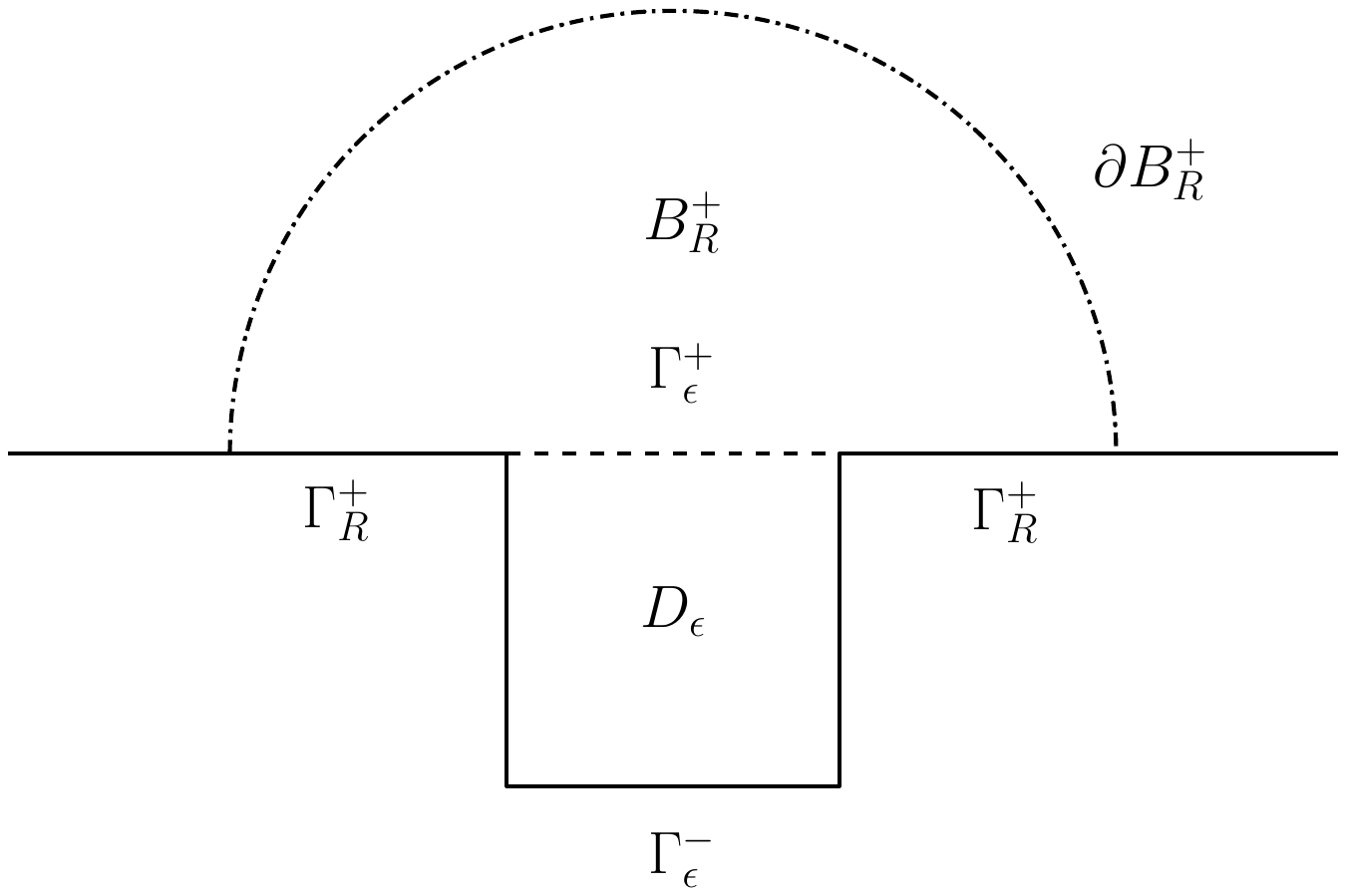}
\caption{The problem geometry with a transparent boundary.}
\label{pg_2}
\end{figure}

In the exterior domain $\mathbb R^2_+ \setminus\bar B_R^+$, by noting
the PEC condition on $\Gamma_0 \setminus \Gamma_R^+$ and the radiation
condition \eqref{rdc}, the scattered field can be expressed as the Fourier
series:
\begin{align*}
u_{\epsilon}^{\rm sc} (r, \eta) =\sum\limits_{n=0}^{\infty}
\frac{H_{n}^{(1)} (\kappa r)}{ H_{n}^{ (1)} (\kappa R)}
u_{{\epsilon}, n}^{\rm sc} \cos (n \eta), \quad 0< \eta < \pi,
\end{align*}
where $u_{{\epsilon}, n}^{\rm sc} = \frac{2}{\pi} \int_0^{\pi} u_{\epsilon}^{\rm
sc} (R, \eta) \cos (n \eta) {\rm d} \eta$, $H_{n}^{(1)}$ is
the Hankel function of the first kind with order $n$, and the polar coordinate
$r= \sqrt {(x_1 -{\epsilon}/2)^2 +x_2^2}$.
Taking the normal derivative of $u_{\epsilon}^{\rm sc}$ on $\partial B_{R}^+$
yields
\begin{align}\label{sdtn}
 \frac{\partial u_{\epsilon}^{\rm s} (R, \eta) }{\partial r} =
\sum\limits_{n=0}^{\infty} \frac{\kappa (H_{n}^{(1)})' (\kappa R)}{ H_{n}^{(1)}
(\kappa R)} u_{{\epsilon}, n}^{\rm sc} \cos (n \eta).
\end{align}

For any $v \in H^{1/2} (\partial B_{R}^+)$ with the Fourier expansion
\[
v= \sum\limits_{n=0}^{\infty} v_{n} \cos (n \eta),\quad v_{n} =\frac{2}{ \pi}
\int_0^{\pi} v (\eta) \cos (n \eta) {\rm d} \eta,
\]
we define the DtN operator
on $\partial B_{R}^+$:
\begin{align}\label{cdtn}
( \mathscr T v ) (\eta) = \sum\limits_{n=0}^{\infty} \frac{\kappa (
H_{n}^{(1) })'(\kappa R)}{H_{n}^{(1)} (\kappa R)} v_n \cos (n \eta), \quad 0 <
\eta < \pi,
\end{align}
It follows from \eqref{sdtn} and \eqref{cdtn} that the normal derivative of
total field on $\partial B_{R}^+$ can be written as
\begin{align}\label{tdtn}
 \partial_r u_{\epsilon} =\partial_r u_{\epsilon}^{\rm sc}+ \partial_r (u^{\rm
inc} + u^{\rm ref})=\mathscr T (u_{\epsilon}) +g,
\end{align}
where $g= \partial_r (u^{\rm inc} + u^{\rm ref}) -\mathscr T (u^{\rm inc} +
u^{\rm ref}).$

\begin{lemm}
 The DtN operator $\mathscr T: H^{1/2} (\partial B_{R}^+) \rightarrow
H^{-1/2} (\partial B_{R}^+)$ is bounded.
\end{lemm}

\begin{proof}
Let $A_{n}= \frac{(H_n^{(1)})' (\kappa R)}{H_{n}^{(1)} (\kappa R)}$.  By the
relation of the Hankel function and the modified  Bessel function
(cf. \cite{Abramowitz1972}), it yields that
 \[|A_n|= \left| \frac{K_{n}'( {\rm i} \kappa  R)}{ K_n ({\rm i} \kappa  R)} \right|,\]
 where $K_n$ is the  modified Bessel function of the second kind. Using the
formula $z K_n' (z) =- n K_{n} (z)-z K_{n-1} (z),$ we obtain
 \begin{align*}
|A_n |=\left| \frac{K_{n}'( {\rm i}\kappa  R)}{ K_n ({\rm i} \kappa  R)}
\right| =\left| \frac{n}{{\rm i} \kappa R} + \frac{K_{n-1} ({\rm i } \kappa R)}
{K_{n} ({\rm i } \kappa  R)} \right|\leq \frac{n}{ \kappa  R} +1 \leq c \sqrt
{1+ n^2},
\end{align*}
where $c$ is a positive constant.

For any $v, w \in H^{1/2} (\partial B_R^+) $, we have from the
definition of \eqref{cdtn} that
 \begin{align*}
 \left| \langle \mathscr T v, w  \rangle_{\partial{B}_{R}^+}\right|
&= \left| \kappa  R \int_{0}^{\pi} \sum_{n=0}^{\infty} A_n v_{n} \cos (n \eta)
\bar w {\rm d} \eta \right|=\left| \frac{\kappa  \pi R}{2} \sum
\limits_{n=0}^{\infty} A_n v_n \bar w_{n} \right|\\
& \leq \frac{ c \kappa  \pi R}{ 2} \sum\limits_{n=0}^{\infty} \sqrt {1+n^2}
|v_n \bar w_n| \\
&\leq  \frac{ c \kappa \pi R}{2}\left( \sum \limits_{n=0}^{\infty} \sqrt {1+n^2}
 |v_n|^2\right)^{\frac{1}{2}}\left( \sum \limits_{n=0}^{\infty} \sqrt {1+n^2}
|w_n|^2\right)^{\frac{1}{2}}\\
&= C(R, \lambda) \|v\|_{H^{1/2} (\partial B_{R}^+)}
\|w\|_{H^{1/2} (\partial B_{R}^+)}.
 \end{align*}
Thus we have
\begin{align*}
 \|\mathscr T v\|_{H^{-1/2} (\partial B_{R}^+)} = \sup \limits_{w \in
H^{1/2} (\partial B_{R}^+)} \frac{ \left| \langle \mathscr T v, w
\rangle_{\partial{B}_{R}^+}\right|}{ \|w\|_{H^{1/2} (\partial B_{R}^+)}
}\leq C(R, \lambda)  \|v\|_{H^{1/2} (\partial B_{R}^+)},
\end{align*}
which completes the proof.
\end{proof}

\begin{lemm}\label{rip}
We have
\begin{align*}
 {\rm Re} \langle \mathscr T v, v \rangle_{\partial B_{R}^+} \leq 0,  \quad
{\rm Im} \langle \mathscr T v, v \rangle_{\partial B_{R}^+}  > 0,
 \quad  \forall v\in H^{1/2} (\partial B_{R}^+).
\end{align*}
\end{lemm}

\begin{proof}
It follows from the definition \eqref{cdtn} that
\begin{align}\label{exp}
\langle \mathscr T v, v \rangle_{\partial B_{R}^+} = \frac {\kappa \pi R}{2}
\sum\limits_{n=0}^{\infty} A_n |v_n|^2.
\end{align}
Note that $H_{n}^{(1)} = J_{n} + {\rm i} Y_n$, where $J_n$ and $Y_n$ are the
Bessel functions of the first and second kind, respectively, and the modulus of
each Hankel function is decreasing function \cite{Abramowitz1972}, then
\begin{align}\label{RA}
 {\rm Re} A_n=   \frac{J_n (k R) J_n ' (k R) + Y_{n} (k R) Y_n ' (k R)}
{J_{n}^2 (k R) + Y_n^2 (k R)} =\frac{1}{2} \frac{ ( J_n ^2) ' (k R) + ( Y_n ^2)
'(k R)}{ J_{n}^2 (k R) + Y_n^2 (k R)}\leq0.
\end{align}
It  follows from the Wronskian formula  \cite{Abramowitz1972},
\begin{align}\label{Im}
 {\rm Im} A_n &= \frac{J_n (\kappa  R)  Y_n' (\kappa  R) - Y_n (\kappa  R)  J_n'
(\kappa  R) }{ J_{n}^2 (\kappa  R) + Y_n^2 (\kappa  R)} \nonumber\\
 &= \frac{W  (J_n (\kappa R), Y_n (\kappa R))}{ J_{n}^2 (\kappa  R) + Y_n^2 (\kappa R)}
 =\frac{2}{ \kappa \pi R } \frac{1}{ J_{n}^2 (\kappa R) + Y_n^2 (\kappa R)} >0,
\end{align}
where $W(\cdot, \cdot)$ denotes the Wronskian determinant of $J_n$ and $Y_n$.
The proof is completed by combining \eqref{exp}--\eqref{Im}.
\end{proof}

By the TBC operator \eqref{tdtn},  the model problem \eqref{rdp} can be
reformulated as follows:
\begin{align}\label{rdm}
\begin{cases}
 \Delta u_{\epsilon} + \kappa^2  u_{\epsilon}=0 \quad  &\text
{in}~B_R^+,\\
 \partial_\nu u_{\epsilon}=0 \quad & \text {on}~ \Gamma_{R}^+ \setminus
\Gamma_{\epsilon}^+,\\
 \partial_{x_2} u_{\epsilon} = \mathscr B^{\rm PMC} [u_{\epsilon}] \quad & \text
{on}~\Gamma_{\epsilon}^+,\\
 \partial_r u_{\epsilon}= \mathscr T [u_{\epsilon}] +g \quad & \text
{on}~\partial B_{R}^+.
 \end{cases}
 \end{align}
The variational formulation of \eqref{rdm} is to find $
u_{\epsilon} \in H^1 (B_R^+)$ such that
\begin{align}\label{vf1}
a (u_{\epsilon}, w)=\langle g, w\rangle_{\partial B_R^+} \quad\forall
w \in H^1 (B_R^+),
\end{align}
where the sesquilinear form
\begin{align}\label{bfn}
a (u_{ \epsilon}, w)=\int_{B_R^+} \left(  \nabla u_{\epsilon} \cdot \nabla
\bar{w} -\kappa^2 u_{\epsilon} \bar w \right) {\rm d} \boldsymbol x
 -\int_{\partial B_{R}^+} \mathscr T[u_{\epsilon}] \bar w  {\rm d} \gamma
+\int_{\Gamma_{\epsilon}^+} \mathscr B^{\rm PMC}[u_{\epsilon}] \bar w {\rm d}
\gamma.
\end{align}

\begin{theo}\label{wpst}
The variational problem \eqref{vf1} has a unique solution $u_{\epsilon} \in H^1
(B_R^+)$.
\end{theo}

\begin{proof}
By the definition \eqref{dtn}, we have for  sufficiently small $\epsilon$ that
\begin{align*}
\langle \mathscr B^{\rm PMC} [u_{\epsilon}], u_{\epsilon}
\rangle_{\Gamma_{\epsilon}^+}&= \sum\limits_{n=0}^{\infty} {\rm i} \beta_{n}
\frac{e^{{\rm i} 2\beta_n d}+1}{e^{{\rm i} 2\beta_n d}-1} |u_{{\epsilon}, n}
|^2\\
&= \frac{ \kappa \cos \kappa d}{ \sin \kappa d} |u_{{\epsilon}, 0}|^2+
\sum\limits_{n=1}^{\infty} \sqrt {(n \pi / \epsilon)^2 - \kappa^2} \frac{ e^{-
\sqrt{ (n \pi / \epsilon)^2 - \kappa^2}} +1}{1- e^{- \sqrt{ (n \pi / \epsilon)^2
- \kappa^2}}}|u_{{\epsilon}, n}|^2>0.
\end{align*}
It follows from Lemma \ref{rip} and the above inequality  that
\begin{align*}
{\rm Re}  \left\{a (u_{\epsilon}, u_{\epsilon}) \right\}
&=\|\nabla u_{\epsilon}\|^2_{L^2 (B_R^+)} - \kappa^2 \|u_{\epsilon}\|^2_{L^2
(B_R^+)} -{\rm Re}\langle \mathscr T [u_{\epsilon}], u_{\epsilon}
\rangle_{\partial B_{R}^+} + \langle \mathscr B^{\rm PMC} [u_{\epsilon}],
u_{\epsilon} \rangle_{\Gamma_{\epsilon}^+}\\
& \geq \|\nabla u_{\epsilon}\|^2_{L^2 (B_R^+)} - \kappa^2
\|u_{\epsilon}\|^2_{L^2 (B_R^+)},
\end{align*}
which shows that the sesquilinear form $a(\cdot, \cdot)$ satisfies a G{\aa}rding
type inequality.

Next is show the uniqueness. It suffices to show that $u_{\epsilon}=0$ if
$g=0$. A simple calculation yields
\begin{align*}
{\rm Im} a (u_{\epsilon}, u_{\epsilon})=- {\rm Im } \langle \mathscr T
[u_{\epsilon}], u_{\epsilon} \rangle_{\partial B_{R}^+} =-\frac {\kappa \pi
R}{2} \sum\limits_{n=0}^{\infty}  {\rm Im }A_n |u_{{\epsilon},n}|^2=0.
\end{align*}
It follows from Lemma \ref{rip} that  $u_{{\epsilon}, n}=0$ for all  $n$.
Therefore $u_{\epsilon}=0$ on $\partial B_{R}^+.$  By the definition
\eqref{cdtn}, we have $\mathscr T u_{\epsilon}=0$, which implies $\partial_r
u_{\epsilon}=\mathscr T u_{\epsilon}=0$ on $\partial B_{R}^+$ by \eqref{tdtn}.
Hence we have $u_\epsilon=0$ in $\mathbb R^2_+\setminus\bar{B}_R^+$. It
follows from the unique continuation \cite{Jerison1985} that
$u_{\epsilon}= 0$ in $B_R^+$. The proof is completed by applying the
Fredholm alternative theorem.
\end{proof}

Correspondingly, the approximate model problem \eqref{ardp} can be written as
follows:
\begin{align}\label{ardm}
 \begin{cases}
  \Delta v_{\epsilon} + \kappa^2 v_{\epsilon}=0 \quad  & \text {in}~B_R^+,\\
  \partial_\nu v_{\epsilon}=0 \quad & \text {on}~\Gamma_{R}^+ \setminus
\Gamma_{\epsilon}^+,\\
  \partial_{x_2} v_{\epsilon} = \mathscr B^{\rm PMC}_0 [v_{\epsilon}] \quad &
\text {on}~\Gamma_{\epsilon}^+,\\
  \partial_r v_{\epsilon} = \mathscr T [v_{\epsilon}] +g \quad &  \text
{on}~\partial B_{R}^+.
 \end{cases}
\end{align}
The variational formulation of \eqref{ardm} is to find $v_{\epsilon} \in H^1
(B_R^+)$ such that
\begin{align}\label{avf1}
a_0  (v_{\epsilon}, w)= \langle g, w\rangle_{\partial B_R^+}\quad\forall w \in
H^1 (B_R^+),
\end{align}
where the sesquilinear form
\begin{align}\label{abfn}
a_0 (v_{\epsilon}, w) =\int_{B_R^+} \left(  \nabla v_{\epsilon} \cdot
\nabla \bar{\epsilon} -\kappa^2 v_{\epsilon} \bar w \right) {\rm d} \boldsymbol
x -\int_{\partial B_{R}^+} \mathscr T [v_{\epsilon}] \bar w  {\rm d} \gamma
+\int_{\Gamma_{\epsilon}^+} \mathscr B^{\rm PMC}_0 [v_{\epsilon}] \bar w {\rm d}
\gamma.
\end{align}
The well-posedness of the variational problem \eqref{avf1} can be shown
similarly to Theorem \ref{wpst}. The proof is omitted for brevity.

In the cavity $D_{\epsilon}$, the wave fields $u_{\epsilon}$ and $v_{\epsilon}$
are given by \eqref{wf} and \eqref{om}--\eqref{za}, respectively. Their
Fourier coefficients can be calculated from the solutions
obtained in \eqref{rdm} and \eqref{ardm}. The following lemma gives the accuracy
of approximate model \eqref{ardm} and \eqref{om}--\eqref{za}.

\begin{lemm}\label{accm}
Let $u_{\epsilon}$ be the expansion of \eqref{wf} inside the cavity  with the
Fourier coefficients given by the model \eqref{rdm}, and let $v_{\epsilon}$ be
the expansion of \eqref{om}--\eqref{za} with the Fourier coefficients given by
the approximate  model \eqref{ardm}, then the following estimates hold:
\begin{align*}
\|\nabla u_{\epsilon} -\nabla v_{\epsilon}\|_{L^2 (D_{\epsilon})} \leq  C
(\lambda, d, R)   {\epsilon}, \quad \|u_{\epsilon} -v_{\epsilon}\|_{L^2
(D_{\epsilon})} \leq C (\lambda, d, R) {\epsilon}  \sqrt {{\epsilon}
|\ln
{\epsilon}|},
\end{align*}
where $C(\lambda, d, R)$ is a positive constant independent of $\epsilon$.
\end{lemm}

\begin{proof}
The proof consists of five steps in order to show the estimates between
$u_\epsilon$ and $v_\epsilon$.

Step 1: we first give the estimate of $\|u_{\epsilon} -u_{0}\|_{H^1
(B_R^+)}$ and $\|v_{\epsilon} -u_{0}\|_{H^1 (B_R^+)}$, where $u_0$ is
the solution of \eqref{rdm} when there is no cavity, i.e., $u_0$ satisfies the
variational problem
\begin{align}\label{uvf}
b(u_0, w) = \langle g, w\rangle_{\partial B_R^+} \quad\forall w \in
H^1 (B_R^+),
\end{align}
where the sesquilinear form
\begin{align*}
b(u_0, w) =\int_{B_R^+} \left( \nabla u_0 \cdot \nabla \bar w -
\kappa^2 u_{0} \bar w \right) {\rm d} \boldsymbol x - \int_{\partial B_{R}^+}
\mathscr T[u_0] \bar w {\rm d} \gamma.
\end{align*}
It follows from \eqref{vf1} and \eqref{uvf} that
\begin{align}\label{i1}
a(u_{\epsilon}-u_0, w )=\langle g, w\rangle_{\partial B_R^+}-a (u_0, w)=b
(u_0, w) -a (u_0, w)=-\int_{\Gamma_{\epsilon}^+} \mathscr B^{\rm PMC} [u_0] \bar
w {\rm d} \gamma.
\end{align}
Using \eqref{dtn} and \eqref{adtn}, we have from Lemma \ref{tth} that
\begin{align}\label{est}
|b(u_0, w) -a (u_0, w)|
&\leq \left|\int_{\Gamma_{\epsilon}^+} \mathscr B^{\rm PMC}_0 [u_0]  \bar w {\rm
d} \gamma \right| +\left|\int_{\Gamma_{\epsilon}^+} (\mathscr B^{\rm PMC} -
\mathscr B^{\rm PMC}_0) [u_0]  \bar w {\rm d} \gamma \right| \notag\\
& \leq \frac{\kappa}{|\sin \kappa d|} |u_{0, 0}^+| |\bar w_0^+|+C
\|u_0\|_{H^{1/2} (\Gamma_{\epsilon}^+)} \|w\|_{H^{1/2}
(\Gamma_{\epsilon}^+)},
\end{align}
where $u_{0, 0}^+=\langle u_0, \phi_0\rangle_{\Gamma_\epsilon^+}$ and
$w_0^+=\langle w, \phi_0\rangle_{\Gamma_\epsilon^+}$.

Using the estimates \cite[Lemma A.1, Lemma A.2]{lin2015}
\begin{align*}
&|u_{0, 0}^+| \leq C (R) \sqrt {\epsilon} \|u_{0}\|_{H^2 (B_R^+)} \quad
\text {if}~ u_0 \in H^2 (B_R^+),\\
&|w_0^+| \leq C (R) \sqrt {{\epsilon} |\ln {\epsilon}|} \|w\|_{H^1 (B_R^+)}
\quad \text {if}~ w \in H^1 (B_R^+),\\
& \|u_0\|_{H^{1/2} (\Gamma_{\epsilon}^+)} \leq C (R) {\epsilon} \|u_0\|_{H^3
(B_R^+)} \quad  \text {if}~ u_0 \in H^3 (B_R^+), \\
& \|w\|_{H^{1/2} (\Gamma_{\epsilon}^+)} \leq C (R) \|w\|_{H^1 (B_R^+)}
\quad \text {if}~ w \in H^1 (B_R^+),
\end{align*}
we obtain from \eqref{est} that
\begin{align*}
|b(u_0, w) -a (u_0, w)| \leq C (\lambda, d, R) \left({\epsilon}  \sqrt {|\ln
{\epsilon}|} \|u_{0}\|_{H^2 (B_R^+)}\|w\|_{H^1 (B_R^+)} + {\epsilon}
\|u_0\|_{H^3 (B_R^+)} \|w\|_{H^1 (B_R^+)} \right).
\end{align*}
Since $u_0 =u^{\rm inc} +u^{\rm ref} =e^{{\rm i} \kappa
\boldsymbol x\cdot\boldsymbol d } +e^{{\rm i} \kappa  \boldsymbol
x\cdot\boldsymbol d'}$, we have from straightforward calculations that
\begin{align}\label{i2}
|b(u_0, w) -a (u_0, w)| \leq C (\lambda, d, R)   {\epsilon} \sqrt {|\ln
{\epsilon}|}\|w\|_{H^1 (B_R^+)}.
\end{align}

Denote by $\mathscr L_{\epsilon}$ be the induced operator for the sesquilinear
form \eqref{bfn} such that
\begin{align*}
a_{\epsilon} (u_{\epsilon}, w)= \left( \mathscr L_{\epsilon} u_{\epsilon}, w
\right),
\end{align*}
where $\left( \cdot, \cdot \right)$ is the inner product on $L^2 (B_R^+)$.
It follows from \cite[Lemma A.3]{lin2015} that there exists a positive constant
${\epsilon}_0$ such that
\begin{align}\label{i4}
\|\mathscr L_{\epsilon}^{-1}\|_{\mathcal L (H^1 (B_R^+))} \leq C (
\lambda, d, R)\quad\forall~ 0 < {\epsilon} < {\epsilon}_0,
\end{align}
where $ C (\lambda, d, R)$ is a constant independent of ${\epsilon}$.
Combing \eqref{i1}--\eqref{i4} yields
\begin{align*}
\|u_{\epsilon} -u_0\|_{H^1 (B_R^+)} \leq   C (\lambda, d, R)  {\epsilon}
\sqrt {|\ln {\epsilon}|}.
\end{align*}

Since $v_{\epsilon}$ satisfies the variational problem  \eqref{avf1}, following
the same arguments as above, we may show for sufficiently small $\epsilon$ that
\begin{align}\label{i6}
\|v_{\epsilon} -u_{0}\|_{H^1 (B_R^+)} \leq C (
\lambda, d, R) {\epsilon} \sqrt {|\ln {\epsilon}|}.
\end{align}

Step 2: we estimate $\|v_{\epsilon} -u_0\|_{H^{3/2}
(B_R^+)}$. Let $\xi =v_{\epsilon} -u_0$. It is easy to verify that $\xi$
satisfies
\begin{align*}
\begin{cases}
\Delta \xi + \kappa^2 \xi =0  \quad  &\text {in}~ B_R^+,\\
\partial_\nu \xi =0 \quad &\text {on}~\Gamma_{R}^+ \setminus
\Gamma_{\epsilon}^+,\\
\partial_{x_2} \xi =\mathscr B^{\rm PMC}_0 [v_{\epsilon}] \quad &\text {on}~
\Gamma_{\epsilon}^+,\\
\partial_r \xi =\mathscr T \xi \quad & \text {on}~\partial B_{R}^+.
\end{cases}
\end{align*}
It follows from the standard regularity estimate for the Neumann problem that
\begin{align}
\|\xi\|_{H^{3/2} (B_R^+)}
&\leq C (\lambda, d, R) \|\mathscr B^{\rm PMC}_0 [v_{\epsilon}]\|_{L^2
(\Gamma_{\epsilon}^+)}\nonumber\\
&\leq C (\lambda, d, R) \left( |u_{0,0}^+| +|v_{w, 0}^+ -u_{0, 0}^+|\right)
\|\phi_0\|_{L^2 (0, {\epsilon})}\nonumber\\
& \leq C (\lambda, d, R) \left( \sqrt {\epsilon} \|u_0\|_{H^2 (B_R^+)} +
\sqrt {{\epsilon} |\ln {\epsilon}|} \|v_{\epsilon} -u_{0}\|_{H^1 (B_R^+)}
\right)\nonumber\\
&\leq C (\lambda, d, R)   \sqrt {\epsilon},\label{i6.1}
\end{align}
where the last inequality is obtained by using \eqref{i6}.

Step 3: This step is to estimate $\|u_{\epsilon} -v_{\epsilon}\|_{H^1
(B_R^+)}$. From \eqref{vf1} and \eqref{avf1}, we have
\begin{align}\label{i7}
a (u_{\epsilon} -v_{\epsilon}, w)=a_0 (v_{\epsilon}, w) -a( v_{\epsilon}, w).
\end{align}
Comparing \eqref{bfn} and \eqref{abfn} yields
\begin{align}\label{i8}
&|a_0 (v_{\epsilon}, w) -a( v_{\epsilon}, w)| \leq \left|
\int_{\Gamma_{\epsilon}^+} (\mathscr B^{\rm PMC}- \mathscr
B^{\rm PMC}_0)[v_{\epsilon}] \bar w {\rm d} \gamma\right| \nonumber\\
&\leq  \left| \int_{\Gamma_{\epsilon}^+} (\mathscr B^{\rm PMC}- \mathscr
B_0^{\rm PMC})[u_0] \bar w {\rm d} \gamma\right| +\left|
\int_{\Gamma_{\epsilon}^+} (\mathscr B^{\rm PMC}- \mathscr
B_0^{\rm PMC})[v_{\epsilon}-u_0] \bar w {\rm d} \gamma\right|.
\end{align}
Using the estimates \eqref{est}--\eqref{i2} gives
\begin{align*}
\left| \int_{\Gamma_{\epsilon}^+} (\mathscr B^{\rm PMC}- \mathscr B_0^{\rm
PMC})[u_0] \bar w {\rm d} \gamma\right| \leq  C (\lambda, d, R) {\epsilon}~
\|w\|_{H^1 (B_R^+)}.
\end{align*}
For the second term in \eqref{i8}, we have from the estimate of $\|v_{\epsilon}
-u_0\|_{H^{3/2} (\Omega_{R})}$ that
 \begin{align*}
 \left| \int_{\Gamma_{\epsilon}^+} (\mathscr B^{\rm PMC}- \mathscr
B^{\rm PMC}_{0})[v_{\epsilon}-u_0] \bar w {\rm d} \gamma\right|  & \leq  C (
\lambda, d, R) \|v_{\epsilon} -u_{0}\|_{H^{1/2} (\Gamma_{\epsilon}^+)}
\|w\|_{H^{1/2} (\Gamma_{\epsilon}^+)}\\
& \leq   C (\lambda, d, R)  \sqrt {\epsilon} \|v_{\epsilon} -u_{0}\|_{H^{3/2}
(B_R^+)} \|w\|_{H^1 (B_R^+)}\\
&\leq   C (\lambda, d, R) {\epsilon} \|w\|_{H^1 (B_R^+)}.
 \end{align*}
For the variational problem \eqref{i7}, a combination of the above estimates
and the boundedness of inverse of the linear operator $\mathscr L_{\epsilon}$ in
\eqref{i4} lead to the following estimate
 \begin{align}\label{i9}
 \|u_{\epsilon} -v_{\epsilon}\|_{H^1 (B_R^+)} \leq  C (\lambda, d, R){\epsilon}.
 \end{align}

 Step 4: we estimate $\|\nabla u_{\epsilon}- \nabla v_{\epsilon}\|_{L^2
(D_{\epsilon})}$.  We have from the expansions
 of $u_{\epsilon}$ in \eqref{eqe} and $v_{\epsilon}$ in \eqref{om} that
 \begin{align*}
 \|\nabla u_{\epsilon} -\nabla v_{\epsilon}\|^2_{L^2 (D_{\epsilon})}
 &= \int_{-d}^0 \kappa^2 \left| - (\alpha_0^+ -\tilde \alpha_0^+)e^{-{\rm i}
\kappa x_2} +(\alpha_0^- -\tilde \alpha_0^- )e^{{\rm i} \kappa (x_2
+d)})\right|^2 {\rm d }x_2 \\
&\hspace{1cm}+\int_{-d}^0 \bigg(  \sum\limits_{n=1}^{\infty} \left(\frac{n \pi
}{ {\epsilon}} \right)^2 \left|\alpha_{n}^+ e^{-{\rm i}
\beta_n x_2} +\alpha_{n}^- e^{{\rm i} \beta_n (x_2 +d)} \right|^2 \\
&\hspace{1cm}+ ({\rm i} \beta_n)^2 \left| -\alpha_{n}^+ e^{-{\rm i} \beta_n x_2}
+\alpha_{n}^- e^{{\rm
i} \beta_n (x_2 +d)}\right|^2\bigg){\rm d}x_2\\
 &\leq 2 \kappa^2 d \left(|\alpha_0^+ -\tilde \alpha_0^+|^2 + |\alpha_0^-
-\tilde \alpha_0^-|^2\right) +\sum\limits_{n=1}^{\infty} \frac{1+ 2 (n \pi /
{\epsilon})^2}{  n \pi / 2 {\epsilon}} \left( |\alpha_n^+|^2
+|\alpha_n^-|^2\right),
 \end{align*}
where the Parseval identity is used in the last inequality.

For the first term  ($n=0$) in the above inequality,  following from  \eqref{ec}
and \eqref{za} and  noting that $d \ll \lambda$, we get
 \begin{align}
 &|\alpha_0^+ -\tilde \alpha_0^+|^2 = \frac{|u_{{\epsilon},0}^+ -
v_{{\epsilon},0}^+|^2}{|e^{{\rm i} 2 \kappa d} -1|^2} \leq  \frac{2}{
\kappa^2 d^2 }|u_{{\epsilon},0}^+ - v_{{\epsilon},0}^+|^2, \nonumber\\
 &|\alpha_0^- -\tilde \alpha_0^-|^2 =\frac{| e^{{\rm i} \kappa
d}(u_{{\epsilon},0}^+ - v_{{\epsilon},0}^+)|^2}{|e^{{\rm i} 2 \kappa d}
-1|^2}\leq  \frac{2}{ \kappa^2 d^2 }|u_{{\epsilon},0}^+ -
v_{{\epsilon},0}^+|^2.\label{p1}
 \end{align}
 For the second term ($n \geq 1$), by \eqref{ec} and ${\rm i} \beta_n  d = -
d\sqrt {(n \pi / {\epsilon})^2 - \kappa^2}$, we have
 \begin{align}\label{p2}
|\alpha_n^+|^2=  \frac{|- u_{{\epsilon},n}^+|^2}{|e^{{\rm i} 2 \beta_{n}d}-1|^2}
\leq 2 | u_{{\epsilon},n}^+|^2, \quad |\alpha_n^- |^2= \frac{| e^{{\rm i}
\beta_n d} u_{{\epsilon},n}^+|^2}{|e^{{\rm i} 2 \beta_{n}d}-1|^2} \leq  2|
u_{{\epsilon},n}^+|^2.
 \end{align}
Consequently, we obtain
 \begin{align*}
 \|\nabla u_{\epsilon} -\nabla v_{\epsilon}\|^2_{L^2 (D_{\epsilon})} &\leq C (
\lambda, d, R) \left( |u_{{\epsilon},0}^+ - v_{{\epsilon},0}^+|^2 +
\sum\limits_{n=1}^{\infty} \left(1 + \left(\frac{n \pi}{\epsilon} \right)^2
\right)^{1/2}  |u_{{\epsilon},n}^+|^2\right)\\
&\leq C (\lambda, d, R) \left(
|u_{{\epsilon},0}^+ - v_{{\epsilon},0}^+|^2  +\|u_{\epsilon}\|^2_{H^{1/2}
(\Gamma_{\epsilon}^+)} \right).
 \end{align*}
It follows from \eqref{i6.1} and \eqref{i9} that
 \begin{align}\label{i10}
 \|u_{\epsilon}\|_{H^{1/2} (\Gamma_{\epsilon}^+)} & \leq \|u_0\|_{H^{1/2}
(\Gamma_{\epsilon}^+)}+\|v_{\epsilon} - u_0\|_{H^{1/2} (\Gamma_{\epsilon}^+)}
+\|u_{\epsilon} -v_{\epsilon}\|_{H^{1/2} (\Gamma_{\epsilon}^+)}\nonumber\\
 & \leq  C (\lambda, d, R) \left( {\epsilon} \|u_{0}\|_{H^3 (B_R^+)}+ \sqrt
{\epsilon} \|v_{\epsilon} -u_0\|_{H^{3/2} (B_R^+)} + \|u_{\epsilon}
-v_{\epsilon}\|_{H^1 (B_R^+)} \right)\nonumber\\
& \leq  C (\lambda, d, R) {\epsilon}.
 \end{align}
On the other hand, we have from \eqref{i9} that
 \begin{align}\label{i11}
 |u_{{\epsilon},0}^+ - v_{{\epsilon},0}^+|  \leq   C (\lambda, d, R)   \sqrt
{{\epsilon} |\ln {\epsilon}|} \|u_{\epsilon} -v_{\epsilon}\|_{H^1 (B_R^+)}
\leq  C (\lambda, d, R)   {\epsilon}  \sqrt {{\epsilon} |\ln {\epsilon}|}.
 \end{align}
Combining the above estimates yields
 \begin{align*}
\|\nabla u_{\epsilon} -\nabla v_{\epsilon}\|^2_{L^2
(D_{\epsilon})}  \leq  C (\lambda, d, R)  {\epsilon}.
 \end{align*}

Step 5: the last step is to estimate $\|u_{\epsilon} -v_{\epsilon}\|_{L^2
(D_{\epsilon})}$. Using the similar method as that in Step 4, we can obtain
from the Parseval identity and \eqref{p1}--\eqref{p2} that
 \begin{align*}
 \|u_{\epsilon} -v_{\epsilon}\|^2_{L^2 (D_{\epsilon})} &\leq 2 d
\left(|\alpha_0^+ -\tilde \alpha_0^+|^2 + |\alpha_0^- -\tilde
\alpha_0^-|^2\right) +\sum\limits_{n=1}^{\infty} \frac{1}{  n \pi / 2
{\epsilon}} \left( |\alpha_n^+|^2 +|\alpha_n^-|^2\right)\\
&\leq   C (\lambda, d, R)  \left( |u_{{\epsilon},0}^+ - v_{{\epsilon},0}^+|^2 +
\sum\limits_{n=1}^{\infty}  \left(\frac{\epsilon}{n \pi} \right)
|u_{{\epsilon},n}^+|^2 \right)\\
 & \leq  C (\lambda, d, R)  \left( |u_{{\epsilon},0}^+ - v_{{\epsilon},0}^+|^2
+{\epsilon}^2 \sum\limits_{n=1}^{\infty}  \left(1 + \left(\frac{n \pi}{\epsilon}
\right) \right)^2  |u_{{\epsilon},n}^+|^2\right)\\
& \leq C (\lambda, d, R) \left( |u_{{\epsilon},0}^+ - v_{{\epsilon},0}^+|^2
+{\epsilon}^2 \|u_{\epsilon}\|_{H^{1/2} (\Gamma_{\epsilon}^+)} \right)\\
 &\leq C (\lambda, d, R)   \left(  {\epsilon}  \sqrt {{\epsilon} |\ln
{\epsilon}|}+  {\epsilon}^3 \right)\\
& \leq 2  C (\lambda, d, R) {\epsilon}
\sqrt {{\epsilon} |\ln {\epsilon}|},
 \end{align*}
where the estimates \eqref{i10}--\eqref{i11} are used.
\end{proof}

\subsection{Field enhancement without resonance}

The electromagnetic field enhancement factor is defined as the ratio of
the energy between the total field and the incident field in the cavity. Hence
the electric and magnetic field enhancement factors are given by
\begin{align*}
Q_{\boldsymbol E}= \frac{\|\boldsymbol E_{\epsilon}\|_{L^2
(D_{\epsilon})}}{\|\boldsymbol E^{\rm inc}\|_{L^2 (D_{\epsilon})}} \quad
\text{and} \quad Q_{\boldsymbol H}= \frac{\|\boldsymbol H_{\epsilon}\|_{L^2
(D_{\epsilon})}}{\|\boldsymbol H^{\rm inc}\|_{L^2 (D_{\epsilon})}}.
\end{align*}

\begin{theo}\label{NTH}
There exists a constant ${\epsilon}_0$ depending on $\lambda$ and $d$ such
that for all $0 < {\epsilon} < {\epsilon}_0$ the following estimates hold:
\begin{align*}
C_1 \left(\lambda/d\right) \leq  Q_{\boldsymbol E} \leq C_2  \left(
\lambda/d\right), \quad  C_1 \leq Q_{\boldsymbol H} \leq C_2,
\end{align*}
where $C_1$ and $C_2$ are some positive constants independent of $\lambda,
{\epsilon}_0, {\epsilon}, d$.
\end{theo}

\begin{rema}
The electric field enhancement has an order $O\left(\lambda/d\right)$,  which is
enormous due to  the length scale $d \ll \lambda$, while the magnetic field has
no significant enhancement in this situation.
\end{rema}

\begin{proof}
Recalling the scale assumption ${\epsilon} \ll d \ll \lambda,$ we have from
Theorem \ref{get} and Lemma \ref{accm} that
\begin{align*}
\|\nabla u_{\epsilon}\|_{L^2 (D_{\epsilon})} &\leq \|\nabla v_{\epsilon}\|_{L^2
(D_{\epsilon})} +\|\nabla u_{\epsilon} -\nabla v_{\epsilon}\|_{L^2
(D_{\epsilon})}\\
& \leq C_2 \sqrt {\frac{{\epsilon}}{d}} +   C (d, \lambda, R) {\epsilon} \leq
\frac{3 C_2}{2}  \sqrt {\frac{{\epsilon}}{d}},
\end{align*}
where we use the fact that $C(\lambda, d, R)\sqrt {{\epsilon} d} \leq
\frac{C_2}{2}$ for sufficiently small ${\epsilon}$. On the
other hand,
\begin{align*}
\|\nabla u_{\epsilon}\|_{L^2 (D_{\epsilon})} &\geq \|\nabla v_{\epsilon}\|_{L^2
(D_{\epsilon})} -\|\nabla u_{\epsilon} -\nabla v_{\epsilon}\|_{L^2
(D_{\epsilon})}\\
& \geq C_1   \sqrt {\frac{{\epsilon}}{d}} -   C (\lambda, d, R) {\epsilon}
\geq \frac{C_1}{2} \sqrt {\frac{{\epsilon}}{d}}.
\end{align*}
A straightforward calculation yields
\begin{align*}
\|\nabla u^{\rm inc}\|^2_{L^2
(D_{\epsilon})}=\int_{-d}^{0}\int_0^{\epsilon} |\nabla u^{\rm
inc}(\boldsymbol x)|^2 {\rm d} \boldsymbol x= \kappa^2 {\epsilon}d .
\end{align*}
Combining the above estimates gives
\begin{align*}
 \frac{C_1}{ 4 \pi} \frac{\lambda}{ d}\leq \frac{\|\nabla u_{\epsilon}\|_{L^2
(D_{\epsilon})}}{\|\nabla u^{\rm inc}\|_{L^2 (D_{\epsilon})}} \leq\frac{3 C_2}{
4 \pi} \frac{\lambda}{d}.
\end{align*}
Noting that $|\nabla \times \boldsymbol H_{\epsilon}| = |\nabla \times (0, 0,
u_{\epsilon})|= |\nabla u_{\epsilon}|$ and $\nabla \times \boldsymbol
H_{\epsilon} =- {\rm i} \omega \varepsilon \boldsymbol E_{\epsilon}$,
we obtain
\begin{align*}
\frac{C_1}{ 4 \pi} \frac{\lambda}{ d}\leq  \frac{\|\nabla u_{\epsilon}\|_{L^2
(D_{\epsilon})}}{\|\nabla u^{\rm inc}\|_{L^2 (D_{\epsilon})}} = \frac{\|
\boldsymbol E_{\epsilon}\|_{L^2 (D_{\epsilon})}}{\| \boldsymbol E^{\rm
inc}\|_{L^2 (D_{\epsilon})}}=Q_{\boldsymbol E}   \leq \frac{3 C_2}{ 4 \pi}
\frac{\lambda}{d}.
\end{align*}

For the magnetic field, it follows from Theorem \ref{ght} and   Lemma \ref{accm} that
\begin{align*}
\|u_{\epsilon}\|_{L^2 (D_{\epsilon})} & \leq   \|v_{\epsilon}\|_{L^2
(D_{\epsilon})} + \|u_{\epsilon}-v_{\epsilon}\|_{L^2 (D_{\epsilon})} \\
&\leq C_2 \sqrt {{\epsilon} d}+ C (\lambda, d, R) {\epsilon}  \sqrt {{\epsilon}
|\ln {\epsilon}|} \leq \frac{3}{2}C_2 \sqrt {{\epsilon} d}
\end{align*}
and
\begin{align*}
\|u_{\epsilon}\|_{L^2 (D_{\epsilon})} &\geq \|v_{\epsilon}\|_{L^2
(D_{\epsilon})} + \|u_{\epsilon}-v_{\epsilon}\|_{L^2 (D_{\epsilon})}\\
& \geq C_1 \sqrt {{\epsilon} d}+ C (\lambda, d, R) {\epsilon}  \sqrt {{\epsilon}
|\ln {\epsilon}|} \geq \frac{1}{2}C_1 \sqrt {{\epsilon} d},
\end{align*}
where we use $C (\lambda, d, R) {\epsilon} \sqrt {|\ln
{\epsilon}|/d} \leq \min \left\{ C_1/2, C_2/ 2 \right\}$ for
sufficiently small $\epsilon$. Clearly, we have
\begin{align*}
\|u^{\rm inc}\|^2_{L^2 (D_{\epsilon})}= \int_{-d}^0\int_{0}^{\epsilon} |u^{\rm
inc}(\boldsymbol x)|^2 {\rm d} \boldsymbol x=  {\epsilon} d.
\end{align*}
Therefore, we obtain
\begin{align*}
\frac{1}{2}C_1 \leq \frac{\|u_{\epsilon}\|_{L^2 (D_{\epsilon})}}{ \|u^{\rm
inc}\|_{L^2 (D_{\epsilon})}} = \frac{\| \boldsymbol H_{\epsilon}\|_{L^2
(D_{\epsilon})}}{\| \boldsymbol H^{\rm inc}\|_{L^2
(D_{\epsilon})}}=Q_{\boldsymbol H}   \leq \frac{3}{2}C_2,
\end{align*}
which completes the proof.
\end{proof}

\section{PEC-PMC cavity with resonances}\label{sec3}

This section is to discuss the Fabry--Perot type resonance and derive the
asymptotic expansions for those resonance; analyze quantitatively the field
enhancement at the resonance frequencies in a narrow cavity, i.e., $\epsilon \ll
\lambda.$

Recall the model problem \eqref{hlme}--\eqref{rdc}:
\begin{align}\label{fe}
\begin{cases}
 \Delta u_{\epsilon}+\kappa^2 u_{\epsilon}= 0 \quad  &\text
{in}~\Omega,\\
 \partial_\nu u_{\epsilon} =0 \quad  &\text {on}~ \partial \Omega \setminus
\Gamma_{\epsilon}^{-} ,\\
 u_{\epsilon}=0 \quad  & \text {on}~ \Gamma_{\epsilon}^{-},\\
 \lim\limits_{r \rightarrow \infty} \sqrt r \left( \partial_r u_{\epsilon}^{\rm
s} - {\rm i} \kappa u_{\epsilon}^{\rm s} \right)=0 \quad & \text{in}~\mathbb
R_+^2.
\end{cases}
\end{align}
It is known that the scattering problem \eqref{fe} has a unique solution for
a complex wavenumber $k$ with ${\rm Im} k \geq 0$ (cf.,\cite{Ammari2002}).
By analytic continuation, the solution has a meromorphic extension to the
whole complex plane, except for a countable number of points, which are poles of
the resolvent associated with the scattering problem \eqref{fe}. These poles
are called the resonances of the scattering problem, and the associated
nontrivial solutions are called resonance modes. If the frequency of the
incident wave is taken to be close to the real part of the resonance, an
enhancement of the field is expected if the imaginary part of the resonance is
small.

\subsection{Boundary integral equation}\label{sub3.1}

Define by $g^{\rm PMC}_{\epsilon} (\boldsymbol x, \boldsymbol y)$  the Green
function in the PEC-PMC cavity $D_{\epsilon}$, i.e., it satisfies
\begin{align*}
\begin{cases}
\Delta g^{\rm PMC}_{\epsilon} (\boldsymbol x, \boldsymbol y) + \kappa^2
g^{\rm PMC}_{\epsilon} (\boldsymbol x, \boldsymbol y)= -\delta (\boldsymbol x-
\boldsymbol y), \quad &
\boldsymbol x, \boldsymbol y \in D_{\epsilon},\\
\frac{\partial g^{\rm PMC}_{\epsilon} (\boldsymbol x, \boldsymbol y)}{\partial
\nu_{\boldsymbol y}} =0 \quad &\text {on}~ \partial D_{\epsilon} \setminus
\Gamma_{\epsilon}^-,\\
g^{\rm PMC}_{\epsilon} (\boldsymbol x, \boldsymbol y) =0 \quad
&\text {on}~\Gamma_{\epsilon}^-.
\end{cases}
\end{align*}
Since the width of the cavity $\epsilon$ tends to zero, we may assume that
$\kappa^2$ is not the eigenvalue of the Laplacian with the above boundary
conditions in the cavity. Thus the Green function of the Helmholtz operator in
$D_{\epsilon}$ exists and can be expressed as \cite{Collin1960}
\begin{align*}
g^{\rm PMC}_{\epsilon} (\boldsymbol x, \boldsymbol y)=\sum\limits_{m, n
=0}^{\infty} c_{m, n} \phi_{m, n} (\boldsymbol x)\phi_{m, n} (\boldsymbol y),
\end{align*}
where
\begin{align}\label{cad}
&c_{m, n} = \frac{1}{ \kappa^2 - (m \pi / {\epsilon})^2 - ((n+1/2) \pi / d)^2}, \nonumber\\
&\phi_{m, n} (\boldsymbol x) =\sqrt {\frac{\alpha_{m, n}}{{\epsilon} d}}
 \cos \left( \frac{m \pi x_1}{ \epsilon} \right) \sin \left(  \frac{(n+1/2)
\pi}{d} (x_2+d)\right),\\
 &\alpha_{m, n} =
 \begin{cases}
 2, \quad & m =0,  n \geq 0,\\
 4, \quad & m \geq 1, n \geq 0.
 \end{cases}\nonumber
 \end{align}

\begin{lemm}\label{ell}
The scattering problem \eqref{fe} is equivalent to the  following boundary
integral equation:
\begin{align}\label{bie}
\int_{\Gamma_{\epsilon}^+} \left(  \left(- \frac{{\rm i}}{2} \right) H_{0}^1
(\kappa |\boldsymbol x- \boldsymbol y|) + g_{\epsilon} (\boldsymbol x,
\boldsymbol y)\right)\frac{\partial u_{\epsilon} (\boldsymbol y)}{ \partial
\nu_{\boldsymbol y}}   {\rm d} s_{\boldsymbol y} +u^{\rm inc}(\boldsymbol x)
+u^{\rm ref} (\boldsymbol x)=0, \quad \boldsymbol x\in \Gamma_{\epsilon}^+.
\end{align}
\end{lemm}

\begin{proof}
 Noting that $ \partial_\nu u^{\rm inc} +\partial_\nu u^{\rm
ref}=0$ on $\Gamma_0$, especially on $\Gamma_{\epsilon}^+$, we have
 \begin{align*}
\partial_\nu u_{\epsilon}^{\rm sc}= \partial_\nu u_{\epsilon} - \partial_\nu
(u^{\rm inc} + u^{\rm ref}) =\partial_\nu u_{\epsilon} \quad \text
{on}~\Gamma_{\epsilon}^+.
\end{align*}
 It follows from Green's identity in $\mathbb R_2^+$ that
 \begin{align}\label{so}
 u_{\epsilon} (\boldsymbol x) &=u^{\rm inc} (\boldsymbol x) +u^{\rm ref}
(\boldsymbol x) + u_{\epsilon}^{\rm sc} (\boldsymbol x) =u^{\rm inc}(\boldsymbol
x) +u^{\rm ref} (\boldsymbol x) +\int_{\Gamma_{\epsilon}^+} G (\boldsymbol x,
\boldsymbol y) \frac{\partial u_{\epsilon}^{\rm sc} (\boldsymbol y)}{ \partial
\nu_{\boldsymbol y}} {\rm d} s_{\boldsymbol y} \nonumber\\
 &=u^{\rm inc}(\boldsymbol x) +u^{\rm ref} (\boldsymbol x)
+\int_{\Gamma_{\epsilon}^+} G (\boldsymbol x, \boldsymbol y) \frac{\partial
u_{\epsilon} (\boldsymbol y)}{ \partial \nu_{\boldsymbol y}} {\rm d}
s_{\boldsymbol y}, \quad \boldsymbol x \in \mathbb R_2^+,
 \end{align}
where the half-space Green function $G$ is given in \eqref{gfhs}. Using Green's
identity  inside the cavity yields
\begin{align}\label{sc}
u_{\epsilon} (\boldsymbol x)=-\int_{\Gamma_{\epsilon}^+}
g^{\rm PMC}_{\epsilon} (\boldsymbol x, \boldsymbol y)\frac{\partial u_{\epsilon}
(\boldsymbol y)}{\partial \nu_{\boldsymbol y}} {\rm d} s_{\boldsymbol y}, \quad
\boldsymbol x \in D_{\epsilon}.
 \end{align}
 By the continuity of the single layer potential, we have  from \eqref{so} and \eqref{sc} that
 \begin{align*}
 u_{\epsilon} (\boldsymbol x)= u^{\rm inc}(\boldsymbol x) +u^{\rm ref}
(\boldsymbol x) +\int_{\Gamma_{\epsilon}^+} \left(- \frac{{\rm i}}{2} \right)
H_{0}^1 (\kappa |\boldsymbol x- \boldsymbol y|) \frac{\partial u_{\epsilon}
(\boldsymbol y)}{\partial \nu_{\boldsymbol y}} {\rm d} s_{\boldsymbol y}, \quad
\boldsymbol x \in \Gamma_{\epsilon}^+
 \end{align*}
and
\begin{align*}
u_{\epsilon} (\boldsymbol x)= -\int_{\Gamma_{\epsilon}^+} g^{\rm PMC}_{\epsilon}
(\boldsymbol x, \boldsymbol y)\frac{\partial u_{\epsilon} (\boldsymbol y)}{
\partial \nu_{\boldsymbol y}} {\rm d} s_{\boldsymbol y}, \quad \boldsymbol
x \in \Gamma_{\epsilon}^+.
\end{align*}
The proof is finished by imposing the continuity of the solution on the open
aperture $\Gamma_{\epsilon}^+$.
\end{proof}

It is clear to note that
\begin{align*}
\partial_\nu u_{\epsilon}\big|_{\Gamma_{\epsilon}^+}=\partial_{x_2}
u_{\epsilon}(x_1, 0), \quad u^{\rm inc} +u^{\rm ref} =2 e^{{\rm i} \kappa \sin
\theta x_1}, \quad x_1\in(0, \epsilon).
\end{align*}
Introduce new variables by rescaling with respect to $\epsilon$: $X=
x_1/{\epsilon}, Y= y_1/{\epsilon}$. Define three functions:
\begin{align}\label{ND1}
&\varphi (Y)= - \partial_{y_2} u_{\epsilon} ({\epsilon} Y, 0),
\nonumber\\
& G_{\epsilon}^{e} (X, Y) =   \left(- \frac{\rm i}{2} \right) H_0^{(1)}
({\epsilon} k |X-Y|), \nonumber\\
&G_{\epsilon}^{i} (X, Y)= g^{\rm PMC}_{\epsilon} ({\epsilon} X, 0; {\epsilon}
Y, 0)=\sum \limits_{m, n=0}^{\infty} \frac{c_{m, n} \alpha_{m,n}} { {\epsilon}
d} \cos (m \pi X) \cos (m \pi Y)
\end{align}
and the boundary integral operators:
\begin{align}
(T^e \varphi) (X)=\int_0^1 G_{\epsilon}^e (X, Y) \varphi (Y) {\rm d} Y, \quad
X\in (0, 1),\nonumber\\
(T^i \varphi ) (X)=\int_0^1 G_{\epsilon}^{i} (X, Y) \varphi (Y) {\rm d} Y,
\quad X\in (0, 1).\label{bio}
\end{align}
Hence the boundary integral equation \eqref{bie} is equivalent to the following
operator equation:
\begin{align}\label{oes}
( T^e+T^i) \varphi = f/{\epsilon},
\end{align}
where $f(X):= (u^{\rm inc}+u^{\rm ref}) ({\epsilon} X, 0) =2 e^{{\rm i} \kappa
\sin \theta  {\epsilon} X}.$

\subsection{Asymptotics of the integral operators}

In this subsection, we study the  asymptotic properties of the integral
operators in \eqref{bio}. For clarity, we first introduce several notation
below.
\begin{align*}
\Gamma_1 (\kappa, {\epsilon})&=\frac{1}{\pi} (\ln \kappa + \gamma_1) +\frac{1}{
\pi} \ln {\epsilon}, \quad \Gamma_2 (\kappa, {\epsilon})=-\frac{\tan \kappa d}{
{\epsilon} \kappa} +\frac{2 \ln 2}{ \pi},\\
R_1^{\epsilon} (X, Y)&= - \frac{\kappa^2 }{4 \pi} |X-Y|^2 {\epsilon}^2  \ln {\epsilon}, \\
R_2 ^{\epsilon} (X,Y) &=\left(\gamma_2 - \frac{\ln (\kappa |X-Y|)}{4 \pi}
\right)\kappa^2 |X-Y|^2  \epsilon^2+o ((\kappa {\epsilon} |X-Y|)^2),\\
R_3^{\epsilon} (X, Y)&= - \frac{{\epsilon}^2 \kappa^2}{\pi}\bigg(\frac{1}{\pi
^2} \sum\limits_{m=1}^{\infty} \frac{1}{m^3}+ \frac{(X+Y)^2}{4} \ln (\pi (X+Y))
+\frac {(X-Y)^2}{4} \ln (\pi (X-Y))\\
 &\hspace{1cm}+ o ((X+Y)^2 +(X-Y)^2)\bigg),
\end{align*}
where $\gamma_1 =\gamma_0 - \ln 2 - \frac{\pi {\rm i}}{2}$, $\gamma_2=
-\frac{\rm i}{4 \pi} + \frac{\rm i}{8}+ \frac{1}{4 \pi} (\ln 2 -\gamma_0)$, and
$\gamma_0$ is the Euler constant. For the fixed depth of the cavity $d$ and the
wave number $\kappa$, if ${\epsilon}$ small enough, we have the following
asymptotic expansions for the kernels $G_{\epsilon}^e$ and $G_{\epsilon}^i$.

\begin{lemm}\label{AE}
If ${\epsilon} \ll \lambda$, then the following estimates hold:
\begin{align*}
&G_{\epsilon}^{e}=\Gamma_1 (\kappa, {\epsilon})+\frac{1}{\pi} \ln
|X-Y|+R_1^{\epsilon} (X, Y)+R_2 ^{\epsilon} (X,Y),\\
&G_{\epsilon}^i = \Gamma_2 (\kappa, {\epsilon})+\frac{1}{\pi } \left(  \ln
\left|\sin \frac{\pi (X+Y)}{2}\right|+\ln \left|\sin \frac{\pi (X-Y)}{2}\right|
\right) +R_3 ^{\epsilon} (X, Y),
\end{align*}
where
\[
|R_1^{\epsilon} (X, Y)|\leq C_1 {\epsilon^2} |\ln {\epsilon}|,\quad
|R_2^{\epsilon} (X, Y)| \leq C_2 {\epsilon}^2,\quad |R_3 ^{\epsilon} (X, Y)|
\leq C_3 {\epsilon}^2.
\]
Here $C_j, j=1, 2, 3$ are positive constants independent of $\epsilon$.
\end{lemm}

\begin{proof}
It follows from the asymptotic of the Hankel function near zero (cf. \cite{Abramowitz1972})  that
\begin{align*}
G_{\epsilon}^{e} (X, Y) &= \left(- \frac{\rm i}{2} \right) H_0^1 \left(
{\epsilon} \kappa |X-Y| \right)\\
&=\frac{1}{\pi} (\gamma_1 + \ln |{\epsilon}\kappa (X-Y)|)-\frac{1}{4 \pi} |
{\epsilon} k (X-Y)|^2 \ln | {\epsilon} \kappa (X-Y)| \\
&\hspace{1cm}+ \gamma_2 ({\epsilon} \kappa |X-Y|)^2
 +o ((\kappa {\epsilon} |X-Y|)^2)\\
&=\Gamma_1 (\kappa, {\epsilon})+\frac{1}{\pi} \ln |X-Y|+R_1^{\epsilon} (X,
Y)+R_2 ^{\epsilon} (X,Y).
\end{align*}
The integral kernel $ G_{\epsilon}^{i} (X, Y)$ can be expressed as
\begin{align}\label{gw}
G_{\epsilon}^{i} (X, Y)=\frac{1} { {\epsilon} d}  \sum \limits_{m, n=0}^{\infty}
c_{m, n} \alpha_{m,n}\cos (m \pi X) \cos (m \pi Y).
\end{align}
Recalling the definitions of $c_{m, n}$ and $\alpha_{m, n}$ in \eqref{cad},
we set
\[
C_m ({\epsilon}, \kappa) =\sum\limits_{n=0}^{\infty} c_{m, n} \alpha_{m,
n}=\sum\limits_{n=0}^{\infty} \frac{\alpha_{m, n}}{\kappa^2 -\left( \frac{m
\pi}{\epsilon}\right)^2 -\left( \frac{(n +\frac{1}{2}) \pi}{d} \right)^2}.
\]
Using the fact (cf. \cite{Gradshteyn2015}) that
\begin{align*}
\tan \frac{\pi x}{2} =\frac{4 x}{\pi} \sum\limits_{n=0}^{\infty} \frac{1}{ (2
n+1)^2 -x^2}, \quad x\neq n+\frac{1}{2}, n\in\mathbb Z,
\end{align*}
we may check for $m=0$  that
\begin{align*}
C_0 (\kappa)= \sum \limits_{n =0}^{\infty} \frac{2}{ \kappa^2 -\left( \frac{(n
+\frac{1}{2}) \pi}{d} \right)^2 }=- \frac{d\tan (\kappa  d)} {\kappa}.
\end{align*}
Using the identity
\[
\tanh \frac{ \pi x}{ 2} =\frac {4 x} {\pi} \sum \limits_{n=0}^{\infty}\frac{1}{
(2n +1)^2  +x^2},
\]
we have for $m \geq 1$ that
\begin{align*}
C_{m} ( {\epsilon}, \kappa) &=\sum\limits_{n=0}^{\infty} \frac{4}
{\kappa^2-\left( \frac{m \pi}{\epsilon}\right)^2 -\left( \frac{(n +\frac{1}{2})
\pi}{d} \right)^2}\\
&=- \frac{2 d}{\sqrt {\left( \frac{m \pi}{\epsilon} \right)^2 -\kappa^2}} \tanh
\left(d\sqrt {\left( \frac{m \pi}{\epsilon} \right)^2 -\kappa^2}\right)\\
&= - \frac{2d {\epsilon} }{ m \pi}-\frac{{\epsilon}^3  \kappa^2 d}{m^3
\pi^3}+O\left(\frac{-2{\epsilon}^5 \kappa^4d}{ 4 ! m^5 \pi^5}\right).
\end{align*}
Using the following facts (cf. \cite{Kress1999,lewin1975theory}):
\begin{align*}
&\sum\limits_{m=1}^{\infty} \frac{\cos m \pi x}{ m}= -\left( \ln 2 + \ln \left|
\sin \frac{\pi x}{2}\right|\right), \quad 0 < x < 2,\\
&\sum \limits_{m=1}^{\infty} \frac{\cos m \pi x
}{m^3}=\sum\limits_{m=1}^{\infty}\frac{1}{m^3} +\frac{(\pi x)^2}{2} \ln (\pi x)+
O (x^2), \quad 0< x<2
\end{align*}
and substituting $C_0(\kappa), C_m (\epsilon, \kappa) (m \geq 1)$ into
\eqref{gw}, we obtain
\begin{align*}
G_{\epsilon}^{i} (X, Y)=&\frac{1}{ {\epsilon} d} \bigg( C_0
(\kappa)-\sum\limits_{m=1}^{\infty} \frac{2d {\epsilon}} { m \pi} \cos (m \pi X)
\cos (m \pi  Y)\\
&\hspace{1cm}-\sum\limits_{m=1}^{\infty}\frac{{\epsilon}^3  \kappa^2 d}{m^3
\pi^3} \cos (m \pi X) \cos (m \pi  Y) + O(\frac{-2{\epsilon}^5 \kappa^4 d}{ 4 !
m^5 \pi^5})\bigg)\\
= & -\frac{\tan (\kappa d)}{ {\epsilon} \kappa}+\frac{1}{\pi } \left( 2 \ln 2
+\ln \left|\sin \frac{\pi (X+Y)}{2}\right|+\ln \left|\sin \frac{\pi
(X-Y)}{2}\right| \right)\\
&\hspace{1cm}- \frac{{\epsilon}^2
\kappa^2}{\pi}\bigg(\frac{1}{\pi^2}\sum\limits_{m=1}^{\infty} \frac{1}{m^3}+
\frac{(X+Y)^2}{4} \ln (\pi (X+Y)) +\frac {(X-Y)^2}{4} \ln (\pi (X-Y))\\
 &\hspace{1cm}+ O ((X+Y)^2 +(X-Y)^2)\bigg),
\end{align*}
which complete the proof.
\end{proof}

Let
\begin{align}
&\Gamma =\Gamma_1 (\kappa, {\epsilon}) +\Gamma_2 (\kappa, {\epsilon}), \quad
k^1_{\infty} (X, Y)= R_1^{\epsilon} (X, Y), \quad  k^2_{\infty}=R_2^{\epsilon}
(X,Y)+R_3 ^{\epsilon} (X, Y), \nonumber\\
& k (X, Y)= \frac{1}{\pi } \left( \ln |X-Y|+ \ln \left|\sin \frac{\pi
(X+Y)}{2}\right|+\ln \left|\sin \frac{\pi (X-Y)}{2}\right| \right) \label{Gd1}.
\end{align}
Denote by $K, K^1_{\infty}, K^2_{\infty}$ the integral operators corresponding
to the Schwarz kernels $k (X, Y), k^1_{\infty} (X, Y)$ and $k^2_{\infty} (X,
Y)$, respectively.

Let $I$ be a bounded open interval in $\mathbb R$ and define
\begin{align*}
H^s (I):= \left\{ u= U |_{I}, ~ U \in H^s (\mathbb R) \right\},
\end{align*}
where $s\in\mathbb R$. Then $H^s (I)$  is a Hilbert space with the norm
\[
\|u\|_{H^s(I)} = \inf \left\{ \|U\|_{H^s (\mathbb R)} \bigg| U \in
H^s(\mathbb R) ~\text {and}~ U |_{I}= u \right\}.
\]
Define
\[
\tilde H^s (I):=\left\{ u= U |_{I} \bigg| U\in H^s (\mathbb R)~ \text
{and}~ {\rm supp}(U) \subset \bar I\right\}.
\]
It can be shown that the space $\tilde H^s (I)$  is the
dual of $H^{- s} (I)$ and the norm for $\tilde H^{s} (I)$ can be defined
via the duality (cf.  \cite{Adams2003}).

We also define the operator $P:  \tilde H^{-1/2} (0, 1) \rightarrow  H^{1/2} (0,
1)$ by
\[
P\phi (X)= \langle  \phi, 1_{(0,1)}\rangle_{L^2 (0,1)} 1_{(0, 1)},
\]
where the duality between $\tilde H^{-1/2} (0, 1)$ and $H^{1/2} (0, 1)$ is
defined by $\langle \cdot, \cdot\rangle_{L^2(0,1)}$, $1_{(0, 1)}$ is the
function defined on the interval $(0, 1)$ and is equal to one. It is easy to
note that $1_{(0, 1)} \in H^{1/2} (0, 1)$.

\begin{lemm}\label{ope}
The following conclusions hold:

(i) The  operator $T^e + T^{i}$ admits the decomposition:
\begin{align*}
T^e +T^{i} =\Gamma P + K + K^1_{\infty}+ K^2_{\infty}.
\end{align*}
Moreover, $K^1_{\infty}$ and $K^2_{\infty}$ are bounded from $ \tilde H^{-1/2}
(0, 1) $ to $ H^{1/2} (0, 1)$ with the operator norm satisfying
$\|K^1_{\infty}\| \leq {\epsilon}^2 \ln |{\epsilon}|,  \|K^2_{\infty} \| \leq
{\epsilon}^2$  uniformly for wavenumber $\kappa$.

(ii) The  operator $K$ is invertible from $\tilde H^{-1/2} (0, 1)$ to $H^{1/2}
(0, 1)$. Moreover, the constant $q_0:= \langle K^{-1} 1_{(0, 1)}, 1_{(0, 1)}
\rangle_{L^2 (0,1)}\neq 0$.

\end{lemm}

\begin{proof}
The proof of (i) follows directly from the definitions of the operators $T^e$,
$T^i$ in \eqref{bio} and the asymptotic expansions of their kernels given by
Lemma \ref{AE}. The proof of (ii) follows  directly from \cite[Theorem 4.1,
Lemma 4.2]{Bonnetier2010}.
\end{proof}

\subsection{Asymptotics of the resonances}

The scattering resonance of \eqref{fe} is defined as a complex wavenumber
$\kappa$ with negative imaginary part such that there is a nontrivial solution
to \eqref{fe} when the incident field is zero. This is the characteristic value
of the operator $ \Gamma P + K + K^1_{\infty}+ K^2_{\infty}$ with respect to the
variable $\kappa.$ For simplicity, we write
 \begin{align*}
 \Gamma P + K + K^1_{\infty}+ K^2_{\infty} :=\mathscr P +\mathscr L,
 \end{align*}
where $\mathscr P= \Gamma P, \mathscr L= K + K^1_{\infty}+ K^2_{\infty}. $ By
Lemma \ref{ope}, it is easy to see that $\mathscr L$ is invertible for
sufficiently small $\epsilon.$  Now assume that there exists $\varphi_{0}$ such
that
\begin{align*}
(\mathscr P +\mathscr L ) \varphi_0 =0.
\end{align*}
Then
\begin{align*}
(\mathscr L^{-1} \mathscr P  + \mathscr I)\varphi_0 =0,
\end{align*}
where $\mathscr I$ is the identity operator. It is straightforward to check that
the eigenvalues of operator $\mathscr L^{-1} \mathscr P + \mathscr I$ are
\begin{align*}
\lambda (\kappa, {\epsilon}) =1+ \Gamma (\kappa, {\epsilon}) \langle \mathscr
L^{-1} 1_{(0, 1)}, 1_{(0, 1)} \rangle_{L^2(0,1)}.
 \end{align*}
 Therefore, the characteristic values of the operator $\mathscr P
+\mathscr L$ are the roots of the analytic functions $\lambda (\kappa,
{\epsilon})$, and the associated characteristic function is given by
 \begin{align*}
 \varphi_0 =\Gamma (\kappa ,{\epsilon}) \mathscr L^{-1} 1_{(0, 1)}.
 \end{align*}

 \begin{lemm} \label{rt}
The resonance of the scattering problem \eqref{fe} are the roots of the
analytic functions $\lambda (\kappa, {\epsilon})=0.$ Moreover,
 \begin{equation}\label{ep}
 \mathscr L^{-1} 1_{(0, 1)}=K^{-1} 1_{(0, 1)}+O({\epsilon}^2\ln
{\epsilon})+O({\epsilon}^2),\quad
\langle \mathscr L^{-1} 1_{(0, 1)}, 1_{(0, 1)}\rangle_{L^2
(0,1)}=q_0+O({\epsilon}^2\ln {\epsilon})+O({\epsilon}^2).
\end{equation}
 \end{lemm}

 \begin{proof}
For given roots of $\lambda (\kappa, {\epsilon})$, it is easy to check that they
are the characteristic values of the operator $\mathscr P
+\mathscr L$ with corresponding characteristic function defined above.

It follows from the asymptotic expansions in  Lemma \ref{ope}  and the Neumann
series that
\begin{align*}
 \mathscr L^{-1} &=(K+K^1_{\infty}+K^2_{\infty})^{-1}= K^{-1}\left(
\sum\limits_{j=0}^{\infty} (-1)^{j} (K^{-1} (K_{\infty}^{1}+K^{2}_{\infty}))^{j}
\right)\\
&=K^{-1}+O({\epsilon}^2\ln {\epsilon})+O({\epsilon}^2),
\end{align*}
which gives \eqref{ep}.
\end{proof}

 Next, we present the asymptotic expansion of resonances for the
scattering problem \eqref{fe}.
 \begin{theo}\label{th3.5}
The scattering problem \eqref{fe} has a set of resonances $\{ k_{n}\}$, which
satisfy
\begin{align}\label{kns}
k_{n}=\frac{n \pi}{d}+\frac{n\pi}{d^2}\left(\frac{1}{\pi} {\epsilon}\ln
{\epsilon} +\left(\frac{1}{q_0} +\frac{1}{\pi} (2\ln 2 +\ln
\frac{n\pi}{d}+\gamma_1)\right){\epsilon}\right)+O({\epsilon}^2\ln {\epsilon}), \quad n =1, 2, \cdots.
\end{align}
\end{theo}

\begin{proof}
By Lemma \ref{rt}, we consider the root of
\begin{align*}
\lambda (\kappa, {\epsilon})=1+(\Gamma_1(\kappa, {\epsilon})+\Gamma_2 (\kappa,
{\epsilon}))\langle \mathscr L^{-1}1_{(0, 1)}, 1_{(0, 1)}\rangle_{L^2 (0,1)}=0.
\end{align*}
Recall that $ \Gamma_1 (\kappa, {\epsilon})=\frac{1}{\pi} (\ln \kappa +
\gamma_1) +\frac{1}{ \pi} \ln {\epsilon}$ and  $\Gamma_2 (\kappa,
{\epsilon})=-\frac{\tan \kappa d}{ {\epsilon} \kappa} +\frac{2 \ln 2}{ \pi}.$
 The above equation can be written as
 \begin{align*}
 1+\left( - \frac{\tan \kappa d}{ {\epsilon} \kappa} +\frac{1}{\pi} (2\ln 2+\ln
\kappa +\gamma_1)+ \frac{1}{\pi}\ln {\epsilon}\right)\langle \mathscr
L^{-1}1_{(0, 1)}, 1_{(0, 1)}\rangle_{L^2 (0,1)}=0.
 \end{align*}
 Using Lemma \ref{rt}, we obtain
 \begin{align}\label{pk}
 p(\kappa, {\epsilon}):= {\epsilon} \lambda (\kappa, {\epsilon})={\epsilon}
+\left(- \frac{\tan \kappa d}{\kappa }+ {\epsilon}\rho
(\kappa)+\frac{1}{\pi}{\epsilon}\ln {\epsilon} \right) \left( q_0 +O
({\epsilon}^2 \ln {\epsilon}) +O ({\epsilon}^2) \right),
 \end{align}
 where $\rho (\kappa)=\frac{1}{\pi} (2\ln 2+\ln \kappa +\gamma_1).$

Note that the axis $\{ z \in \mathbb C: {\rm arg} z=\pm \frac{\pi}{ 2 d} \}$
is the branch cut for $p(\kappa, \epsilon)$, hence we choose a small number
$\theta_0 >$ 0 and consider the
domain $\{z \in \mathbb C:  -\frac{\pi}{2d} +\theta_0 <{\rm arg} z <
\frac{\pi}{2d} -\theta_0 ~\text{or}~  \frac{\pi}{2d} +\theta_0 <{\rm arg} z <
\frac{3\pi}{2d} -\theta_0  \}$. On the other hand, we are only interested in
those resonances which are not in the high frequency regime. Therefore, we only
need to find all the roots of $p(\kappa, {\epsilon})$ in the domain
\[\Omega_{\theta_0, M}=\left\{z \in \mathbb C: |z| \leq M,~ -\frac{\pi}{2d}
+\theta_0 <{\rm arg} z < \frac{\pi}{2d} -\theta_0 ~\text{or}~  \frac{\pi}{2d}
+\theta_0 <{\rm arg} z < \frac{3\pi}{2d} -\theta_0 \right\}\]
for some fixed number $M>0.$  Since $p (\kappa, {\epsilon})$ blows up as $\kappa
\rightarrow 0 $ or $\kappa \rightarrow  (j+1/2) \pi /d, j\in \mathbb Z$.  As a
result, there exists $\delta_0>0$ such that all the roots of $p$ in
$\Omega_{\theta_0, M}$ actually lies in the smaller domain
\[\Omega_{\delta_0, \theta_0, M}:= \{z \in \mathbb C: |z| \geq \delta_0 ~\text
{or}~ |z- (j+1/2) \pi /d| \geq \delta_0, j \in \mathbb Z\} \cap
\Omega_{\theta_0, M}.\]

It is clear that $- \frac{\tan \kappa d}{  \kappa}$ is analytic in
$\Omega_{\delta_0, \theta_0, M}$ and its roots are given by $k_{n,0}=n\pi/d,
n=1,2,\cdots.$ Note that each root is simple.
Denote by $k_{n}$ the roots of  $\lambda (\kappa, {\epsilon})$.  From Rouche's
theorem, we deduce that $k_n$ are also simple and are close to $k_{n,0}$ if
${\epsilon}$ is sufficiently small. We now derive the leading order asymptotic
terms for these roots. Define
  \begin{align}\label{dfp}
  p_1 (\kappa, {\epsilon})= {\epsilon} +\left(- \frac{\tan \kappa d}{\kappa }+
{\epsilon}\rho (\kappa)+\frac{1}{\pi}{\epsilon}\ln {\epsilon} \right) q_0.
  \end{align}
Expanding $p_1(\kappa, {\epsilon})$ at $k_{n,0}$ yields
\begin{align*}
p_1(\kappa, {\epsilon})&= {\epsilon} + \left(\left(- \frac{\tan \kappa
d}{\kappa }\right)'\bigg|_{\kappa=k_{n0}}(\kappa-k_{n0})+{\epsilon}\rho(k_{n0})+
{\epsilon}\rho'(k_{n0})(\kappa-k_{n0})+O(\kappa-k_{n0})^2+\frac{1}{\pi}{\epsilon
} \ln {\epsilon} \right) q_0\\
  &={\epsilon} +\bigg(\frac{-d^2}{n\pi} (\kappa-k_{n0})+\frac{\epsilon}{\pi}
\left(2\ln 2 +\ln\frac{n\pi}{d}+\gamma_1 \right)+ \frac{{\epsilon}
d}{n\pi^2}(\kappa-k_{n0})  +O (\kappa-k_{n0})^2+\frac{1}{\pi}{\epsilon}\ln
{\epsilon}\bigg)q_0.
  \end{align*}
We can conclude that $p_1(\kappa, {\epsilon})$ has simple roots in
$\Omega_{\delta_0, \theta_0, M}$ which are close to $k_{n0}$'s. Moreover, these
roots are analytic with respect to the variables ${\epsilon}$ and  ${\epsilon}
\ln{\epsilon}.$  Denoting the roots of $p_1$ by $k_{n1}$ and expanding them in
term of ${\epsilon}$ and ${\epsilon} \ln {\epsilon},$ we obtain
\begin{align}\label{K1}
k_{n1}=&k_{n0}+\frac{k_{n0}}{d}\left(\frac{1}{\pi} {\epsilon}\ln {\epsilon}
+\left(\frac{1}{q_0} +\frac{1}{\pi} (2\ln 2
+\ln \frac{n\pi}{d}+\gamma_1)\right){\epsilon}\right)
+O({\epsilon}^2\ln {\epsilon}).
\end{align}

Using Rouche's theorem, we claim that $k_{n1}$ gives the leading order
for the asymptotic expansion of the roots $k_n$. More precisely, we have
 \begin{align}\label{cl}
 k_n=k_{n1}+O({\epsilon}^2\ln {\epsilon}).
\end{align}
In the following, we give a proof of the claim. Note that
\begin{align*}
p(\kappa, {\epsilon})-p_1(\kappa, {\epsilon})=\left( -\frac{\tan \kappa
d}{\kappa}+\frac{1}{\pi}{\epsilon} \ln {\epsilon}+\rho(\kappa) {\epsilon}\right)
O({\epsilon}^2 \ln {\epsilon})
\end{align*}
and
\begin{align*}
p_{1}(\kappa, {\epsilon})= -\frac{\tan \kappa d}{\kappa} q_0 +\frac{q_0}{\pi}
{\epsilon} \ln {\epsilon}+(1+\rho(\kappa) q_0) {\epsilon}.
\end{align*}
Hence there exists a  constant $C_n$ such that
\begin{align*}
|p(\kappa, {\epsilon}) -p_1(\kappa, {\epsilon})|<|p_1(\kappa,
{\epsilon})|,\quad \forall \kappa\in \left\{
|\kappa-k_{n1}| =C_n {\epsilon}^2 \ln {\epsilon} \right\}.
\end{align*}
which implies the claim \eqref{cl} by  Rouche's theorem.
\end{proof}

\subsection{The field enhancement with resonance}

First we give an asymptotic expansion of $p(\kappa, {\epsilon})$ at the
resonant frequencies.

\begin{lemm}\label{lm3.6}
At the resonant frequencies $\kappa={\rm Re} k_n$ where $k_n$ is given in
\eqref{kns}, we have
\begin{align*}
p(\kappa, {\epsilon})=-\frac{{\rm i}  q_0 }{2}{\epsilon}+O({\epsilon}^2 \ln ^2{\epsilon}).
\end{align*}
\end{lemm}

\begin{proof}
Assume that $|\kappa-{\rm Re}k_n|\leq {\epsilon} \ln {\epsilon}.$
It follows from the definition of $p_1$ in \eqref{dfp} that
\begin{align*}
p(\kappa, {\epsilon})
&=p_1(\kappa, {\epsilon})+O({\epsilon}^2 \ln {\epsilon})\\
&=p_1'(k_{n})(\kappa-k_n)+O(\kappa-k_n)^2 +O({\epsilon}^2 \ln {\epsilon})\\
&=\left( \frac{-d}{k_{n0}}+\frac{\epsilon}{\pi k_{n0}}+O(|k_n-k_{n0}|)\right)
q_0 (\kappa-k_{n})+O({\epsilon}^2 \ln ^2{\epsilon})\\
&=\left( \frac{-d}{k_{n0}}+\frac{\epsilon}{\pi k_{n0}}\right)q_0
(\kappa-{\rm Re}k_n -{\rm i}{\rm Im}k_n)+O({\epsilon}^2 \ln ^2{\epsilon}).
\end{align*}
Recalling  $\gamma_1 =\gamma_0 - \ln 2 - \frac{\pi {\rm i}}{2}$, we have from
the expansion of $k_{n1}$ in \eqref{K1} that
\begin{align*}
{\rm Im} k_n={\rm Im} k_{n1}+O({\epsilon}^2 \ln {\epsilon})=\frac{k_{n0}}{d\pi}
({\rm Im}\gamma_1) {\epsilon}=-\frac{k_{n0} {\epsilon} }{2d}.
\end{align*}
When $\kappa={\rm Re}k_n$, we have
\begin{align*}
p(\kappa, {\epsilon})={\rm i} q_0 \left( -\frac{ 1 }{2}{\epsilon}+\frac{1}{2 d}
{\epsilon}^2\right)+O({\epsilon}^2 \ln ^2{\epsilon})=-\frac{{\rm i}  q_0
}{2}{\epsilon}+O({\epsilon}^2 \ln ^2{\epsilon}),
\end{align*}
which completes the proof.
\end{proof}

\begin{lemm}\label{lm3.7}
The solution $\varphi$ of equation \eqref{oes} has the following asymptotic
expansion in $\tilde H^{-1/2} (0, 1)$:
\begin{align*}
\varphi=K^{-1} 1_{(0, 1)} \cdot  \frac{\left( 2 + \sin \theta  \cdot O(\kappa {\epsilon})
 \right)}{p}+O ({\epsilon}^2 \ln {\epsilon}).
\end{align*}
Moreover,
\begin{align*}
\langle \varphi, 1_{(0, 1)} \rangle_{L^2 (0,1)}=\frac{ q_0\left( 2 + \sin \theta
 \cdot O(\kappa {\epsilon})  \right)}{p}+O ({\epsilon}^2 \ln {\epsilon}).
\end{align*}
\end{lemm}

\begin{proof}
It follows from Lemma \ref{ope} that equation \eqref{oes} is equivalent to
\begin{align*}
(\mathscr L^{-1} \mathscr P  +\mathscr I)\varphi =\mathscr L^{-1}
(f/{\epsilon}).
\end{align*}
Recall that  the operator $\mathscr L^{-1} \mathscr P +\mathscr I$ has the
eigenvalue $\lambda (\kappa, {\epsilon}).$ Thus
\begin{align*}
 \langle \varphi, 1_{(0,1)} \rangle_{L^2 (0, 1)} =\frac{1}{\lambda (\kappa,
{\epsilon})} \langle  \mathscr L^{-1} (f/{\epsilon}), 1_{(0,1)}\rangle_{L^2
(0,1)},
\end{align*}
and
\begin{align*}
\varphi&=  \mathscr L^{-1} (f/{\epsilon})- \mathscr L^{-1} \Gamma  P \varphi =
\mathscr L^{-1} (f/{\epsilon})-\mathscr L^{-1} \Gamma \langle
\varphi,1_{(0,1)}\rangle 1_{(0, 1)}\\
&=\mathscr L^{-1} (f/{\epsilon})-\mathscr L^{-1} 1_{(0,1)} \frac{\Gamma}{\lambda
(\kappa, {\epsilon})}\langle  \mathscr L^{-1}
(f/{\epsilon}), 1_{(0,1)}\rangle_{L^2 (0,1)}.
\end{align*}
Recall that $\lambda (\kappa, {\epsilon})= 1+\Gamma \langle \mathscr L^{-1}
1_{(0, 1)}, 1_{(0,1)}\rangle_{L^2 (0,1)}$. A simple calculation yields
\begin{align*}
\varphi=\mathscr L^{-1} (f/{\epsilon})-\mathscr L^{-1} 1_{(0,1)} \frac{(\lambda
(\kappa, {\epsilon})-1)/\langle \mathscr L^{-1} 1_{(0, 1)},
1_{(0,1)}\rangle_{L^2 (0,1)}}{\lambda (\kappa, {\epsilon})}\langle  \mathscr
L^{-1} (f/{\epsilon}), 1_{(0,1)}\rangle_{L^2(0,1)}.
\end{align*}
Using the definition of $ f$  and expanding it with respect to $\epsilon$, we
get
\begin{align*}
\frac{f}{\epsilon}= \frac{2 e^{{\rm i}  \kappa {\epsilon}X  \sin \theta
}}{\epsilon}=\frac{2}{\epsilon}+{\rm i} \kappa  X\sin  \theta + O(\kappa^2
{\epsilon}),
\end{align*}
which gives
\begin{align*}
\mathscr L^{-1} (f/{\epsilon})=\frac{1}{\epsilon} \left( 2 + \sin \theta \cdot O
(\kappa \epsilon)  \right) \left( K^{-1} 1_{(0,1)} + O({\epsilon}^2 \ln
{\epsilon})\right).
\end{align*}
Combining the above equation with  \eqref{ep}, we obtain
\begin{align*}
{\epsilon} \varphi = &\left( 2 +  \sin\theta \cdot  O (\kappa \epsilon)\right)
\left( K^{-1} 1_{(0,1)} + O ({\epsilon}^2 \ln {\epsilon})\right)\\
&\hspace{1cm} +\frac{(1-\lambda)/(q_0+O ({\epsilon}^2 \ln {\epsilon}))}{\lambda}
 \cdot \left(  2 q_0 + q_0 \sin \theta \cdot O( \kappa {\epsilon})  \right)\cdot
\left( K^{-1} 1_{(0,1)} + O ({\epsilon}^2 \ln {\epsilon})\right)\\
=& \frac{\left( 2 + \sin\theta \cdot O(\kappa {\epsilon})\right)}{\lambda}
\left( K^{-1} 1_{(0,1)} + O ({\epsilon}^2 \ln {\epsilon})\right).
\end{align*}
Thus
\begin{align*}
&\varphi=K^{-1} 1_{(0,1)} \cdot  \frac{\left( 2 +  \sin \theta  \cdot O(\kappa
{\epsilon})   \right)}{p}+O ({\epsilon}^2 \ln {\epsilon}), \\
&\langle \varphi, 1_{(0,1)} \rangle_{L^2 (0,1)}=\frac{ q_0\left( 2 + \sin \theta
 \cdot O(\kappa {\epsilon})  \right)}{p}+O ({\epsilon}^2 \ln {\epsilon}),
\end{align*}
which completes the proof.
\end{proof}

Combining  Lemmas \ref{lm3.6} and Lemma  \ref{lm3.7}, we obtain
at the resonant frequencies  $ \kappa={\rm Re} k_n$ that
\begin{align}\label{af}
 \varphi= \frac{f_1}{\epsilon} +O ({\epsilon}^2 \ln {\epsilon}), \quad \langle
\varphi, 1_{(0,1)} \rangle_{L^2 (0,1)} =\frac{c_2}{\epsilon} +O ({\epsilon}^2
\ln {\epsilon}),
\end{align}
where
\[
f_1= K^{-1} 1_{(0,1)} \cdot  \frac{ 2 + \sin \theta  \cdot
O(\kappa {\epsilon}) }{-\frac{\rm i}{2} q_0 +O({\epsilon} \ln {\epsilon})} \in
\tilde H^{-1/2} (0,1),\quad
c_2=\frac{ q_0\left( 2 + \sin \theta  \cdot O(\kappa {\epsilon})
\right)}{-\frac{\rm i}{2} q_0 +O({\epsilon} \ln {\epsilon})}
\]

\subsubsection{Enhancement in the far field}\label{sub3.4}

We first investigate the scattered field in the domain $\mathbb R_2^+\setminus
\bar{B}_{R}^+$. Recall that
\begin{align*}
u_{\epsilon}^{\rm sc}(\boldsymbol x)=\int_{\Gamma_{\epsilon}^+} G(\boldsymbol x,
\boldsymbol y) \frac{\partial u_{\epsilon}^{\rm sc} (\boldsymbol y) }{\partial
\nu_{\boldsymbol y}} {\rm d} s_{\boldsymbol y}=\int_{\Gamma_{\epsilon}^+}
G(\boldsymbol x, \boldsymbol y) \frac{\partial u_{\epsilon} (\boldsymbol y)
}{\partial \nu_{\boldsymbol y}} {\rm d} s_{\boldsymbol y}, \quad \boldsymbol x
\in \mathbb R_2^+.
\end{align*}
Here we use the fact $\partial_\nu u^{\rm inc}
+\partial_\nu u^{\rm ref}=0$ on $\Gamma_0$, especially on
$\Gamma_{\epsilon}^+$. Using
\begin{align*}
\frac{\partial u_{\epsilon} (\boldsymbol y) }{\partial \nu_{\boldsymbol y}}
\bigg|_{\Gamma_{\epsilon}^+}=\frac{\partial
u_{\epsilon} (y_1, 0)}{\partial {y_2}}=-\varphi
\left(\frac{y_1}{\epsilon}\right),
\end{align*}
we have
\begin{align*}
u_{\epsilon}^{\rm sc}(\boldsymbol x)=- \int_{\Gamma_{\epsilon}^+}  G(\boldsymbol
x, (y_1, 0)) \varphi \left(\frac{y_1}{\epsilon}\right) {\rm
d}y_1=-{\epsilon}\int_{0}^1 G(\boldsymbol x, ({\epsilon} Y, 0))\varphi (Y) {\rm
d} Y.
\end{align*}
Since
\begin{align*}
G(\boldsymbol x, ({\epsilon} Y, 0))= G(\boldsymbol x, (0, 0))
(1+O({\epsilon})),\quad \boldsymbol x\in \mathbb R_2^+\setminus \bar{B}_R^+,
\end{align*}
we get from Lemma \ref{lm3.7} that
\begin{align*}
\int_0^1 \varphi (Y) {\rm d} Y=\langle \varphi, 1_{(0,1)} \rangle_{L^2
(0,1)}=\frac{ q_0\left( 2 + \sin \theta  \cdot O(\kappa {\epsilon})
 \right)}{p}+O ({\epsilon}^2 \ln {\epsilon}).
\end{align*}
Hence
\begin{align*}
u_{\epsilon}^{\rm sc}(\boldsymbol x) &=- {\epsilon} G(\boldsymbol x, (0, 0))
(1+O({\epsilon})) \int_0^1 \varphi (Y) {\rm d} Y\\
&=- {\epsilon} G(\boldsymbol x, (0, 0)) \cdot  \frac{ q_0\left( 2 + \sin \theta
\cdot O(\kappa {\epsilon})  \right)}{p} +O ({\epsilon}^2 \ln {\epsilon}).
\end{align*}
In the resonance case when $\kappa={\rm Re} k_n$, the enhancement comes from
the term $\frac{1}{p}.$ It follows from Lemma \ref{lm3.6} that
\begin{align*}
\frac{1}{p}=\frac{2 {\rm i}} {q_0 {\epsilon}} (1+O ({\epsilon} \ln^2 {\epsilon})).
\end{align*}
Correspondingly, we have
\begin{align*}
u_{\epsilon}^{\rm sc}(\boldsymbol x)=-2 {\rm i}  G(\boldsymbol x, (0, 0)) \cdot
\left( 2 + \sin \theta  \cdot O(\kappa {\epsilon})   \right)+O ({\epsilon} \ln^2
{\epsilon}),
\end{align*}
which shows that the enhancement of the scattered magnetic field in the far
field region has an order $O(1/{\epsilon})$ compared to the non-resonance case
when $\kappa\neq {\rm Re} k_n$.

It follows from the Ampere law $\nabla \times \boldsymbol
H_{\epsilon}=-{\rm i} \omega \varepsilon \boldsymbol E_{\epsilon}$ and
$\boldsymbol H_{\epsilon}=(0, 0, u_{\epsilon})$ that the enhancement of the
scattered electric field also has an order $O (1/{\epsilon})$ in the far field
region when $\kappa={\rm Re} k_n.$

\subsubsection {Enhancement in the cavity}

Now we study the field enhancement in the cavity. It follows from
\eqref{wg} that the total  field  $u_{\epsilon}$ may be expanded
as the sum of waveguide modes:
\begin{align}\label{wgm}
u_{\epsilon} (\boldsymbol x) =&\frac{1}{\sqrt{\epsilon}} \left( \alpha_0^+
e^{-{\rm i} \kappa x_2} +\alpha^{-}_0 e^{{\rm i} \kappa
(x_2+d)}\right)+\sum\limits_{m\geq 1}\sqrt{\frac{2}{\epsilon}} \alpha_m^{+} \cos
\left( \frac{m \pi x_1}{\epsilon}\right)e^{-{\rm i}\beta_m x_2}\nonumber\\
&\hspace{1cm}+\sum\limits_{m\geq 1}\sqrt{\frac{2}{\epsilon}} \alpha_m^{-} \cos
\left( \frac{m \pi x_1}{\epsilon}\right)e^{{\rm i}\beta_m (x_2+d)},
\end{align}
where $\beta_m={\rm i} \sqrt {(m \pi/ {\epsilon})^2 -\kappa^2}.$

\begin{lemm}\label{cep}
The coefficients in \eqref{wgm} have the following expansions:
\begin{align}\label{AL1}
\frac{\alpha_0^+}{\sqrt {\epsilon}}=-\frac{q_0 (2+ \sin \theta  \cdot O(\kappa
{\epsilon}) )e^{-{\rm i} \kappa d}}{ 2{\rm i} \kappa  \cos (\kappa d)
p}+O({\epsilon}^2\ln {\epsilon}),\quad \frac{\alpha_0^-}{\sqrt {\epsilon}}=
\frac{q_0 (2+ \sin \theta  \cdot O(\kappa {\epsilon}) )}{ 2{\rm i} \kappa  \cos
(\kappa d) p}+O({\epsilon}^2\ln {\epsilon}),
\end{align}
and
\begin{align*}
\sqrt{\frac{2}{\epsilon}}|\alpha_m^{+}| \leq  \frac{C}{\sqrt m},\quad
\sqrt{\frac{2}{\epsilon}}|\alpha_m^{-}| \leq  \frac{C}{\sqrt m},
\end{align*}
where the positive  constant $C$ independents  of  ${\epsilon}, \kappa$ and $m$.
\end{lemm}

\begin{proof}
Recall that $u_{\epsilon}=0$ on $\Gamma_{\epsilon}^-$ then   equation
\eqref{wgm} becomes
\begin{align}\label{ce2}
u_{\epsilon} (\boldsymbol x) =&\frac{\alpha_0^-}{\sqrt {\epsilon}}\left( -
e^{-{\rm i} \kappa (x_2 + d)}+e^{{\rm i} \kappa (x_2+d)}\right)\nonumber\\
 &\hspace{1cm}+\sum\limits_{m\geq 1} \sqrt {\frac {2}{\epsilon}} \alpha_m^{-}
\left( - e^{-{\rm i} \beta_m (x_2+d)} + e^{{\rm i} \beta_m (x_2+d)}\right)\cos
\left( \frac{m \pi x_1}{\epsilon}\right).
\end{align}
Taking the derivative of \eqref{ce2} with respect to $x_2$ yields
\begin{align*}
\frac{\partial u_{\epsilon}(\boldsymbol x)}{\partial x_2}=
\frac{\alpha_0^-}{\sqrt {\epsilon}}{\rm i}\kappa \left(  e^{-{\rm i} \kappa
(x_2 + d)}+e^{{\rm i} \kappa  (x_2+d)}\right)+\sum\limits_{m\geq 1}
\sqrt {\frac {2}{\epsilon}} \alpha_m^{-} {\rm i} \beta_m\left(  e^{-{\rm i}
\beta_m (x_2+d)} + e^{{\rm i} \beta_m (x_2+d)}\right)\cos \left( \frac{m \pi
x_1}{\epsilon}\right),
\end{align*}
which gives
\begin{align}\label{al2}
\frac{\partial u_{\epsilon}}{\partial x_2}(x_1, 0)= \frac{\alpha_0^-}{\sqrt
{\epsilon}}2{\rm i}\kappa \cos (\kappa d) +\sum\limits_{m\geq 1}\sqrt {\frac
{2}{\epsilon}} \alpha_m^{-} 2{\rm i} \beta_m \cos (\beta_m d)\cos \left( \frac{m
\pi x_1}{\epsilon}\right).
\end{align}
Multiplying \eqref{al2} with $\phi_0(x_1)$ and    integrating it  over
$\Gamma^+_{\epsilon}$, we get
\begin{align*}
\frac{\alpha_0^-}{\sqrt {\epsilon}}{\rm i}\kappa 2 \cos (\kappa
d) &=\frac{1}{\epsilon}\int_{\Gamma_{\epsilon}^+} \frac{\partial
u_{\epsilon}}{\partial x_2}(x_1, 0) {\rm d}
x_1 =\frac{1}{\epsilon}\int_{\Gamma_{\epsilon}^+} -\varphi
\left(\frac{x_1}{\epsilon}\right) {\rm d}x_1\\
&=\int_0^1 \varphi (X) {\rm d} X =\langle \varphi, 1_{(0,1)} \rangle_{L^2
(0,1)}\\
&=\frac{ q_0\left( 2 + \sin \theta  \cdot O(\kappa {\epsilon})  \right)}{p}+O
({\epsilon}^2 \ln {\epsilon}),
\end{align*}
which gives  the formulas for $\alpha_0^+$ and $\alpha_0^-$ in \eqref{AL1}.

For $m\geq 1$, it follows from multiplying\eqref{al2}  with $\cos \left(
\frac{m \pi x_1}{\epsilon}\right)$ and  integrating over $\Gamma^+_{\epsilon}$
that
\begin{align*}
\sqrt{\frac{2}{\epsilon}}\alpha_m^{-} {\rm i} \beta_m \cos (\beta_m d)
&=\frac{1}{\epsilon}\int_{\Gamma_{\epsilon}^+} \frac{\partial
u_{\epsilon}}{\partial x_2}(x_1, 0) \cos \left( \frac{m \pi
x_1}{\epsilon}\right) {\rm d}x_1\\
&=-\int_0^1 \varphi(X) \cos (m \pi X) {\rm d} X.
\end{align*}
Note that ${\rm i} \beta_m =O \left( \frac{m}{\epsilon} \right)$ for $m \geq 1$,
and by \eqref{af} that $\|\varphi\|_{H^{-1/2} (0,1)} \lesssim
\frac{1}{\epsilon}, \|\cos (m \pi X)\|_{H^{1/2}(0,1)} \lesssim \sqrt m.$
Thus, we get
\begin{align*}
\sqrt{\frac{2}{\epsilon}}|\alpha_m^{-}| \leq  \frac{C}{\sqrt m},
\end{align*}
which completes the proof.
\end{proof}

In the following, we give the field enhancement inside the cavity.

\begin{theo}
Denote by $\mathring{D}_{\epsilon}=\left\{ \boldsymbol x \in D_{\epsilon} \big|
-d +{\epsilon} \leq  x_2 \leq -{\epsilon} \right\}$ the interior of the
cavity. If $\kappa= {\rm Re} k_n$ where $k_n$ is given in \eqref{kns}, then we
have for $\epsilon \ll d$ that
\begin{align*}
u_{\epsilon} (\boldsymbol x)=\left( \frac{2}{\epsilon}+  \sin \theta  \cdot O(1)
 + O(\ln ^2{\epsilon}) \right) \frac{2 {\rm i} \sin \kappa (x_2 +d)}{\kappa \cos
(\kappa d)} +O ({\epsilon}^2 \ln {\epsilon}), \quad\boldsymbol x\in
\mathring{D}_{\epsilon}.
\end{align*}
Moreover, the electric and magnetic field enhancements are of an order
$O(1/{\epsilon})$.
\end{theo}

\begin{proof}
By the expansion \eqref{wgm} and Lemma \ref{cep}, it is clear that in the region
$\mathring{D}_{\epsilon}$
\begin{align*}
u_{\epsilon} (\boldsymbol x) &=\frac{q_0 (2+ \sin \theta  \cdot O(\kappa
{\epsilon}) )}{ 2 {\rm i} \kappa  \cos (\kappa d) p} \left( -e^{{\rm -i} \kappa
(x_2+d)} +e^{{\rm i} \kappa (x_2+d)} \right) +O ({\epsilon}^2 \ln {\epsilon})+O
(e^{- 1/{\epsilon}})\\
&=\frac{q_0  (2+  \sin \theta  \cdot O(\kappa {\epsilon}) )\sin \kappa (x_2+d)}
{\kappa \cos \kappa d} \frac{1}{p}+O ({\epsilon}^2 \ln{\epsilon})+O (e^{-
1/{\epsilon}}).
\end{align*}
Noting when $\kappa={\rm Re} k_n,$ we have
\begin{align*}
\frac{1}{p}=\frac{2 {\rm i}} {q_0 {\epsilon}} (1+O ({\epsilon} \ln^2
{\epsilon})),
\end{align*}
which yields that
\begin{align*}
u_{\epsilon} (\boldsymbol x)= \left( \frac{2}{\epsilon}+  \sin \theta  \cdot O(1
)  + O(\ln ^2{\epsilon}) \right) \frac{2 {\rm i} \sin \kappa (x_2 +d)}{\kappa
\cos (\kappa d)} +O ({\epsilon}^2 \ln {\epsilon}).
\end{align*}
Therefore, the magnetic field enhancement due to the resonance is of an order $O
(1/{\epsilon})$ in $\mathring{D}_{\epsilon}$. We conclude from Ampere's law that
the electric field enhancement is also of an order $O(1/{\epsilon})$.
\end{proof}

\subsubsection {Field enhancement on the open aperture}\label{sub3.4.3}

Now we present the enhancement of the scattered field on the open aperture of
the cavity. Recall \eqref{so} for the scattered field on $\Gamma_{\epsilon}^+$:
\begin{align*}
u_{\epsilon}^{\rm sc} (x_1, 0)=\int_{\Gamma_{\epsilon}^+} G ((x_1, 0),
\boldsymbol y) \frac{\partial u_{\epsilon} (\boldsymbol y)}{ \partial
\nu_{\boldsymbol y}} {\rm d} s_{\boldsymbol y}.
\end{align*}
Since $x_1= {\epsilon} X, y_1= {\epsilon} Y, \varphi=- \partial_\nu
u_{\epsilon}\big|_{\Gamma_{\epsilon}^+}$, we have
\begin{align*}
u_{\epsilon}^{\rm sc} (x_1, 0)=-\int_0^1 G_{\epsilon}^e (X, Y) {\epsilon}
\varphi (Y) {\rm d }Y.
\end{align*}
It follows from the asymptotic expansion of $G_{\epsilon}^e$ in Lemma \ref{AE}
and Lemma \ref{lm3.7} that
\begin{align*}
u_{\epsilon}^{\rm sc} (x_1, 0)=& -{\epsilon}\int_0^{1} \left( \Gamma_1
+\frac{1}{\pi} \ln |X-Y| +O({\epsilon}^2 \ln {\epsilon})\right) \varphi (Y) {\rm
d} Y\\
=&-{\epsilon} \Gamma_1 \langle \varphi, 1_{(0,1)} \rangle_{L^2 (0,1)} -
{\epsilon}\frac {2+ \sin \theta  \cdot O(\kappa {\epsilon}) }{p}  \frac{1}{\pi}
h_2(X)+O({\epsilon}^2 \ln {\epsilon})\\
=&-{\epsilon} \left(\frac{1}{\pi} (\ln \kappa +\gamma_1)+\frac{1}{\pi} \ln
{\epsilon} \right) \left( \frac{q_0 (2+ \sin \theta  \cdot O(\kappa {\epsilon})
)}{p} +O ({\epsilon}^2 \ln {\epsilon}) \right)\\
&\hspace{1cm}- {\epsilon}\frac {2+ \sin\theta  \cdot O(\kappa {\epsilon}) }{p}
\frac{1}{\pi}h_2(X) +O({\epsilon}^2 \ln {\epsilon})\\
=&-\frac{1}{\pi}\frac{2 q_0}{p} {\epsilon} \ln {\epsilon}-\frac{1}{\pi}
\frac{2}{p} \left( q_0(\ln \kappa + \gamma_1)+h_2(X) \right) {\epsilon}
-\frac{1}{\pi} \left( q_0(\ln \kappa +\gamma_1) + h_2(X)   \right) \frac{
\sin \theta  \cdot O(1) }{p} \epsilon^2\\
&\hspace{1cm}-\frac{1}{\pi}\frac{ q_0  \sin \theta
\cdot O(1) }{p} \epsilon^2\ln{\epsilon}+O({\epsilon}^2 \ln {\epsilon}).
\end{align*}
where $h_2(X):=\int_0^1 K^{-1} 1_{(0, 1)} (Y) \ln |X-Y| {\rm d} Y$. Note
that at the resonant frequencies $\kappa ={\rm Re} k_n,$
\begin{align*}
\frac{1}{p}=\frac{2 {\rm i}} {q_0 {\epsilon}} (1+O ({\epsilon} \ln^2 {\epsilon})).
\end{align*}
Thus we have
\begin{align*}
u_{\epsilon}^{\rm sc} (x_1, 0)= -\frac{ 4 {\rm i}}{\pi} \ln {\epsilon} - \frac{4
{\rm i}}{\pi} \left((\ln \kappa + \gamma_1) +\frac{
h_2(x_1/{\epsilon})}{q_0}\right)+O ({\epsilon} \ln^3 {\epsilon}).
\end{align*}
It is clear to note that the leading order of the resonant mode is a constant
on the open aperture with an order of $O (\ln {\epsilon})$ and the magnetic  field
enhancement is of an order $O(1/{\epsilon})$.

On the other hand, since
\[
\partial_{x_1}  u_{\epsilon}^{\rm sc}(x_1, 0)=-\frac{2 +\sin\theta \cdot O(k
\epsilon)}{p} \frac{1}{\pi }h_2'(x_1/\epsilon) +O(\epsilon^2 \ln
\epsilon),
\]
which implies that  at the resonant frequencies $\kappa ={\rm Re} k_n$ , on the
top aperture $\Gamma_{\epsilon}^+$, the scatter field satisfies
\begin{align*}
\partial_{x_1} u^{\rm sc} (x_1, 0)=- \frac{4 {\rm i} (1+O (\epsilon \ln
^2\epsilon))}{\pi q_0 \epsilon }h' (x_1/ \epsilon)  -\frac{2 {\rm i} \sin \theta
\cdot O(1)}{\pi q_0} h_2'(x_1/\epsilon)+O (\epsilon \ln^2 \epsilon).
\end{align*}
By Ampere's law, we know the electric field enhancement is also of an
order $O (1/\epsilon)$ on the top aperture $\Gamma_{\epsilon}^+$ when $\kappa
={\rm Re}  k_n.$

\subsection{Numerical experiments}

In this section, we present some numerical experiments to verify the theoretical
studies of the field enhancement for the PEC-PMC cavity. As a representative
example, we take the cavity width $\epsilon=0.005$, the cavity depth $d=1$, and
the angle of incidence $\theta=\pi/3$ in the numerics. Figure \ref{PECPMC}
(left) and (right) show the plot of the electric field enhancement
$Q_{\boldsymbol E}$ the the magnetic field enhancement factor $Q_{\boldsymbol
H}$ against the wavenumber $\kappa$, respectively. We observe that both the
electric and magnetic enhancement factors do reach peaks at the resonant
frequencies \eqref{kns}, which are close to $\kappa=n\pi, n=1, 2, \dots$.
Besides the peaks at the resonant frequencies, the electric field enhancement
factor $Q_{\boldsymbol E}$ also displays a peak when the wavenumber $\kappa$ is
sufficiently small, i.e., when the wavelength of the incident field $\lambda$
is sufficiently large. In this situation, the scales satisfy $\epsilon\ll
d\ll \lambda$. The electric field should have an enormous enhancement according
to our theoretical result in Theorem \ref{NTH}. In contrast, the magnetic
field enhancement factor $Q_{\boldsymbol H}$ does not show a peak when
the wavenumber approaches zero, which is also confirmed in Theorem \ref{NTH}
that the magnetic field has no enhancement when  $\epsilon\ll
d\ll \lambda$.

\begin{figure}
\centering
\includegraphics[width=0.45\textwidth]{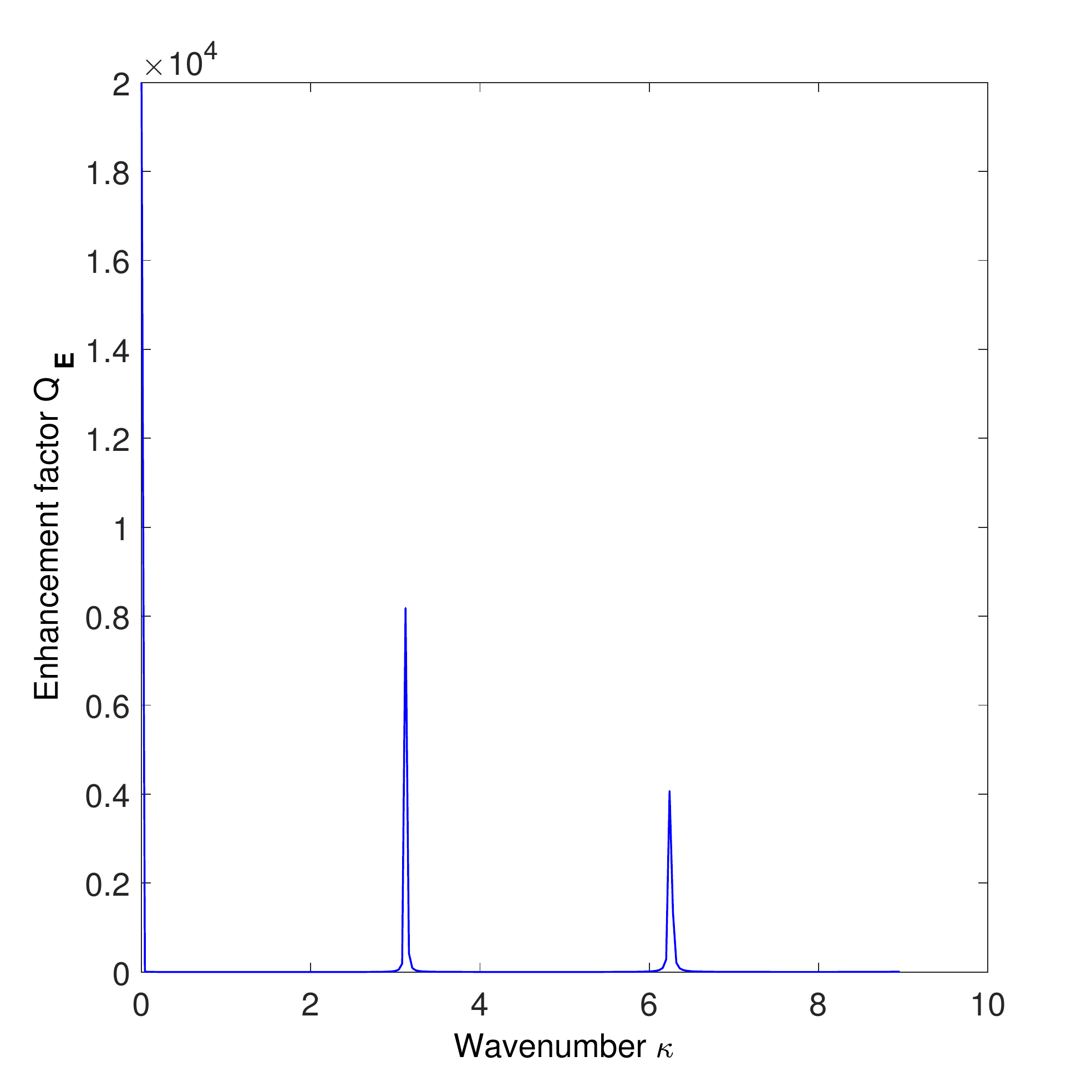}
\includegraphics[width=0.45\textwidth]{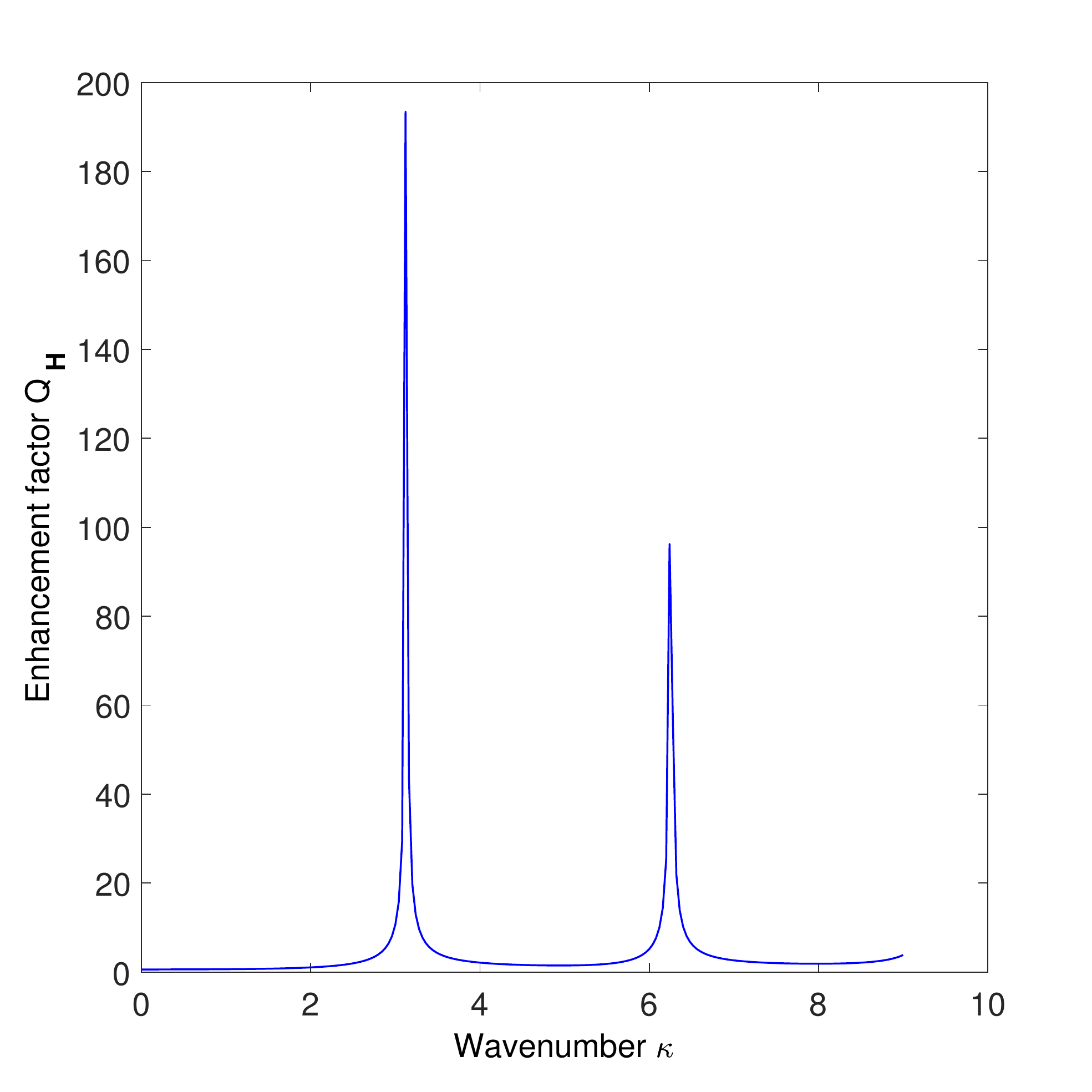}
\caption{The field enhancement factor is plotted against the
wavenumber $\kappa$ for the PEC-PMC cavity: (left) the electric
field enhancement factor $Q_{\boldsymbol E}$; (righ) the magnetic
field enhancement factor $Q_{\boldsymbol H}$.}
\label{PECPMC}
\end{figure}

\section{PEC-PEC cavity without resonance}\label{sec4}

In this section, assuming $\epsilon \ll d \ll \lambda$, we show that there is a
weak enhancement of the electric field and there is no enhancement of the
magnetic field in the PEC-PEC cavity.

\subsection{Problem formulation}

Consider the following model problem of the electromagnetic scattering by an
open cavity:
\begin{align}\label{PEC1}
\begin{cases}
\Delta u_{\epsilon}+\kappa^2 u_{\epsilon} = 0
\quad  &\text {in}~\Omega,\\
\partial_\nu u_{\epsilon}=0  \quad  &\text {on}~ \partial \Omega,\\
\lim\limits_{r \rightarrow \infty} \sqrt r \left( \partial_r u_{\epsilon}^{\rm
s} - {\rm i} \kappa  u_{\epsilon}^{\rm s} \right)=0 \quad &\text{in}~\mathbb
R_+^2,
 \end{cases}
\end{align}
where $u_\epsilon^{\rm sc} = u_{\epsilon} -(u^{\rm inc} +u^{\rm ref})$ in
$\mathbb R_2^+$. Here the incident field $u^{\rm inc} = e^{{\rm i}
\kappa\boldsymbol x\cdot \boldsymbol d}$ is a plane wave propagating in the
direction $\boldsymbol d =(\sin \theta, -\cos \theta)^{\top}$ and the reflected
field $u^{\rm r}= e^{{\rm i} \kappa \boldsymbol x \cdot\boldsymbol d'}$  with
the propagating direction $d'= (\cos \theta, \sin \theta)$.

Let $\phi_0 (x_1) =\frac{1}{\sqrt \epsilon}$ and $\phi_n (x_1) =\sqrt
{\frac{2}{\epsilon}} \cos \left( \frac{n \pi x_1}{\epsilon}\right) (n \geq 1)$
be an orthonormal basis on the interval $(0, \epsilon)$. Since the total field
satisfies the Neumann boundary condition on the bottom and lateral sides of the
cavity, we have from \eqref{PEC1} that $u_{\epsilon}$ can be expressed as the
sum of waveguide modes:
\begin{align}\label{wg1}
 u_{\epsilon} (\boldsymbol x) = \sum\limits_{n =0}^{\infty} \left(  \alpha_n^{+}
e^{ -{\rm i} \beta_n x_2} + \alpha_n^{-} e^{{\rm i} \beta_n (x_2 + d)}\right)
\phi_n (x_1),\quad \boldsymbol x\in D_\epsilon,
\end{align}
where the coefficients $\beta_n$ are defined as
\begin{align*}
 \beta_n =
 \begin{cases}
  \kappa, \quad & n=0,\\
  {\rm i} \sqrt{({n \pi}/{  \epsilon})^2 -\kappa^2}, \quad & n \geq 1.
 \end{cases}
\end{align*}
If $\epsilon<\lambda/2$, then $\beta_n$ are pure imaginary
numbers for all $n \geq 1$. For each $n$, it can be seen if $ n \pi
/{\epsilon} \leq \kappa $, the expansion consists of two propagating wave modes
traveling upward and downward respectively; if $ n \pi / \epsilon > \kappa$,
the expansion consists of two evanescent wave modes decaying exponentially away
from the bottom side and open aperture of the cavity, respectively.

Using the expansion \eqref{wg}, we may reformulate the model \eqref{PEC1} into
the following coupled problem:
\begin{align}\label{cpe1}
 \begin{cases}
\Delta u_{\epsilon} + \kappa^2 u_{\epsilon} =0
\quad  &\text {in}~\mathbb R_2^+,\\
u_{\epsilon} = \sum\limits_{n =0}^{\infty} \left(  \alpha_n^{+}
e^{ -{\rm i} \beta_n x_2} + \alpha_n^{-} e^{{\rm i} \beta_n (x_2 + d)}\right)
\phi_n (x_1) \quad & \text {in}~D_{\epsilon},\\
\partial_\nu u_{\epsilon}=0 \quad & \text {on}~ \partial \Omega \setminus
\Gamma^+_{\epsilon} ,\\
u_{\epsilon} (x_1, 0+) =u_{\epsilon} (x_1, 0-),~\partial_{x_2} u_{\epsilon}
(x_1, 0+)=\partial_{x_2}  u_{\epsilon} (x_1, 0-)
\quad & \text{on}~ \Gamma_\epsilon^+,\\
\lim\limits_{r \rightarrow \infty} \sqrt r \left( \partial_r u_{\epsilon}^{\rm
sc} - {\rm i} \kappa u_{\epsilon}^{\rm sc} \right)=0 \quad  &\text
{in}~\mathbb R_2^+.
\end{cases}
\end{align}
By matching the expansion series  \eqref{wg1} and using the continuity
conditions in \eqref{cpe1}, we may obtain
\begin{align*}
\alpha_n^+= \frac{e^{{\rm i} \beta_n d} u_{{\epsilon}, n}^{-} -u_{{\epsilon},
n}^+}{ e^{{\rm i} 2 \beta_n d}-1},\quad
\alpha_n^-= \frac{e^{{\rm i} \beta_n d} u_{{\epsilon}, n}^{+} -u_{{\epsilon},
n}^-}{ e^{{\rm i} 2 \beta_n d}-1}.
\end{align*}
Therefore, the wave field
\begin{align}\label{wf1}
u_{\epsilon} (\boldsymbol x) = \sum\limits_{n =0}^{\infty} \left(  \frac{e^{{\rm
i} \beta_n d} u_{{\epsilon}, n}^{-} -u_{{\epsilon}, n}^+}{ e^{{\rm i} 2 \beta_n
d}-1} e^{ -{\rm i} \beta_n x_2} + \frac{e^{{\rm i} \beta_n d} u_{{\epsilon},
n}^{+} -u_{{\epsilon}, n}^-}{ e^{{\rm i} 2 \beta_n d}-1} e^{{\rm i} \beta_n (x_2
+ d)}\right) \phi_n (x_1), \quad \boldsymbol x\in D_\epsilon.
\end{align}
A straightforward calculation yields
\begin{align}
&\partial_{x_2} u_{\epsilon} (x_1, 0-) =\sum\limits_{n=0}^{\infty} {\rm i }
\beta_n\left( \frac{e^{{\rm i} \beta_n d} u_{{\epsilon}, n}^{+} -u_{{\epsilon},
n}^-}{ e^{{\rm i} 2 \beta_n d}-1} e^{{\rm i} \beta_n  d}
-\frac{e^{{\rm i} \beta_n d} u_{{\epsilon}, n}^{-} -u_{{\epsilon}, n}^+}{
e^{{\rm i} 2 \beta_n d}-1}\right)\phi_n (x_1) \quad \text
{on}~\Gamma_{\epsilon}^+, \label{td1}\\
& \partial_{x_2} u_{\epsilon } (x_1, -d+) =\sum\limits_{n=0}^{\infty} {\rm i}
\beta_n\left( \frac{e^{{\rm i} \beta_n d} u_{{\epsilon}, n}^{+} -u_{{\epsilon},
n}^-}{ e^{{\rm i} 2 \beta_n d}-1}- \frac{e^{{\rm i} \beta_n d} u_{{\epsilon},
n}^{-} -u_{{\epsilon}, n}^+}{ e^{{\rm i} 2 \beta_n d}-1}
 e^{ {\rm i} \beta_n  d} \right)\phi_n (x_1)=0 \quad \text
{on}~\Gamma_{\epsilon}^{-}. \label{ttd1}
\end{align}
It follows from \eqref{ttd1} that
\begin{align}\label{uf1}
 u_{{\epsilon}, n}^{-} =  \frac{2  u_{{\epsilon}, n}^{+}}{e^{-{\rm i} \beta_n
d}+e^{{\rm i} \beta_n d}}.
\end{align}
Substituting \eqref{uf1} into \eqref{td1} yields
\begin{align*}
\partial_{x_2} u_{\epsilon} (x_1, 0-) =  \sum\limits_{n=0}^{\infty} {\rm i }
\beta_n  \frac{e^{{\rm i}  \beta_n d} -e^{-{\rm i}  \beta_n d}}{e^{{\rm i}
\beta_n d} +e^{-{\rm i}  \beta_n d}} u_{{\epsilon}, n}^+ \phi_n (x_1)
\quad \text {on}  ~\Gamma_{\epsilon}^+.
\end{align*}

Given a function $u\in H^{1/2} (\Gamma_\epsilon^+)$, we let $u^{+}= u  (x_1, 0)$
and define a DtN map on $\Gamma_{\epsilon}^+$:
\begin{align}\label{dtn1}
\mathscr B^{\rm PEC}[u] = \sum\limits_{n=0}^{\infty} {\rm i}\beta_n\frac{e^{{\rm
i} \beta_n d} -e^{-{\rm i}  \beta_n d}}{e^{{\rm i}  \beta_n d} +e^{-{\rm i}
\beta_n d}} u_{ n}^+ \phi_n (x_1),
\end{align}
where the Fourier coefficients $u_{n}^+ = \langle u^{+}, \phi_n
\rangle_{\Gamma_{\epsilon}^+}.$ Following a similar proof to Lemma \ref{tth},
we may show that the DtN map $\mathscr B^{\rm PEC}$ is also bounded. The proof
is omitted for brevity.

\begin{lemm}\label{tth1}
The DtN map $\mathscr B^{\rm PEC}$ defined by \eqref{dtn} is bounded from
$H^{1/2} (\Gamma_{\epsilon}^+) \rightarrow  H^{-1/2} (\Gamma_{\epsilon}^+)$,
i.e.,
\begin{align*}
\|\mathscr B^{\rm PEC} [u] \|_{H^{-1/2} (\Gamma_{\epsilon}^+)} \leq C
\|u\|_{H^{1/2} (\Gamma_{\epsilon}^+)} \quad \forall u \in {H^{1/2}
(\Gamma_{\epsilon}^+)},
 \end{align*}
where the constant $C =\max \left\{ \frac{\kappa \sin (\kappa d)}{\cos (\kappa
d)}, \frac{1-e^{-2 \sqrt {(n \pi / \epsilon)^2-\kappa^2}}}{1+e^{-2
\sqrt {(n \pi / \epsilon)^2-\kappa^2}}}\right\} <1$ is independent of
$\lambda, d.$
\end{lemm}

By \eqref{dtn1}, we derive the TBC:
\begin{align*}
\partial_{x_2} u_{\epsilon} = \mathscr B^{\rm PEC}[u_{\epsilon}] \quad \text
{on}~\Gamma_\epsilon^+,
\end{align*}
which helps to reduce \eqref{PEC1} into to a boundary value problem in
$\mathbb R_2^+$:
\begin{align}\label{rdp11}
\begin{cases}
 \Delta u_{\epsilon} + \kappa^2 u_{\epsilon}=0 \quad  & \text {in}~\mathbb
R_2^+,\\
 \partial_\nu u_{\epsilon} =0 \quad  &\text {on}~  \Gamma_0 \setminus
\Gamma_{\epsilon}^+,\\
 \partial_{x_2} u_{\epsilon} = \mathscr B^{\rm PEC} [u_{\epsilon}] \quad & \text
{on}~\Gamma_{\epsilon}^+,\\
 \lim\limits_{r \rightarrow \infty} \sqrt r \left( \partial_r u_{\epsilon}^{\rm
sc} - {\rm i} \kappa u_{\epsilon}^{\rm sc} \right)=0 \quad  &\text {in}~\mathbb
R_2^+.
 \end{cases}
\end{align}
The solution $u_\epsilon$ in the cavity $D_{\epsilon}$ may be obtained from
\eqref{wf1} once the Fourier coefficients $u_{{\epsilon}, n}^+$ are available
(note that $u_{{\epsilon}, n}^-$ can be computed from from \eqref{uf1}).

To find an approximate the model problem to \eqref{rdp11}, we similarly drop all
the nonzero wave modes and define
 \begin{align}\label{adtn1}
 \mathscr B_0^{\rm PEC} [v]= {\rm i } \kappa  \frac{e^{{\rm i} \kappa  d}
-e^{-{\rm i} \kappa  d}}{e^{{\rm i}  \kappa d}+e^{-{\rm i}  \kappa d}} v_0^+
\phi_0 (x_1).
 \end{align}
Hence we may consider the following problem:
\begin{align}\label{Eardp}
\begin{cases}
\Delta v_{\epsilon} + \kappa^2 v_{\epsilon}=0 \quad & \text {in}~\mathbb
R_2^+,\\
\partial_\nu v_{\epsilon} =0  \quad  &\text {on}~   \Gamma_0\setminus
\Gamma_{\epsilon}^+,\\
\partial_{x_2} v_{\epsilon}=\mathscr B_0^{\rm PEC} [v_{\epsilon}]\quad & \text
{on}~\Gamma_{\epsilon}^+,\\
\lim\limits_{r \rightarrow \infty} \sqrt r \left( \partial_r v_{\epsilon}^{\rm
sc} - {\rm i} \kappa  v_{\epsilon}^{\rm sc} \right)=0 \quad  &\text
{in}~\mathbb R_2^+,
  \end{cases}
\end{align}
where $v_{\epsilon}^{s} = v_{\epsilon} -(u^{\rm inc} +u^{\rm ref})$ in
$\mathbb R_2^+$. Inside the cavity, we may approximate
$u_{\epsilon}$ with one single mode:
\begin{align}\label{Eom1}
v_{\epsilon} (\boldsymbol x)= \left(  \tilde  \alpha_0^+ e^{-{\rm i} \kappa x_2}
+ \tilde  \alpha_0^{-} e^{{\rm i} \kappa (x_2 + d)}\right) \phi_0 (x_1),
\end{align}
where the coefficients $\tilde \alpha_0^+$ and  $\tilde \alpha_0^-$  are given by
\begin{align}\label{Eza1}
\tilde \alpha_0^+= \frac{e^{{\rm i} \kappa d} v_{{\epsilon}, 0}^{-} -
v_{{\epsilon}, 0}^{+}}{ e^{{\rm i}  2 \kappa d} -1},
 \quad \tilde \alpha_0^-= \frac{e^{{\rm i} \kappa d} v_{{\epsilon}, 0}^{+} -
v_{{\epsilon}, 0}^{-}}{ e^{{\rm i}  2 \kappa d} -1}.
\end{align}
It follows from  $ \partial_\nu v_{\epsilon} =0 $ on $\Gamma_{\epsilon}^{-} $
that $\tilde {\alpha}_0^{-} = \tilde {\alpha}_0^{+} e^{{\rm i} \kappa d}$, which
yields
\begin{align}\label{Eom11}
 v_{{\epsilon}, 0}^{-} = \frac{2 v_{\epsilon, 0}^+}{e^{{\rm i} \kappa d} +e^{-{\rm i} \kappa  d} }.
\end{align}

\subsection{Enhancement of the approximated field}

This section presents the estimates of the solution for the approximate model
problem \eqref{Eardp}.

\begin{theo}\label{EEE1}
Let $v_{\epsilon}$ be the solution of \eqref{Eardp} and be given by \eqref{Eom1}
inside the cavity, then there exist positive constants $C_1, C_2$ independent on
$\lambda, \epsilon, d$ such that
\begin{align*}
 C_1  \kappa^2  \sqrt  \epsilon d^{3/2} \leq \|\nabla v_{\epsilon}\|_{L^2
(D_{\epsilon})}  \leq  C_2 \kappa^2  \sqrt \epsilon d^{3/2}.
\end{align*}
\end{theo}

\begin{proof}
For the approximate model \eqref{Eardp}, we obtain from Green's formula that
\begin{align*}
v_{\epsilon} (\boldsymbol x)=u^{\rm inc}(\boldsymbol x) +u^{\rm r}(\boldsymbol
x) +\int_{\Gamma_\epsilon^+} G (\boldsymbol x, \boldsymbol y) \frac{\partial
v_{\epsilon}(\boldsymbol y)}{ \partial \nu_{\boldsymbol y}} {\rm d}
s_{\boldsymbol y}, \quad \boldsymbol x \in \mathbb R_2^+,
\end{align*}
where $G$ is the half-space Green function satisfying the Neumann boundary
condition and is given in \eqref{gfhs}. It follows from the continuity of the
single layer potential that
\begin{align}\label{Evs}
  v_{\epsilon} (\boldsymbol x)= u^{\rm inc}(\boldsymbol x)+ u^{\rm
ref}(\boldsymbol x) - \frac{\rm i}{2} \int_{\Gamma_{\epsilon}^+} H_0^{(1)} (k
|\boldsymbol x -\boldsymbol y|) \frac{\partial v_{\epsilon}(\boldsymbol y)}{
\partial \nu_{\boldsymbol y}} {\rm d} s_{\boldsymbol y}, \quad \boldsymbol x \in
\Gamma_{\epsilon}^+.
 \end{align}
In light of \eqref{Eom1}--\eqref{Eom11}, it yields that
\begin{align*}
 \partial_{x_2} v_{\epsilon}= {\rm i} \kappa \left(-\tilde {\alpha}_0^+ + \tilde
{\alpha}_0^{-} e^{{\rm i} \kappa d} \right) \phi_0 (x_1) \quad \text
{on}~\Gamma_{\epsilon}^+.
\end{align*}
Substituting the above equality into \eqref{Evs} and using the fact that
$\phi_0 (x_1) = \frac{1}{ \sqrt {\epsilon}}$, we get
\begin{align*}
 v_{\epsilon} (x_1, 0)=u^{\rm inc} (x_1, 0) +u^{\rm ref} (x_1, 0)
+\frac{\kappa}{2} \left( -\tilde {\alpha}_0^+ + \tilde {\alpha}_0^{-} e^{{\rm i}
\kappa  d }\right) \frac{1}{\sqrt {\epsilon}} h_1 (x_1),  \quad x_1 \in (0,
\epsilon),
\end{align*}
where $h_1 (x_1)$ is given in \eqref{H1}.
Therefore, the Fourier coefficients $v_{\epsilon, 0}^+$ may be expressed as
\begin{align*}
 v_{\epsilon, 0}^+ = \langle u^{\rm inc}, \phi_0  \rangle_{\Gamma_{\epsilon}^+}
 +\langle u^{\rm ref}, \phi_0  \rangle_{\Gamma_{\epsilon}^+} +\frac{\kappa }{2}
\left( -\tilde {\alpha}_0^+ + \tilde {\alpha}_0^{-} e^{{\rm i} \kappa d }\right)
\frac{1}{\sqrt {\epsilon}} \langle h_1, \phi_0 \rangle_{\Gamma_{\epsilon}^+}.
\end{align*}
It follows from the fact   $ u^{\rm inc} (x_1, 0)= u^{\rm ref} (x_1, 0)$ and
the inner product of \eqref{Eom1} with $\phi_0(x_1)$ that
\begin{align*}
 \tilde \alpha_{0}^+ +\tilde \alpha_{0}^{-} e^{{\rm i} \kappa  d} =2 \langle
u^{\rm inc}, \phi_0  \rangle_{\Gamma_{\epsilon}^+}
 + \frac{\kappa }{2  \sqrt {\epsilon}}  \langle h_1, \phi_0
\rangle_{\Gamma_{\epsilon}^+}  \left( -\tilde {\alpha}_0^+
 + \tilde {\alpha}_0^{-} e^{{\rm i} \kappa  d }\right).
\end{align*}
Using $ \tilde {\alpha}_0^{-} = \tilde {\alpha}_0^{+} e^{{\rm i} \kappa d}$ and
the above equation gives
\begin{align}\label{Ealp}
 \tilde \alpha_0^+ = \frac{2  \langle u^{\rm inc}, \phi_0
\rangle_{\Gamma_{\epsilon}^+} }{  (1+ c_0) +  (1- c_0) e^{{\rm i} 2 \kappa d}},
 \quad \tilde \alpha_0^{-}= \frac{2  e^{{\rm i } \kappa  d} \langle u^{\rm inc},
\phi_0  \rangle_{\Gamma_{\epsilon}^+} }{  (1+ c_0) +  (1- c_0) e^{{\rm i} 2
\kappa  d}},
\end{align}
where $c_0 =  \frac{\kappa }{2  \sqrt {\epsilon}}  \langle h_1, \phi_0
\rangle_{\Gamma_{\epsilon}^+}. $

Using the assumption $\epsilon\ll d \ll \lambda$ and expression of $c_0$ in \eqref{C0}, it gives
that
\begin{align}
&1/2 \leq |1+c_0| \leq 2, \quad 1/2 \leq  |1-c_0|\leq 2,\notag\\
&|(1+c_0)+(1-c_0)e^{{\rm i} 2 \kappa  d}|\leq  2 |1+e^{{\rm i} 2 \kappa d}| \leq 4, \nonumber\\
&|(1+c_0)+(1-c_0)e^{{\rm i} 2 \kappa  d}|\geq 2- |c_0 (1-e^{{\rm i} 2 \kappa d})| \geq 1. \label{Ei11}
\end{align}
Recalling the definition of the incident wave, we get
\begin{align}\label{Einq}
|\langle u^{\rm inc}, \phi_0
\rangle_{\Gamma_{\epsilon}^+}|=\left|\frac{1}{\sqrt \epsilon }\int_0^\epsilon
e^{{\rm i} \kappa  \sin \theta x_1} {\rm d} x_1\right|=\sqrt{\epsilon} +O
(\epsilon ^{3/2}).
\end{align}
 Substituting \eqref{Ealp} into \eqref{Eom1} yields
\begin{align}\label{Enep}
 v_{\epsilon} (\boldsymbol x)= \left(
 \frac{2  \langle u^{\rm inc}, \phi_0  \rangle_{\Gamma_{\epsilon}^+} }{  (1+
c_0) + (1- c_0) e^{{\rm i} 2 \kappa d}} e^{-{\rm i} \kappa  x_2}
 +\frac{2  e^{{\rm i } \kappa d} \langle u^{\rm inc}, \phi_0
\rangle_{\Gamma_{\epsilon}^+} }{  (1+ c_0) +  (1- c_0) e^{{\rm i} 2 \kappa d}}
e^{{\rm i} \kappa  (x_2 + d)}\right)\phi_0 (x_1).
\end{align}
A simple calculation yields
\begin{align*}
 \frac{\partial v_{\epsilon}}{\partial {x_2}}
 &= {\rm i} \kappa e^{{\rm i} \kappa  d } \frac{2  \langle u^{\rm inc}, \phi_0
\rangle_{\Gamma_{\epsilon}^+} }{  (1+ c_0) +(1- c_0) e^{{\rm i} 2 \kappa  d}}
\phi_0 (x_1) \left( -e^{ -{\rm i} \kappa  (x_2 +d)} + e^{ {\rm i} \kappa  (x_2
+d)}\right) \\
 &= - \kappa  e^{{\rm i} \kappa  d } \frac{2  \langle u^{\rm inc},
\phi_0  \rangle_{\Gamma_{\epsilon}^+} }{  (1+ c_0) + (1- c_0) e^{{\rm i} 2
\kappa  d}} 2 \sin \kappa  (x_2 +d)\phi_0 (x_1),
\end{align*}
which gives
\begin{align*}
 \|\nabla v_{\epsilon}\|^2_{L^2 (D_{\epsilon})}
   =\kappa ^2\left|\frac{2  \langle u^{\rm inc}, \phi_0
\rangle_{\Gamma_{\epsilon}^+} }{  (1+ c_0) + (1- c_0) e^{{\rm i} 2 \kappa
d}}\right|^2   \int_{-d}^{0} 4 \sin ^2 \kappa  (x_2 +d) {\rm d} x_2.
\end{align*}
Noting that $d \ll \lambda$, we also have \eqref{I1}--\eqref{I2} holds.
It follows from \eqref{Ei11}--\eqref{Einq} and  the estimates in  \eqref{I1}--\eqref{I2} that there
exist constants $C_1$ and $C_2$ independent on $\epsilon, \lambda, d
$ such that
\begin{align*}
C_1  \kappa^2   \sqrt {\epsilon }  d^{3/2} \leq \|\nabla v_{\epsilon}\|_{L^2
(D_{\epsilon})}  \leq  C_2 \kappa^2 \sqrt {\epsilon} d^{3/2},
\end{align*}
which completes the proof.
\end{proof}

\begin{theo}\label{TET2}
Let $v_{\epsilon}$ be the solution of \eqref{Eardp} and be given by \eqref{Eom1}
inside the cavity, then there exist positive constants $C_3, C_4$ independent on
$\lambda, \epsilon, d$ such that
\begin{align*}
C_3  \sqrt {\epsilon  d}  \leq \|v_\epsilon\|_{L^2 (D_{\epsilon})} \leq
C_4 \sqrt { \epsilon d}.
\end{align*}
\end{theo}

\begin{proof}
By \eqref{Enep}, we have
\begin{align*}
\|v_{\epsilon}\|^2_{L^2 (D_{\epsilon})} &=\int_{0}^{\epsilon}\int_{- d}^{0}
|v_{\epsilon}|^2{\rm d} x_2 {\rm d} x_1\\
&=\left|\frac{2  \langle u^{\rm inc}, \phi_0
\rangle_{\Gamma_{\epsilon}^+} e^{{\rm i} \kappa d} }{  (1+ c_0) + (1- c_0)
e^{{\rm i} 2 \kappa d}}\right|^2 \int_{-d}^{0}| e^{-{\rm i} \kappa (x_2
+d)}+e^{{\rm i} \kappa (x_2 +d)}|^2 {\rm d} x_2 \\
& =\frac {\left| 2 \langle u^{\rm inc},
\phi_{0}\rangle_{\Gamma_{\epsilon}^+}  \right|^2}
{\left| (1+c_0)+ (1 -c_0)^{{\rm i} 2 \kappa  d}\right|^2} \int_{-d}^{0}  4
\cos^2  \kappa (x_2 +d) {\rm d} x_2.
\end{align*}
It follows from $d \ll \lambda$ that
\begin{align*}
\int_{-d}^{0}  4 \cos^2  \kappa (x_2 +d) {\rm d} x_2= 2 (d + O((d/\lambda)^2)).
\end{align*}
In light of  \eqref{Ei11}--\eqref{Einq}, there exist constants $C_3, C_4$ such
that
\begin{align*}
C_3^2 \epsilon  d \leq \|v_\epsilon\|^2_{L^2 (D_{\epsilon})} \leq C_4^2
\epsilon d,
\end{align*}
which completes the proof.
\end{proof}

\subsection{Field enhancement without resonance}

Following from the same steps as those in subsection \ref{sub2.4}, we may also
show that
\begin{align}\label{Eaccm}
\|\nabla u_{\epsilon} -\nabla v_{\epsilon}\|_{L^2 (D_{\epsilon})} \leq
C(R) \epsilon, \quad \|u_{\epsilon} -v_{\epsilon}\|_{L^2
(D_{\epsilon})} \leq C(R) \epsilon \sqrt{\epsilon |\ln \epsilon|},
\end{align}
where $C(R)$ is a positive constant depending on $R$ but is independent of
$\lambda, \epsilon, d$. The details are omitted for brevity.

\begin{rema}
In \eqref{Eaccm}, the constant $C(R)$ comes from two places: one is the estimate
of the DtN operator $\mathscr B^{\rm PEC}$ in Lemma \ref{tth1}, where the
constant $C$ is independent of $\lambda, d$, and the other is the Sobolev
embedding results in $\mathbb R_+^2$, which depends only on $R$. In Lemma
\ref{accm}, the constant $C(\lambda, d, R)$ depends on $\lambda, d, R$ because
of the estimate of the DtN operator $\mathscr B^{\rm PMC}$ in Lemma \ref{tth},
where the constant $C(\lambda, d)$ depends on $\lambda, d$.
\end{rema}

The following estimates are the main results of the electromagnetic field
enhancement for the PEC-PEC cavity when $\epsilon\ll d\ll\lambda$.

\begin{theo}\label{NTH1}
If $\epsilon $ is small enough such that $2 C(R)  \sqrt {\epsilon d}
<\min \left\{C_1 (\kappa d)^2, ~C_2(\kappa d)^2\right\}$, then the electric
field has no enhancement. If  $d \ll \lambda$ such that $\frac{d}{\lambda}
\leq \min \left\{\frac{C(R)}{ 8 \pi^2 C_1 }\sqrt {\epsilon d},
~\frac{C(R)}{ 8 \pi^2 C_1 }\sqrt {\epsilon d} \right\}$, then
the electric field enhancement factor satisfies
\begin{align*}
\frac{ C(R)\lambda}{4 \pi} \sqrt{\frac{\epsilon}{ d}}  \leq
Q_{\boldsymbol E} \leq \frac{3 C(R)\lambda}{4 \pi}
\sqrt{\frac{\epsilon}{ d}}
\end{align*}
and the magnetic field enhancement factor satisfies
\[\frac{1}{2} C_3
\leq Q_{\boldsymbol H} \leq  \frac{3}{2} C_4,
\]
where the constant $C(R)$ is given in \eqref{Eaccm}, the
constants $C_1, C_2$ are given in Theorem \ref{EEE1}, and the constants $C_3,
C_4$ are given in Theorem \ref{TET2}.
\end{theo}

\begin{proof}
By Theorem \ref{EEE1} and \eqref{Eaccm}, we can obtain
\begin{align*}
\|\nabla v_{\epsilon}\|_{L^2 (D_{\epsilon})} -\|\nabla u_{\epsilon} -\nabla
v_{\epsilon}\|_{L^2 (D_{\epsilon})}  \leq \|\nabla u_{\epsilon} \|_{L^2
(D_{\epsilon})}\leq \|\nabla v_{\epsilon}\|_{L^2 (D_{\epsilon})} +\|\nabla
u_{\epsilon} -\nabla v_{\epsilon}\|_{L^2
(D_{\epsilon})}.
\end{align*}
We consider two cases:

(i) If $\epsilon $ is small enough such that $2 C(R)  \sqrt
{\epsilon d} < \min \left\{C_1 (\kappa d)^2, C_2(\kappa d)^2\right\}$, then we
have
\begin{align*}
\|\nabla u_{\epsilon} \|_{L^2 (D_{\epsilon})}
 \leq C_2 (\kappa d)^2  \sqrt{\frac {\epsilon }{d}}+  C(R)
\epsilon \leq \frac{3}{2} C_2 (\kappa d)^2  \sqrt{\frac {\epsilon }{d}}
\end{align*}
and
\begin{align*}
\|\nabla u_{\epsilon} \|_{L^2 (D_{\epsilon})} \geq C_1 (\kappa d)^2
\sqrt{\frac {\epsilon }{d}} -  C(R) \epsilon\geq   \frac{1}{2} C_1
(\kappa d)^2  \sqrt{\frac {\epsilon }{d}}.
\end{align*}
Since $\|\nabla u^{\rm inc}\|^2_{L^2 (D_\epsilon)}=\kappa^2 \epsilon d$, we
obtain
\[
C_1 \pi \frac{d}{\lambda} \leq  \frac{\|\nabla u_{\epsilon} \|_{L^2
(D_{\epsilon})}}{\|\nabla u^{\rm inc}\|_{L^2 (D_\epsilon)}} \leq 3 \pi
C_2\frac{d}{\lambda}.
\]

(ii) If $d \ll \lambda$ such that $\frac{d}{\lambda}
\leq \min \left\{\frac{C(R)}{ 8 \pi^2 C_1 }\sqrt {\epsilon d},
~\frac{C(R)}{ 8 \pi^2 C_1 }\sqrt {\epsilon d} \right\}$, i.e.,
\[
C(R) \geq \max\{2C_1 \kappa^2 d^{3/2}\sqrt{\epsilon}, ~2C_2\kappa^2
d^{3/2}\sqrt{\epsilon}\},
\]
then we get
\begin{align*}
 \frac{ C(R)}{2} \epsilon  \leq \|\nabla u_{\epsilon} \|_{L^2
(D_{\epsilon})} \leq \frac{3 C(R)}{2} \epsilon,
\end{align*}
which gives
\begin{align*}
\frac{ C(R)\lambda}{4 \pi} \sqrt{\frac{\epsilon}{ d}}
\leq \frac{\|\nabla u_{\epsilon} \|_{L^2 (D_{\epsilon})}}{\|\nabla u^{\rm
inc}\|_{L^2 (D_\epsilon)}} \leq \frac{3 C (R)\lambda}{4 \pi}
\sqrt{\frac{\epsilon}{ d}}.
\end{align*}

Noting that $|\nabla \times \boldsymbol H_{\epsilon}|=|\nabla u_{\epsilon}|$
and using Ampere's law $\nabla \times \boldsymbol H_{\epsilon}=-{\rm i}\omega
\varepsilon \boldsymbol E_{\epsilon}$, we obtain
\begin{align*}
Q_{\boldsymbol E}=\frac{\|\boldsymbol E\|_{L^2 (D_\epsilon)}}{\|\boldsymbol
E^{\rm inc}\|_{L^2 (D_\epsilon)}}=\frac{\|\nabla u_{\epsilon} \|_{L^2
(D_{\epsilon})}}{\|\nabla u^{\rm inc}\|_{L^2 (D_\epsilon)}},
\end{align*}
which shows that for case (i) the electric field $\boldsymbol E$  has no
enhancement; for case (ii) the electric field enhancement factor has
an order $O(\lambda\sqrt{\epsilon/d}).$

For the magnetic field, by Theorem \ref{TET2} and \eqref{Eaccm}, if
$C (R)\epsilon \sqrt {|\ln \epsilon|/d} \leq \min \{ C_3/2,~
C_4/2\}$ then we have
\begin{align*}
 \frac{ C_3}{2} \sqrt {\epsilon d} \leq \|u_{\epsilon}\|_{L^2
(D_\epsilon)} \leq \frac{3 C_4}{2} \sqrt {\epsilon d}.
\end{align*}
Since $\|u^{\rm inc}\|^2_{L^2 (D_{\epsilon})}= \epsilon d$, we obtain
\begin{align*}
\frac{1}{2}C_3 \leq Q_{\boldsymbol H}= \frac{\|\boldsymbol H\|_{L^2
(D_\epsilon)}}{\|\boldsymbol H^{\rm inc}\|_{L^2 (D_\epsilon)}}=\frac{\|
u_{\epsilon} \|_{L^2 (D_{\epsilon})}}{\| u^{\rm inc}\|_{L^2 (D_\epsilon)}} \leq
\frac{3}{2} C_4,
\end{align*}
which completes the proof.
\end{proof}

\subsection{Numerical experiments}

We present a numerical example to illustrate and verify the results stated in
Theorem \ref{NTH1}. The width of the cavity is $\epsilon=0.005$, the depth of
the cavity is $d=1$, and the incident angle is $\theta=\pi/3$. Figure
\ref{PECdunear0} plots the electric field enhancement factor $Q_{\boldsymbol E}$
against the wavenumber $\kappa\in (0, 10^{-8})$. It shows that when $\kappa$ is
sufficient small, i.e., when the wavelength $\lambda$ is sufficiently large, the
electric field does have an enhancement, which corresponds to case (ii) in
Theorem \ref{NTH1}; when the wavenumber $\kappa$ increases (the wavelength
$\lambda$ decreases) but is still small, the electric field has no enhancement,
which corresponds to case (i) in Theorem \ref{NTH1}. The plot is not shown for
the magnetic field since it has no enhancement in this region.

\begin{figure}
\centering
\includegraphics[width=0.45\textwidth]{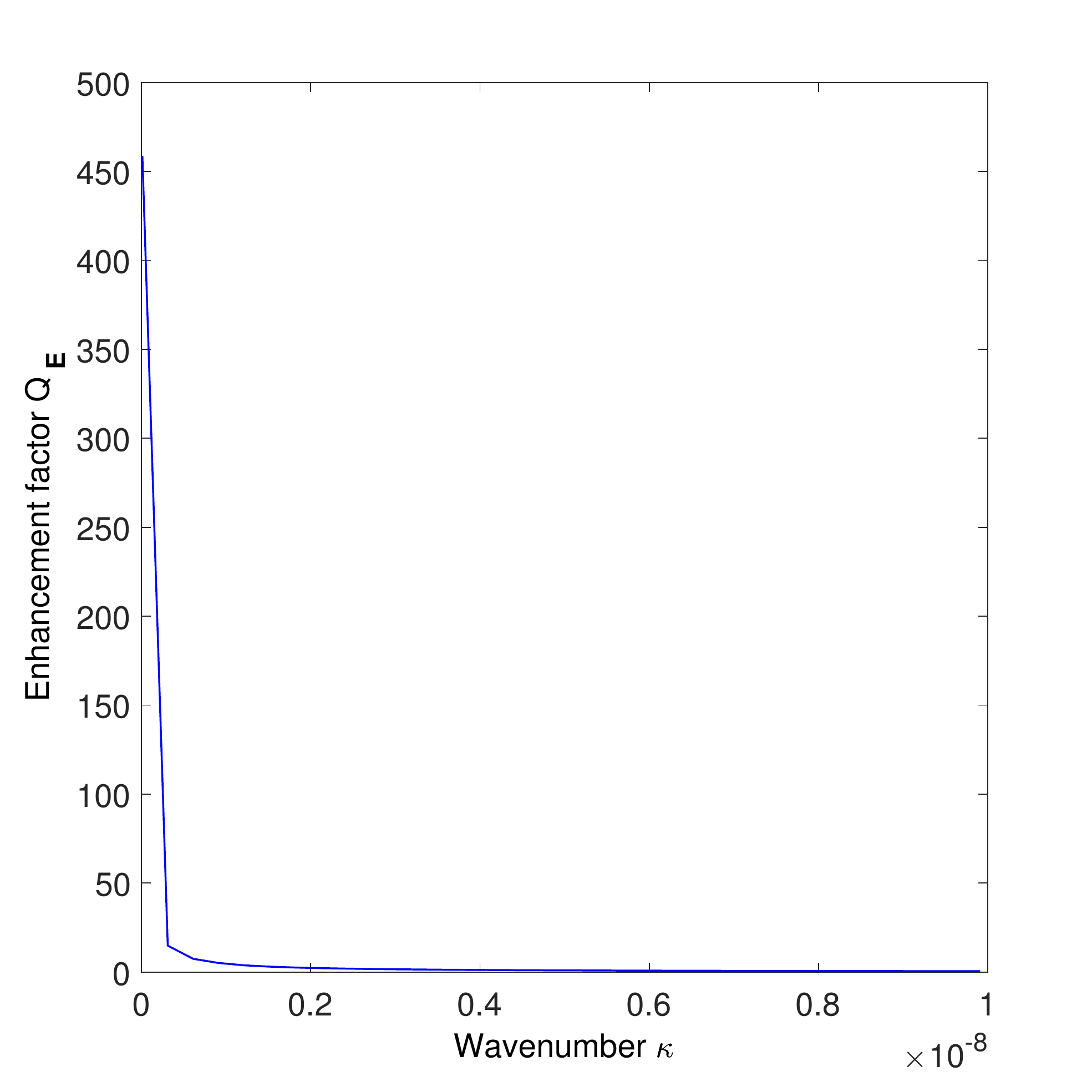}
\caption{The amplification  of electric field in the PEC cavity with
$\epsilon=0.005$, $d=1$ and $\theta =\pi/3$ when $\kappa$ near zero.}
\label{PECdunear0}
\end{figure}

\section{PEC-PEC cavity with resonance} \label{sec5}

In this section, we show the field enhancement at resonant frequencies for the
PEC-PEC cavity. Consider the model problem:
\begin{align}\label{PM1}
\begin{cases}
\Delta u_{\epsilon} +\kappa^2 u_{\epsilon} = 0
\quad  &\text {in}~\Omega,\\
\partial_\nu u_{\epsilon}=0  \quad  &\text {on}~ \partial \Omega,\\
\lim\limits_{r \rightarrow \infty} \sqrt r \left( \partial_r u_{\epsilon}^{\rm
sc} - {\rm i} \kappa  u_{\epsilon}^{\rm sc} \right)=0 \quad &\text{in}~\mathbb
R_2^+.
 \end{cases}
\end{align}
Here we assume that $\epsilon\ll \lambda$.

\subsection{Boundary integral equations}

Let $g^{\rm PEC}_{\epsilon}$ be the Green function for the Helmholtz equation
with Neumann boundary condition in $D_{\epsilon}$, i.e., it satisfies
\begin{align*}
\begin{cases}
\Delta g^{\rm PEC}_{\epsilon}(\boldsymbol x, \boldsymbol y)+\kappa^2
g^{\rm PEC}_{\epsilon} (\boldsymbol x, \boldsymbol y)=-\delta (\boldsymbol x,
\boldsymbol y), \quad & \boldsymbol x, \boldsymbol y \in D_{\epsilon},\\
\frac{\partial g^{\rm PEC}_{\epsilon} (\boldsymbol x, \boldsymbol y)}{\partial
\nu_{\boldsymbol y}}=0 \quad & \text {on}~\partial D_{\epsilon}.
\end{cases}
\end{align*}
It can be verified that
\begin{align*}
g^{\rm PEC}_{\epsilon} (\boldsymbol x, \boldsymbol y)=\sum\limits_{m, n
=0}^{\infty} c_{m, n} \phi_{m, n} (\boldsymbol x) \phi_{m, n}
(\boldsymbol y),
\end{align*}
where
\begin{align}\label{Ecad}
& c_{m, n} = \frac{1}{ \kappa^2 - (m \pi / {\epsilon})^2 - (n \pi / d)^2},
\nonumber\\
&\phi_{m, n} (\boldsymbol x) =\sqrt {\frac{\alpha_{m, n}}{{\epsilon} d}}
\cos \left( \frac{m \pi x_1}{ \epsilon} \right) \cos \left(  \frac{n \pi}{d}
(x_2+d)\right),\nonumber\\
 & \alpha_{m, n} =
 \begin{cases}
 1,\quad & m=0, n =0,\\
 2, \quad & m =0,  n \geq 1 ~\text {or}~n=0, m\geq 1\\
 4, \quad & m \geq 1, n \geq 1.
 \end{cases}
 \end{align}

Using Green's formula, we obtain
 \begin{align*}
 &u_{\epsilon} (\boldsymbol x)=\int_{\Gamma_{\epsilon}^+} G (\boldsymbol x,
\boldsymbol y) \frac{\partial u_{\epsilon} (\boldsymbol y)}{ \partial
\nu_{\boldsymbol y}} {\rm d} s_{\boldsymbol y}+u^{\rm inc} (\boldsymbol
x)+u^{\rm ref} (\boldsymbol x),  \quad  \boldsymbol x \in \mathbb R_2^+,\\
 &u_{\epsilon} (\boldsymbol x)=-\int_{\Gamma_{\epsilon}^+}  g^{\rm
PEC}_{\epsilon} (\boldsymbol x, \boldsymbol y) \frac{\partial u_{\epsilon}
(\boldsymbol y)}{\partial \nu_{\boldsymbol y}} {\rm d} s_{\boldsymbol y}, \quad
\boldsymbol x \in D_{\epsilon}.
 \end{align*}
It follows from the continuity of the single layer potential that we have
\begin{align*}
u_{\epsilon} (\boldsymbol x)=\int_{\Gamma_{\epsilon}^+} \left(-\frac{\rm i}{2}
\right)H_0^1 (\kappa |\boldsymbol x -\boldsymbol y|)
\frac{\partial u_{\epsilon} (\boldsymbol y)}{\partial \nu_{\boldsymbol y}} {\rm
d} s_{\boldsymbol y}+u^{\rm inc} (\boldsymbol x)+u^{\rm ref} (\boldsymbol x),
\quad \boldsymbol x \in \Gamma_{\epsilon}^+
\end{align*}
and
\begin{align*}
u_{\epsilon} (\boldsymbol x)=-\int_{\Gamma_{\epsilon}^+} g^{\rm PEC}_{\epsilon}
(\boldsymbol x, \boldsymbol y) \frac{\partial u_{\epsilon} (\boldsymbol
y)}{\partial \nu_{\boldsymbol y}} {\rm d} s_{\boldsymbol y}, \quad  \boldsymbol
x \in \Gamma_{\epsilon}^+.
\end{align*}
By imposing the continuity of the solution on $\Gamma_\epsilon^+$, we have the
following lemma.

\begin{lemm}
The scattering problem \eqref{PM1} is equivalent to the boundary
integral equation:
\begin{align}\label{BIT}
\int_{\Gamma_{\epsilon}^+} \left(  \left(- \frac{{\rm i}}{2} \right) H_{0}^1
(\kappa  |\boldsymbol x- \boldsymbol y|)
+g^{\rm PEC}_{\epsilon} (\boldsymbol x, \boldsymbol y)\right)\frac{\partial
u_{\epsilon} (\boldsymbol y)}{ \partial \nu_{\boldsymbol y}}   {\rm d}
s_{\boldsymbol y} +u^{\rm inc}(\boldsymbol x) +u^{\rm ref} (\boldsymbol x)=0,
\quad \boldsymbol x\in\Gamma_{\epsilon}^+.
\end{align}
\end{lemm}

Similarly, we introduce the rescaling variables $X=x_1/\epsilon, Y=y_1/\epsilon$
and define the following boundary integral operators:
\begin{align}\label{Ebio}
(T^e \varphi) (X):=\int_0^2 G_{\epsilon}^e (X, Y) \varphi (Y) {\rm d} Y, \quad
X\in (0, 1), \nonumber\\
(T^{\rm i} \varphi) (X):= \int_0^1  G_{\epsilon}^i(X, Y) \varphi
(Y) {\rm d} Y, \quad X \in (0, 1).
\end{align}
where $G_{\epsilon}^e, G_{\epsilon}^i, \varphi$ are give in
\eqref{ND1} with $g^{\rm PMC}_{\epsilon} $ being replaced by $g^{\rm
PEC}_{\epsilon}$. Then the boundary integral equation \eqref{BIT} is equivalent
to the operator equation:
\begin{align}\label{FOE}
(T^e + T^i)\varphi=f/\epsilon,
\end{align}
where $f=(u^{\rm inc} +u^{\rm ref})\big|_{\Gamma_{\epsilon}^+}= 2 e^{{\rm i}
\kappa  \sin\theta \epsilon X}.$

\subsection{Asymptotics of the integral operators}

In this subsection, we study the asymptotic properties of the integral
operators in \eqref{Ebio}.

\begin{lemm}
If $\epsilon\ll\lambda$, then we have the following asymptotic formulas:
\begin{align*}
&G_{\epsilon}^{e}=\Gamma_1 (\kappa, {\epsilon})+\frac{1}{\pi} \ln
|X-Y|+R_1^{\epsilon} (X, Y)+R_2 ^{\epsilon} (X,Y),\\
&G_{\epsilon}^i =\Gamma_2 (\kappa, {\epsilon})+\frac{1}{\pi }
\left(  \ln \left|\sin \frac{\pi (X+Y)}{2}\right|+\ln \left|\sin \frac{\pi
(X-Y)}{2}\right| \right) +R_3 ^{\epsilon} (X, Y),
\end{align*}
where $\Gamma_2 (\kappa, {\epsilon}) =\frac{\cot \kappa d}{\epsilon
\kappa}+\frac{2 \ln 2}{\pi},$ and $\Gamma_1 (\kappa, {\epsilon}), R_1^{\epsilon}
(X, Y), R_2 ^{\epsilon} (X,Y), R_3 ^{\epsilon} (X, Y) $ are given in Lemma
\ref{AE}.
\end{lemm}

\begin{proof}
The proof is similar to that for Lemma \ref{AE}. The integral kernels
$G_{\epsilon}^e (X, Y)$ and $G_{\epsilon}^i (X, Y)$ can be expressed as
\begin{align*}
&G_{\epsilon}^e (X, Y)=\Gamma_1 (\kappa, {\epsilon})+\frac{1}{\pi} \ln
|X-Y|+R_1^{\epsilon} (X, Y)+R_2 ^{\epsilon} (X,Y),\\
&G_{\epsilon}^{i} (X, Y)=\frac{1}{\epsilon d} \sum\limits_{m,
n=0}^{\infty} c_{m, n} \alpha_{m,n} \cos (m \pi X) \cos (m \pi Y).
\end{align*}
Recalling the definitions of $c_{m, n}$ and $\alpha_{m, n}$ in
\eqref{Ecad}, we let
\begin{align*}
 C_{m} (\epsilon, \kappa)=\sum\limits_{n=0}^{\infty} c_{m, n}
\alpha_{m,n}=\sum\limits_{n=0}^{\infty} \frac{\alpha_{m, n}}{\kappa^2 -
(m \pi / {\epsilon})^2 - (n \pi / d)^2}.
\end{align*}
Using the fact (cf. \cite{Gradshteyn2015}) that
\begin{align*}
\cot (\pi x)=\frac{1}{\pi x}+\frac{2 x}{\pi} \sum\limits_{n=1}^{\infty}
\frac{1}{ x^2 -n^2}, \quad  x\neq n,
\end{align*}
for $m=0$, we get
\begin{align*}
C_0(\kappa)=\frac{1}{\kappa^2}+\sum\limits_{n=1}^{\infty}\frac{2}{\kappa^2
-(n \pi / d)^2}=\frac{d}{\kappa} \left( \frac{1}{\pi (d \kappa/\pi) } +\frac{2 d
\kappa /\pi}{\pi} \sum\limits_{n=1}^{\infty} \frac{1}{ (d \kappa/\pi)^2 -n^2}
\right)= \frac{d \cot (d \kappa)}{\kappa}.
\end{align*}
Following from the identity
\begin{align*}
\coth (\pi x)=\frac{1}{ \pi x} +\frac{2 x}{\pi} \sum\limits_{n=1}^{\infty}
\frac{1}{x^2 +n^2},
\end{align*}
we have for $m \geq 1$ that
\begin{align*}
C_{m} (\epsilon, \kappa) &=\frac{2}{\kappa^2 - (m\pi/\epsilon)^2}+\sum\limits_{n
\geq 1}^{\infty} \frac{4}{\kappa^2 -(m\pi/\epsilon)^2 - (n \pi/d)^2}\\
&=\frac{-2d}{\sqrt{(m \pi/\epsilon)^2 -\kappa^2}} \left(  \frac{2
\frac{d}{\pi}\sqrt{(m \pi/\epsilon)^2 -\kappa^2} }{\pi}
\sum\limits_{n\geq 1}^{\infty} \frac{1}{n^2 +(\frac{d}{\pi})^2
((m\pi/\epsilon)^2 -\kappa^2)} +\frac{1}{d \sqrt{(m \pi/\epsilon)^2 -\kappa^2}
} \right)\\
&=\frac{-2d}{\sqrt{(m \pi/\epsilon)^2 -\kappa^2}} \coth (d \sqrt{(m
\pi/\epsilon)^2 -\kappa^2})\\
&=-\frac{2 d \epsilon}{m \pi}-\frac{\epsilon^3 \kappa^2 d}{m^2 \pi^3}+O
\left(\frac{\epsilon^5}{m^5}\right).
\end{align*}
Here we use the asymptotic expression of $(1-x^2)^{-1/2}$ and $\coth (d
\sqrt{(m \pi/\epsilon)^2 -\kappa^2}) \rightarrow 1$ as $\epsilon \rightarrow 0.$

Substituting $C_0(\kappa), C_m( \epsilon, \kappa) (m \geq 1)$ into
to $G_\epsilon^i$, we obtain
\begin{align*}
G_{\epsilon}^{i} (X, Y) &=\frac{1}{\epsilon d} \left( \frac{d \cot
(\kappa d)}{\kappa}-\sum\limits_{m=1}^{\infty} \left(\frac{2 d \epsilon }{m\pi}
+ \frac{\epsilon^3 \kappa^2 d}{m^3 \pi^3}\right) \cos (m \pi X) \cos (m \pi Y)
+O \left(  \sum\limits_{m=1}^{\infty}\frac{\epsilon^5}{m^5}\right)\right)\\
&=  \frac{d \cot (\kappa d)}{\kappa \epsilon}+\frac{1}{\pi } \left( 2 \ln 2 +\ln
\left|\sin \frac{\pi (X+Y)}{2}\right|+\ln \left|\sin \frac{\pi (X-Y)}{2}\right|
\right)\\
&\hspace{1cm}- \frac{{\epsilon}^2 \kappa^2}{\pi}\bigg(\frac{1}{\pi ^2}
\sum\limits_{m=1}^{\infty} \frac{1}{m^3}+ \frac{(X+Y)^2}{4} \ln (\pi (X+Y))
+\frac {(X-Y)^2}{4} \ln (\pi (X-Y))\\
&\hspace{1cm}+ O ((X+Y)^2 +(X-Y)^2)\bigg)\\ &= \Gamma_2 (\kappa,
\epsilon)+ \frac{1}{\pi } \left(\ln \left|\sin \frac{\pi (X+Y)}{2}\right|+\ln
\left|\sin \frac{\pi (X-Y)}{2}\right| \right)
 +R_3^{\epsilon} (X, Y),
\end{align*}
which completes the proof.
\end{proof}

Let $\Gamma =\Gamma_1(\kappa, \epsilon)+  \Gamma_2 (\kappa,
\epsilon)$, $K, K_{\infty}^1, K_{\infty}^2$ are defined same as in \eqref{Gd1}.
Then for the operator $T^e+T^i$ in \eqref{FOE}, we have a similar
decomposition
\begin{align*}
T^e + T^i = \Gamma  P +K +K_{\infty}^1 +K_{\infty}^2,
\end{align*}
where $K,  K_{\infty}^1, K_{\infty}^2$ have the same properties as those in
Lemma \ref{ope}.

\subsection{Asymptotics of the resonances}

For convenience, we write
\begin{align*}
\Gamma P +K +K_{\infty}^1 +K_{\infty}^2 := \mathscr P +\mathscr L,
\end{align*}
where $\mathscr P =\Gamma P,  \mathscr L = K +K_{\infty}^1
+K_{\infty}^2.$  The eigenvalues of operator $ {\mathscr L}^{-1} {\mathscr
P} +\mathscr I$ are
\begin{align*}
\lambda (\kappa, \epsilon)=1+ \Gamma (\kappa ,\epsilon)\langle \mathscr
L^{-1} 1_{(0,1)}, 1_{(0,1)}\rangle_{L^2 (0, 1)}.
\end{align*}
Therefore, the characteristic values of the operator $\mathscr P
+\mathscr L$ are the roots of the analytic functions $\lambda (\kappa,
{\epsilon})$, and the associated characteristic function is given by
 \begin{align*}
\varphi_0 = \Gamma (\kappa ,{\epsilon}) \mathscr L^{-1} 1_{(0,1)}.
 \end{align*}

\begin{theo}
The resonance of the scattering problem \eqref{PM1}  are the roots of the
analytic function $ \lambda (\kappa, \epsilon)=0.$
Moreover, the resonance set $\{k_n\}, n=0, 1, 2, \cdots$ satisfies the following asymptotic
expansion
\begin{align}\label{knse}
k_{n}=\frac{(n+\frac{1}{2})
\pi}{d}+\frac{(n+\frac{1}{2})\pi}{d^2}\left(\frac{1}{\pi} {\epsilon}\ln
{\epsilon} +\left(\frac{1}{q_0} +\frac{1}{\pi} (2\ln 2 +\ln
\frac{(n+\frac{1}{2})\pi}{d}+\gamma_1)\right){\epsilon}\right)+O({\epsilon}^2\ln
{\epsilon}).
 \end{align}
\end{theo}

\begin{proof}
Given roots of $\lambda (\kappa, {\epsilon})$, it is easy to check that they
are the characteristic values of the operator $\mathscr P +\mathscr L$ with
corresponding characteristic function defined above.

Consider the roots of
\begin{align*}
\lambda (\kappa, \epsilon)= 1+(\Gamma_1(\kappa, {\epsilon})+
\Gamma_2 (\kappa, {\epsilon}))\langle \mathscr L^{-1}1_{(0,1)},
1_{(0,1)}\rangle_{L^2 (0,1)}=0.
\end{align*}
 Recall that
$\Gamma_1 (\kappa, {\epsilon})=\frac{1}{\pi} (\ln \kappa + \gamma_1) +\frac{1}{
\pi} \ln {\epsilon}$ and  $\Gamma_2 (\kappa, {\epsilon})=\frac{\cot
\kappa d}{ {\epsilon} \kappa} +\frac{2 \ln 2}{ \pi}.$
The above equation can be written as
\begin{align*}
1+\left( \frac{\cot  \kappa d}{ {\epsilon} \kappa} +\frac{1}{\pi} (2\ln 2+\ln
\kappa +\gamma_1)+ \frac{1}{\pi}\ln {\epsilon}\right)\langle \mathscr
L^{-1}1_{(0,1)}, 1_{(0,1)}\rangle_{L^2 (0,1)}=0.
 \end{align*}
By Lemma \ref{rt}, similar as \eqref{pk}, we let
\begin{align*}
p(\kappa, {\epsilon}):= {\epsilon} \lambda (\kappa,
{\epsilon})={\epsilon} +\left(\frac{\cot \kappa d}{\kappa}+ {\epsilon}\rho
(\kappa)+\frac{1}{\pi}{\epsilon}\ln {\epsilon} \right) \left( q_0 +O
({\epsilon}^2 \ln {\epsilon})  \right),
 \end{align*}
 where $\rho (\kappa):=\frac{1}{\pi} (2\ln 2+\ln \kappa +\gamma_1).$ All the
roots of $p (\kappa, \epsilon)$ lie in the domain
$ \Omega_{\delta_0, \theta_0, M}:= \{z \in \mathbb C: |z-2 j\pi/d| \geq
\delta_0, j \in \mathbb Z\} \cap\Omega_{\theta_0, M},$
where
 \[
 \Omega_{\theta_0, M}=\left\{z \in \mathbb C: |z| \leq M, - (\frac{\pi}{d}
-\theta_0) \leq {\rm agr} z \leq \frac{\pi}{d} -\theta_0 \right \}.
\]

It is clear to note that $\frac{\cot \kappa d}{\kappa}$ is analytic in the
domain $\Omega_{\delta_0, \theta_0, M}$ and the roots are given by
 \[
 k_{n, 0} =\left(n+\frac{1}{2}\right) \frac{\pi}{d}, \quad n=0, 1, 2, \cdots.
 \]
Denote by $k_n$ the roots of $\lambda (\epsilon, \kappa)$. Applying  Rouche's
theorem, we deduce that $k_n$ are simple and  close to $ k_{n,
0}$. The corresponding $ p_1 (\kappa, \epsilon)$ are defined by
 \begin{align*}
p_1 (\kappa, {\epsilon})= {\epsilon} +\left( \frac{\cot \kappa d}{\kappa
}+ {\epsilon}\rho (k)+\frac{1}{\pi}{\epsilon}\ln {\epsilon} \right) q_0.
\end{align*}
 Expanding $p_1(\kappa, {\epsilon})$ at $k_{n,0}$ yields
  \begin{align*}
p_1(\kappa, {\epsilon})
=&{\epsilon} +\bigg(\frac{-d^2}{(n+\frac{1}{2})\pi}
(\kappa-k_{n0})+\frac{\epsilon}{\pi} \left(2\ln 2
+\ln\frac{(n+\frac{1}{2})\pi}{d}+\gamma_1 \right)\\
&\hspace{1cm}+ \frac{{\epsilon} d}{(n+\frac{1}{2})\pi^2}(\kappa-k_{n0})
  +O (\kappa- k_{n0})^2+\frac{1}{\pi}{\epsilon}\ln {\epsilon}\bigg)q_0.
  \end{align*}
  We conclude that $p_1(\kappa, {\epsilon})$ has simple roots in
$\Omega_{\delta_0, \theta_0, M}$ which are close to $k_{n0}$.
Moreover, these roots are analytic with respect to the variable ${\epsilon}$ and
  ${\epsilon} \ln {\epsilon}.$ Denote the roots of $p_1 (\kappa,
\epsilon)$ by $k_{n1}$ and expand them in term of ${\epsilon}$ and
${\epsilon} \ln {\epsilon},$ we obtain
\begin{align*}
k_{n1}=&k_{n0}+\frac{k_{n0}}{d}\left(\frac{1}{\pi} {\epsilon}\ln
{\epsilon} +\left(\frac{1}{q_0} +\frac{1}{\pi} (2\ln 2 +\ln
\frac{(n+\frac{1}{2})\pi}{d}+\gamma_1)\right){\epsilon}\right)
  +O({\epsilon}^2\ln {\epsilon}).
  \end{align*}
 Using Rouche's theorem again, we obtain
 \begin{align*}
k_n = k_{n1}+O (\epsilon^2 \ln \epsilon),
 \end{align*}
which completes the proof.
\end{proof}

\subsection {The field enhancement with resonance}

In the far field region, we can follow the same steps as those in Subsection
\ref{sub3.4} with the new $p (\kappa, \epsilon)$. Then we have the following
theorem with the detailed proof omitted.

\begin{theo}
At the resonant frequencies $\kappa={\rm Re}k_n$ where $k_n$ is given in
\eqref{knse}, the scattered electric and magnetic fields have enhancement of an
order $O (1/\epsilon)$  in the far field region.
\end{theo}

Next we demonstrate the field enhancement inside the cavity. It
follows from \eqref{wg1} that the total field $u_{\epsilon }$
can be expanded as the sum of waveguide modes:
\begin{align}\label{Ewg}
 u_{\epsilon} (\boldsymbol x) =  \frac{1}{\sqrt \epsilon} \left( \alpha_0^+
e^{-{\rm i} \kappa x_2} +\alpha_0^- e^{{\rm i} \kappa (x_2+d)}\right)+
\sum\limits_{n= 1}^{\infty}  \sqrt {\frac{2}{\epsilon}}\left(  \alpha_n^{+} e^{
-{\rm i} \beta_n x_2} + \alpha_n^{-} e^{{\rm i} \beta_n (x_2 + d)}\right) \cos
\left(\frac{m \pi x_1}{\epsilon} \right).
\end{align}

\begin{lemm}\label{Lem5.5}
The coefficients in \eqref{Ewg} have the following properties:
\begin{align*}
\frac{\alpha_0^+}{\sqrt{\epsilon}} = \frac{ q_0\left( 2 + \sin \theta \cdot
O(\kappa \epsilon)  \right)}{{\rm i} \kappa (-1+e^{{\rm i} 2 \kappa d})p }+O
({\epsilon}^2 \ln {\epsilon}),\quad
\frac{\alpha_0^+}{\sqrt{\epsilon}}=\frac{ q_0\left( 2 +\sin \theta \cdot
O(\kappa \epsilon) \right)}{-2 \kappa \sin (\kappa d) p }+O ({\epsilon}^2 \ln
{\epsilon})
\end{align*}
and
\begin{align*}
\left|\sqrt{\frac{2}{\epsilon}} \alpha_n^+ \right|\leq \frac{C}{ \sqrt n}, \quad
 \left|\sqrt{\frac{2}{\epsilon}} \alpha_n^- \right|\leq \frac{C}{ \sqrt n},
\quad n\geq 1,
\end{align*}
where the positive constant $C$ is independent of $\lambda, \epsilon $ and $ n$.
\end{lemm}

\begin{proof}
It follows from the boundary condition $\partial_\nu u_{\epsilon}=0$ on
$\Gamma_{\epsilon}^{-}$ that the expansion of \eqref{Ewg}
becomes
\begin{align*}
u_{\epsilon} (\boldsymbol x)=&\frac{\alpha_0^+ e^{{\rm i} \kappa
d}}{\sqrt{\epsilon}}\left(e^{-{\rm i} \kappa (x_2 +d)} +e^{{\rm i} \kappa
(x_2 +d)} \right)\\
&+\sum\limits_{n= 1}^{\infty} \sqrt{\frac{2}{\epsilon}}
\alpha_{n}^+  e^{{\rm i} \beta_n d} \left(e^{-{\rm i} \beta_n (x_2 +d)} +e^{{\rm
i} \beta_n (x_2 +d)} \right)\cos \left( \frac{n \pi x_1}{ \epsilon } \right),
\quad \boldsymbol x\in D_\epsilon.
\end{align*}
Taking the derivative of the above equality with respect to $x_2$ on
$\Gamma_{\epsilon}^+$ yields
\begin{align*}
\frac{\partial u_{\epsilon}(x_1, 0)}{\partial x_2} =
\frac{\alpha_0^+}{\sqrt{\epsilon}} {\rm i} \kappa (-1+e^{{\rm i} 2 \kappa
d})+\sum\limits_{n=1}^{\infty} \sqrt{\frac{2}{\epsilon}} \alpha_n^+ {\rm i}
\beta_n (-1+e^{{\rm i} 2 \beta_n d}) \cos \left( \frac{n \pi x_1}{ \epsilon }
\right).
\end{align*}
Multiplying the above equality by $\phi_0(x_1)$ and integrating it on
$\Gamma_{\epsilon}^+$ yields
\begin{align*}
\frac{\alpha_0^+}{\sqrt{\epsilon}} {\rm i} \kappa (-1+e^{{\rm i} 2 \kappa d})
&=\frac{1}{\epsilon}\int_{\Gamma_{\epsilon}^+} \frac{\partial
u_{\epsilon}}{\partial x_2} (x_1, 0) {\rm d} x_1=\frac{1}{\epsilon}
\int_{\Gamma_{\epsilon}^+} -\varphi (x_1/\epsilon) {\rm d} x_1\\
&=\int_0^1 \varphi (X) {\rm d} X=\frac{ q_0\left( 2 +\sin \theta \cdot O(\kappa
\epsilon)  \right)}{\tilde p}+O ({\epsilon}^2 \ln {\epsilon}).
\end{align*}
Here we use $ \langle
\varphi, 1_{(0,1)} \rangle_{L^2 (0,1)}=\frac{ q_0\left( 2 +\sin \theta \cdot
O(\kappa \epsilon) \right)}{\tilde p}+O ({\epsilon}^2 \ln {\epsilon})$, which
is shown in the proof of Lemma \ref{lm3.7}. By  $\alpha_0^-= \alpha_0^+ e^{{\rm
i} \kappa d}$, we get
\begin{align*}
&\frac{\alpha_0^+}{\sqrt{\epsilon}} = \frac{ q_0\left( 2 +\sin \theta \cdot
O(\kappa \epsilon) \right)}{{\rm i} \kappa (-1+e^{{\rm i} 2 \kappa d}) \tilde p
}+O ({\epsilon}^2 \ln {\epsilon}),\\
&\frac{\alpha_0^-}{\sqrt{\epsilon}}=\frac{ q_0\left( 2 +\sin \theta \cdot
O(\kappa \epsilon)  \right)}{-2 \kappa \sin (\kappa d) \tilde p }+O
({\epsilon}^2 \ln {\epsilon}).
\end{align*}

Similarly, for $n \geq 1,$ we may obtain
\begin{align*}
\sqrt{\frac{2}{\epsilon}} \alpha_n^+ {\rm i} \beta_n (-1+e^{{\rm i} 2 \beta_n
d}) =-\int_0^1 \varphi (X) \cos (n \pi X) {\rm d} X.
\end{align*}
Note that ${\rm i} \beta_n =O (- \frac{n}{\epsilon})$, $\|\varphi\|_{H^{-1/2}(0,
1)} \lesssim \frac{1}{\epsilon}$, $\|\cos (n \pi X)\|_{H^{1/2} (0,
1)}\lesssim \sqrt n$. Thus we have
\begin{align*}
\left|\sqrt{\frac{2}{\epsilon}} \alpha_n^+ \right|\leq \frac{C}{ \sqrt n}, \quad
 \left |\sqrt{\frac{2}{\epsilon}} \alpha_n^- \right|\leq \frac{C}{ \sqrt n},
\quad n\geq 1,
\end{align*}
where the positive constant $C$ is independent of $\lambda, \epsilon,  n$.
\end{proof}

The following result gives the field enhancement inside the
cavity.

\begin{theo}\label{PECt}
Let $\mathring{D}_{\epsilon}:=\{  \boldsymbol x \in
D_{\epsilon}:-d \leq  x_2 \leq -\epsilon\}$ be the interior of the cavity
$D_\epsilon$. If $\kappa = {\rm Re}k_n$ where $k_n$ is given in \eqref{knse},
then we have for $\epsilon \ll \lambda$ that
\begin{align*}
u_{\epsilon} (\boldsymbol x)= \left( \frac{2}{\epsilon}+\sin \theta \cdot O(1) +
O(\ln ^2{\epsilon}) \right) \frac{2 {\rm i} \cos \kappa (x_2 +d)}{\kappa  \sin
(\kappa d)} +O ({\epsilon}^2 \ln {\epsilon}),\quad x\in \mathring{D}_\epsilon.
\end{align*}
Moreover, the electric and magnetic field enhancements are of an order $O
(1/\epsilon).$
\end{theo}

\begin{proof}
In the region $\mathring{D}_{\epsilon},$ we have from  the expansion of
\eqref{Ewg} and Lemma \ref{Lem5.5} that
\begin{align*}
u_{\epsilon} (\boldsymbol x)
&=\frac{q_0 (2+\sin \theta \cdot O(\kappa \epsilon)}{  {\rm i} \kappa
(-e^{-{\rm i} \kappa  d} +e^{{\rm i} \kappa  d}) p} \left( e^{{\rm -i} \kappa
(x_2+d)} +e^{{\rm i} \kappa  (x_2+d)} \right) +O ({\epsilon}^2 \ln {\epsilon})+O
(e^{- 1/{\epsilon}})\\
&=\frac{q_0  (2+ \sin \theta \cdot O(\kappa \epsilon))\cos  \kappa (x_2+d)}
{-\kappa \sin \kappa d} \frac{1}{p}+O ({\epsilon}^2 \ln {\epsilon})+O
(e^{- 1/{\epsilon}}).
\end{align*}
Using the fact that at the resonant frequencies $\kappa ={\rm Re} k_n,$
\begin{align*}
\frac{1}{ p}=\frac{2 {\rm i}} {q_0 {\epsilon}} (1+O ({\epsilon} \ln^2
{\epsilon})).
\end{align*}
we obtain
\begin{align*}
u_{\epsilon} (\boldsymbol x)= \left( \frac{2}{\epsilon}+ \sin \theta \cdot O(1)
+ O(\ln ^2{\epsilon}) \right) \frac{2 {\rm i} \cos \kappa (x_2 +d)}{\kappa \sin
(\kappa d)} +O ({\epsilon}^2 \ln {\epsilon}).
\end{align*}
Therefore, the magnetic field enhancement has an order of $O (1/{\epsilon})$.
We conclude from Ampere's law that the electric field enhancement also has
an order of $O(1/{\epsilon})$.
\end{proof}

Finally, on the open aperture $\Gamma_\epsilon^+$, we may follow similar steps
as those in subsection \ref{sub3.4.3} and show the electromagnetic field
enhancement. The proof is omitted again for brevity.

\begin{theo}
At the resonant frequencies $\kappa  ={\rm Re}k_n$ where $k_n$ is given in
\eqref{knse}, the enhancement of the scattered electric and magnetic fields has
an order $O(1/\epsilon)$ on the open aperture $\Gamma_{\epsilon}^+$.
\end{theo}

\subsection{Numerical experiments}

We show some numerical experiments to verify the theoretical findings
of the field enhancement for the PEC-PEC cavity. We take the cavity width
$\epsilon=0.005$, the cavity depth $d=1$, and the angle of incidence
$\theta=\pi/3$. Figure \ref{PEC} (left) and (right) show the plot of the
electric field enhancement $Q_{\boldsymbol E}$ the the magnetic field
enhancement factor $Q_{\boldsymbol H}$ against the wavenumber $\kappa$,
respectively. We observe that both the electric and magnetic enhancement factors
do obtain peaks at the resonant frequencies \eqref{knse}, which are close to
$\kappa=(n+\frac{1}{2})\pi, n=0, 1, 2, \dots$.

\begin{figure}
\centering
\includegraphics[width=0.45\textwidth]{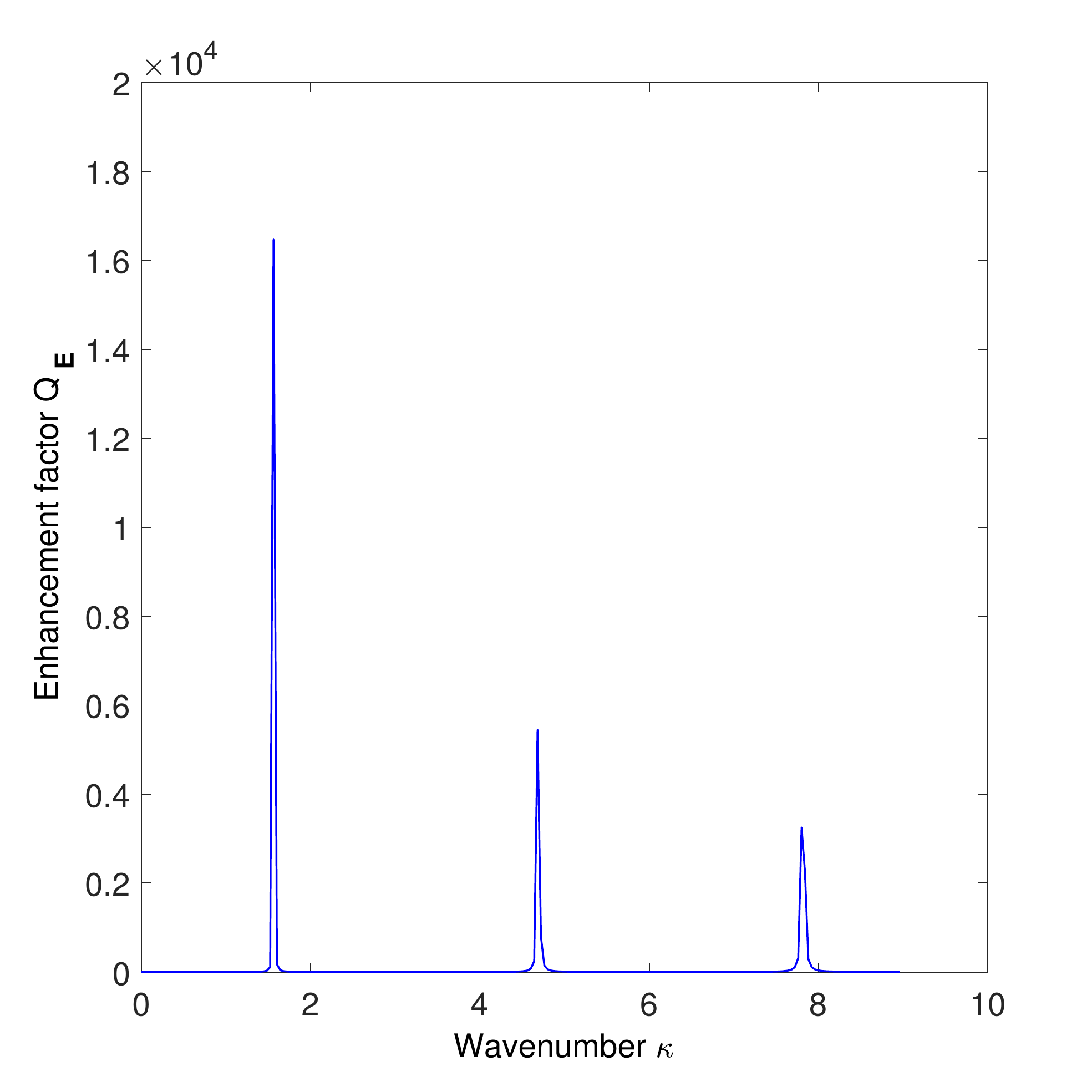}
\includegraphics[width=0.45\textwidth]{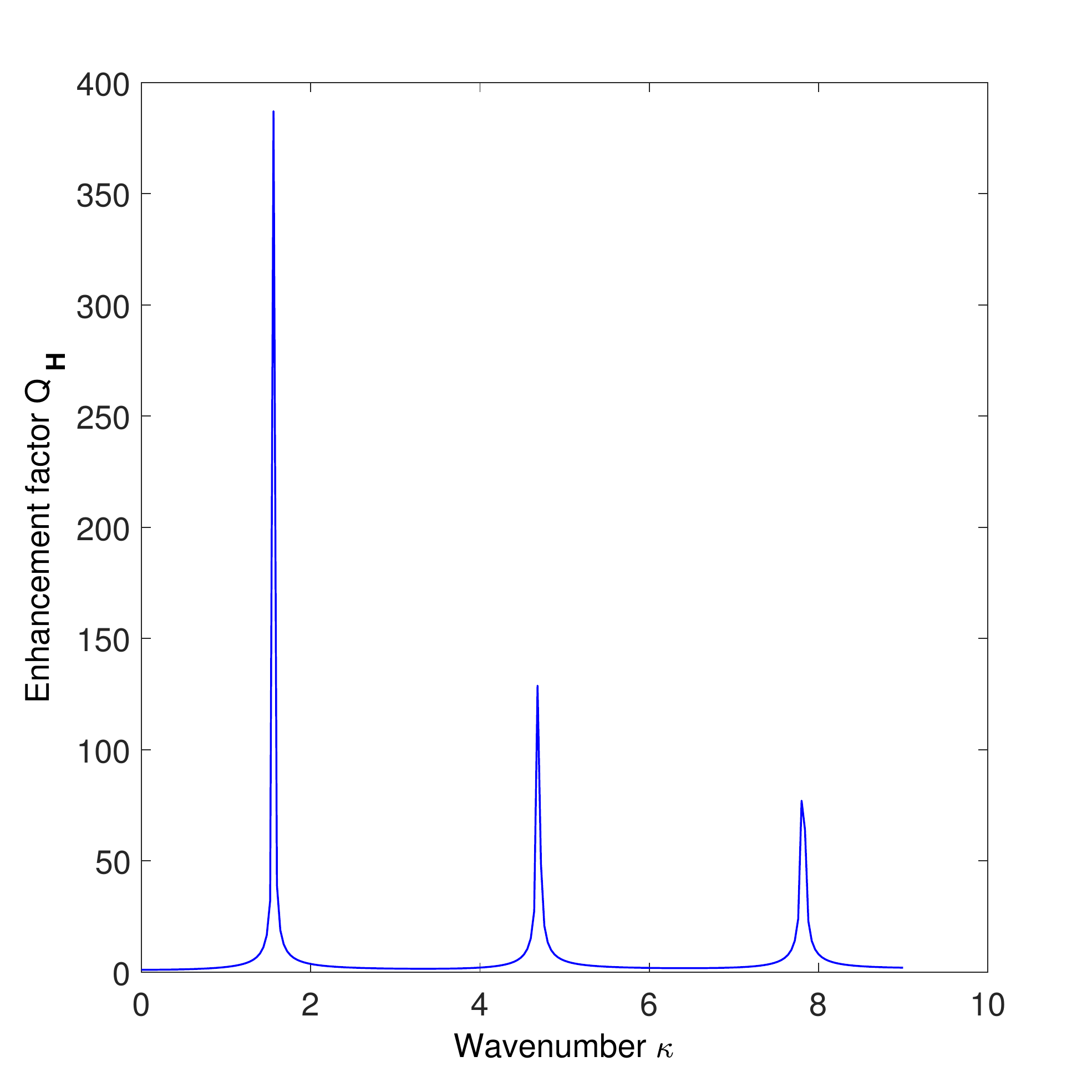}
\caption{The field enhancement factor is plotted against the
wavenumber $\kappa$ for the PEC-PEC cavity: (left) the electric
field enhancement factor $Q_{\boldsymbol E}$; (righ) the magnetic
field enhancement factor $Q_{\boldsymbol H}$.}
\label{PEC}
\end{figure}

\section{Concluding remarks}\label{sec6}

We have studied the electromagnetic field enhancement for the scattering of a
plane wave by a subwavelength rectangular open cavity. The TM polarized wave
model problem is considered for both the PEC-PMC and the PEC-PEC cavities. The
electric and magnetic field enhancements are examined in both the nonresonant
and resonant regimes. These results provide an understanding on some of the
mechanisms for the wave field amplification and localization in a subwavelength
structure. This work focuses on the TM polarization of the electromagnetic wave
in a single cavity. We will investigate the influence of multiple cavities,
whether the enhancement may occur in the transverse electric polarization (TE),
the more complicated three-dimensional Maxwell equations, and the elastic wave
equation. The results will be reported elsewhere.


\begin{thebibliography}{10}

\bibitem{Abramowitz1972}
M.~Abramowitz and I.~A. Stegun, editors, Handbook of Mathematical Functions with
Formulas, Graphs, and Mathematical Tables, Dover Publications, Inc., New York,
1992.

\bibitem{Adams2003}
R.~A. Adams and J.~J.~F. Fournier, Sobolev Spaces, vol. 140, Pure and
Applied Mathematics, Elsevier/Academic Press, Amsterdam, second edition, 2003.

\bibitem{Ammari2002}
H.~Ammari, G.~Bao, and A.~W. Wood, Analysis of the electromagnetic scattering
from a cavity, Japan J. Indust. Appl. Math., 19 (2002), 301--310.

\bibitem{ammari2016minnaert}
H.~Ammari, B.~Fitzpatrick, D.~Gontier, H.~Lee, and H.~Zhang, Minnaert resonances
for acoustic waves in bubbly media, arXiv:1603.03982, 2016.

\bibitem{AmmariZhang2015}
H.~Ammari and H.~Zhang, A mathematical theory of super-resolution by using a
system of sub-wavelength Helmholtz resonators, Comm. Math. Phys.,
337 (2015), 379--428.

\bibitem{astilean2000light}
S.~Astilean, P.~Lalanne, and M.~Palamaru, Light transmission through metallic
channels much smaller than the wavelength, Optics Communications,
175 (2000), 265--273.

\bibitem{Bonnetier20102}
J.-F. Babadjian, E.~Bonnetier, and F.~Triki, Enhancement of electromagnetic
fields caused by interacting subwavelength cavities, Multiscale Model. Simul.,
8 (2010), 1383--1418.

\bibitem{BaoYun2016}
G.~Bao and K.~Yun, Stability for the electromagnetic scattering from
large cavities, Arch. Ration. Mech. Anal., 220 (2016), 1003--1044.

\bibitem{BaoYunZhou2012}
G.~Bao, K.~Yun, and Z.~Zhou, Stability of the scattering from a large
electromagnetic cavity in two dimensions, SIAM J. Math. Anal., 44
(2012), 383--404.

\bibitem{barbara2003electromagnetic}
A.~Barbara, P.~Qu{\'e}merais, E.~Bustarret, T.~L{\'o}pez-Rios, and
T.~Fournier, Electromagnetic resonances of sub-wavelength rectangular
metallic gratings, The European Physical Journal D-Atomic, Molecular, Optical
and Plasma Physics, 23 (2003), 143--154.

\bibitem{Bonnetier2010}
E.~Bonnetier and F.~Triki, Asymptotic of the {G}reen function for the
diffraction by a perfectly conducting plane perturbed by a sub-wavelength
rectangular cavity, Math. Methods Appl. Sci., 33 (2010), 772--798.

\bibitem{chen2013atomic}
X.~Chen, H.-R. Park, M.~Pelton, X.~Piao, N.~C. Lindquist, H.~Im, Y.~J. Kim,
J.~S. Ahn, K.~J. Ahn, N.~Park, et~al., Atomic layer lithography of wafer-scale
nanogap arrays for extreme confinement of electromagnetic waves, Nature
communications, 4 (2013), 2361.

\bibitem{Clausel2006}
M.~Clausel, M.~Durufl\'e, P.~Joly, and S.~Tordeux, A mathematical analysis of
the resonance of the finite thin slots, Appl. Numer. Math.,
56 (2006), 1432--1449.

\bibitem{Collin1960}
R.~E. Collin, Field Theory of Guided Waves, International Series in Pure and
Applied Physics. Mc-Graw-Hill Book Co., Inc., New York-Toronto-London, 1960.

\bibitem{Colton1998}
D.~Colton and R.~Kress, Inverse Acoustic and Electromagnetic Scattering
Theory, vol. 93, Applied Mathematical Sciences, Springer-Verlag, Berlin, second
edition, 1998.

\bibitem{Colton1983}
D.~Colton and R.~Kress, Integral Equation Methods in Scattering Theory,
vol. 72, Classics in Applied Mathematics, SIAM, Philadelphia, PA, 2013.

\bibitem{ebbesen1998extraordinary}
T.~W. Ebbesen, H.~J. Lezec, H.~Ghaemi, T.~Thio, and P.~Wolff, Extraordinary
optical transmission through sub-wavelength hole  arrays, Nature, 391
(1998), 667--669.

\bibitem{garcia2010light}
F.~J. Garcia-Vidal, L.~Martin-Moreno, T.~Ebbesen, and L.~Kuipers, Light passing
through subwavelength apertures,  Rev. Modern Phys., 82 (2010), 729.

\bibitem{Gradshteyn2015}
I.~S. Gradshteyn and I.~M. Ryzhik, Table of Integrals, Series, and
Products, Elsevier/Academic Press, Amsterdam, eighth edition, 2015.

\bibitem{Jerison1985}
D.~Jerison and C.~E. Kenig, Unique continuation and absence of positive
eigenvalues for Schr\"odinger operators, Ann. of Math., 121 (1985), 463--494.

\bibitem{Patrick2006}
P.~Joly and S.~Tordeux, Asymptotic analysis of an approximate model for time
harmonic waves in media with thin slots, ESAIM: Math. Model. Numer. Anal.,
40 (2006), 63--97.

\bibitem{Joly2006}
P.~Joly and S.~Tordeux, Matching of asymptotic expansions for wave propagation
in media with thin slots. I. The asymptotic expansion, Multiscale Model.
Simul., 5 (2006), 304--336.

\bibitem{Kress1999}
R.~Kress, Linear Integral Equations, vol. 82, Applied Mathematical Sciences,
Springer-Verlag, New York, second edition, 1999.

\bibitem{Kriegsmann2004}
G.~A. Kriegsmann, Complete transmission through a two-dimensional diffraction
grating, SIAM J. Appl. Math., 65 (2004), 24--42.

\bibitem{lewin1975theory}
L.~Lewin, Theory of Waveguides: Techniques for the Solution of Waveguide
Problems, New York, Halsted Press, 1975.

\bibitem{li2016survey}
P.~Li, A survey of open cavity scattering problems, J. Comp. Math., 36
(2018), 1--16.

\bibitem{liedberg1983surface}
B.~Liedberg, C.~Nylander, and I.~Lunstr{\"o}m, Surface plasmon resonance for gas
detection and biosensing, Sensors and actuators, 4 (1983), 299--304.

\bibitem{lin2015}
J.~Lin and F.~Reitich, Electromagnetic field enhancement in small gaps: a
rigorous mathematical theory, SIAM J. Appl. Math., 75 (2015), 2290--2310.

\bibitem{linzhang2017}
J.~Lin and H.~Zhang, Scattering and field enhancement of a perfect conducting
narrow slit, SIAM J. Appl. Math., 77 (2017), 951--976.

\bibitem{sarrazin2007bounded}
M.~Sarrazin and J.-P. Vigneron, Bounded modes to the rescue of optical
transmission, Europhysics News, 38 (2007), 27--31.

\bibitem{seo2009terahertz}
M.~Seo, H.~Park, S.~Koo, D.~Park, J.~Kang, O.~Suwal, S.~Choi, P.~Planken,
G.~Park, N.~Park, et~al., Terahertz field enhancement by a metallic nano slit
operating beyond the skin-depth limit, Nature Photonics, 3 (2009), 152--156.

\bibitem{sturman2010transmission}
B.~Sturman, E.~Podivilov, and M.~Gorkunov, Transmission and diffraction
properties of a narrow slit in a perfect metal, Physical Review B,
82 (2010), 115419.

\bibitem{takakura2001optical}
Y.~Takakura, Optical resonance in a narrow slit in a thick metallic screen,
Phys. Rev. Lett., 86 (2001), 5601.

\bibitem{Watson1995}
G.~N. Watson, A Treatise on the Theory of Bessel Functions, Cambridge
Mathematical Library, Cambridge University Press, Cambridge, 1995.

\bibitem{yang2002resonant}
F.~Yang and J.~R. Sambles, Resonant transmission of microwaves through a narrow
metallic slit, Phys. Rev. Lett., 89 (2002), 063901.

\end{thebibliography}
\end{document}